\def\highlightchangessince{99}
\def\otherhighlightchangessince{99}
\def\otherotherhighlightchangessince{99}

\documentclass[12pt]{article}
\usepackage{bm,amsmath,amsthm,amssymb,graphicx,paralist,setspace}
\usepackage{color}
\usepackage[margin=1in]{geometry}
\usepackage{natbib}
\usepackage[colorlinks]{hyperref}
\usepackage{multicol}
\usepackage{multirow}
\usepackage{enumitem}
\usepackage{setspace}
\usepackage{booktabs}
\usepackage{floatrow}

\newcommand{\reals}{\mathbb{R}}

\DeclareMathOperator{\logit}{logit}
\DeclareMathOperator{\inversegamma}{InverseGamma}
\DeclareMathOperator{\inversegaussian}{InverseGaussian}

\DeclareMathOperator{\betad}{Beta}
\DeclareMathOperator{\betaprime}{BetaPrime}
\DeclareMathOperator{\negbin}{NegBin}
\DeclareMathOperator{\betanegbin}{BetaNegBin}
\DeclareMathOperator{\gammad}{Gamma}
\DeclareMathOperator{\bin}{Bin}

\DeclareMathOperator{\expd}{Exp}
\DeclareMathOperator{\var}{Var}
\DeclareMathOperator{\cov}{Cov}
\DeclareMathOperator{\corr}{Corr}
\DeclareMathOperator{\tv}{TV}

\DeclareMathOperator{\tr}{tr}
\DeclareMathOperator{\diag}{Diag}
\DeclareMathOperator{\iid}{iid\,}
\DeclareMathOperator{\ind}{ind.\,}
\DeclareMathOperator{\pa}{pa}
\DeclareMathOperator{\fa}{fa}

\DeclareMathOperator{\eb}{EB}

\theoremstyle{plain}
\newtheorem{thm}{Theorem}
\newtheorem{lem}[thm]{Lemma}
\newtheorem{cor}[thm]{Corollary}
\theoremstyle{definition}
\newtheorem{assum}[thm]{Assumption}
\newtheorem{ex}[thm]{Example}
\newtheorem*{rem}{Remark}

\numberwithin{thm}{section}
\numberwithin{equation}{section}

\let\originalepsilon\epsilon
\renewcommand{\epsilon}{\varepsilon}
\newcommand{\given}{\;\middle\vert\;}

\let\originalleft\left
\let\originalright\right
\renewcommand{\left}{\mathopen{}\mathclose\bgroup\originalleft}
\renewcommand{\right}{\aftergroup\egroup\originalright}

\def\bas#1\eas{\begin{align*}#1\end{align*}}
\def\basn#1\easn{\begin{align}#1\end{align}}

\definecolor{changehighlightcolor}{RGB}{0,170,255}
\definecolor{changeotherhighlightcolor}{RGB}{204,136,0}
\definecolor{changeotherotherhighlightcolor}{RGB}{170,0,85}
\def\bc#1#2\ec#3{\ifnum#1>\highlightchangessince
\ifnum#1>\otherhighlightchangessince\ifnum#1>\otherotherhighlightchangessince\textcolor{changeotherotherhighlightcolor}{\ignorespaces#2\unskip}\else\textcolor{changeotherhighlightcolor}{\ignorespaces#2\unskip}\fi\else\textcolor{changehighlightcolor}{\ignorespaces#2\unskip}\fi\else\ignorespaces#2\unskip\fi}
\newcommand{\marginparc}[2]{\ignorespaces\ifnum#1>\highlightchangessince\marginpar{\vspace*{-\baselineskip}\raggedright\singlespacing\footnotesize\textsl{\textcolor{changehighlightcolor}{$\star$~#2}}}\fi\ignorespaces}

\begin{document}

\onehalfspacing

\title{%
MCMC-Based Inference in the Era of Big Data: A~Fundamental Analysis of the Convergence Complexity of High-Dimensional Chains
\author{Bala Rajaratnam and Doug Sparks\\
Stanford University
}
}
\maketitle

\begin{abstract}\noindent
Markov chain Monte Carlo (MCMC) lies at the core of modern Bayesian methodology, much of which would be impossible without it. Thus, the convergence properties of MCMCs have received significant attention, and in particular, proving (geometric) ergodicity is of critical interest. Trust in the ability of MCMCs to sample from modern-day high-dimensional posteriors, however, has been limited by a widespread perception that chains typically experience serious convergence problems in such regimes. Though there may be a good practical understanding of convergence problems (and the associated role of priors) in some settings, a clear theoretical characterization of these problems is not available. Current methods for obtaining convergence rates of such MCMCs typically proceed as if the dimension of the parameter, $p$, and sample size, $n$, are fixed. In this paper, we first demonstrate that contemporary methods have serious limitations when the dimension grows. We then propose a framework for rigorously establishing the convergence behavior of commonly used high-dimensional  MCMCs. In particular, we demonstrate theoretically the precise nature and severity of the convergence problems of popular MCMCs when implemented in high dimensions, including phase transitions in the convergence rates in various $n$~and~$p$ regimes. We then proceed to show a universality result for the convergence rate of MCMCs across an entire spectrum of models. We also show that convergence problems in some important models effectively eliminate the apparent safeguard of geometric ergodicity. We then demonstrate theoretical principles by which MCMCs can be constructed and analyzed to yield bounded geometric convergence rates (essentially recovering geometric ergodicity) even as the dimension~$p$ grows without bound. Additionally, we propose a diagnostic tool for establishing convergence (or the lack thereof) for high-dimensional MCMCs.
\end{abstract}

\thispagestyle{empty}

\setcounter{page}{0}

\section{Introduction}
\label{sec:intro}

Markov chain Monte Carlo (MCMC) is an indispensable tool that has enabled much of modern Bayesian inference, and advances in MCMC have revolutionized Bayesian methodology in recent decades \citep[see][for an overview]{diaconis2009}.
The rise of MCMC has been aided by the steady increase in computing capabilities, which has enabled many complex and sophisticated MCMC techniques.
Thus, modern MCMC allows consideration of Bayesian posteriors for which no closed-form inferential solutions can be obtained. The applicability of MCMC to a wide range of problems has enabled an ``honest exploration" of the Bayesian posterior \citep{jones2001}.

An enormous amount of effort has been invested in establishing convergence properties of Markov chains. The basic question that most such work seeks to answer is the question of how long the chain must be run in order to approximate posterior quantities of interest to a desired precision. To this end, a primary goal is typically to show that chains arising in commonly used Bayesian methods are geometrically ergodic. A general approach 
for establishing geometric ergodicity
was provided by \citet{rosenthal1995}, and many subsequent results have been based on the method that \citeauthor{rosenthal1995} laid out. Further details can be found in the work of \citet{meyn1993}, \citet{gilks1995}, \citet{jones2001}, \citet{flegal2008}, and the references therein.

Modern high-dimensional settings have created new challenges when considering the limiting properties of inferential procedures.  Statistical theory has traditionally considered regimes in which the sample size~$n$ is large and the number of parameters~$p$ is small. However, there is now much interest in so-called ``small~$n$, large~$p$'' or ``large~$n$, large~$p$'' settings, and considerable advances have been made toward asymptotics in various sample complexity regimes \citep[see, e.g.,][for an overview]{hero2015,hero2016}. Bayesian inference enjoys certain advantages in such high-dimensional settings. Bayesian procedures often yield natural ways to undertake regularization and provide straightforward quantification of uncertainty. For both Bayesian and frequentist inference, substantial attention has been paid to two different types of complexity in high-dimensional regimes.  The first type, computational complexity, considers the computing time and resources that are required for the execution of an inferential algorithm.  The second type, sample complexity, deals with the fundamental ability to recover an underlying signal in various $n$ and $p$ regimes. However, a third type of complexity is also of vital importance for modern MCMC schemes involving large numbers of parameters.  This concept, which we call \emph{convergence complexity}, is an issue that is unique to Bayesian inference. More precisely,  convergence complexity considers the ability of an MCMC scheme to draw samples from the posterior, and how the ability to do so changes as the dimension of the parameter set grows. Although MCMC is perhaps the most important tool of modern Bayesian inference, to our knowledge a systematic theoretical treatment of the convergence complexity of modern Markov chains in various $n$ and~$p$ regimes is not available. 

The need for such an investigation also stems from the perceived scalability (or lack thereof) of Bayesian inferential methods to modern high-dimensional settings. It is well understood that approaches based on $\ell_1$ or lasso regularization have enabled frequentist approaches to be scaled to
high-dimensional
settings. However,
despite heroic
efforts from the MCMC community, there is a still a widely held perception that scaling MCMCs to modern high-dimensional settings is simply not feasible. The end result is that the benefits of posterior inference are lost (especially the ability to readily quantify uncertainty). Having said this, there is however a general understanding among practitioners that scaling classical MCMCs to very high dimensions can be problematic and that prior specification can play a role in convergence issues.
Thus, we believe that a general framework for undertaking a
theoretical analysis of high-dimensional MCMCs in various $n$ and~$p$ regimes is long 
overdue, since
it is vital to understand the effectiveness of using MCMCs as a tool to draw from high-dimensional posteriors.

In this paper, we undertake a detailed investigation of the convergence complexity of modern
MCMCs
that form the basis of more sophisticated models in many applications%
\setcitestyle{notesep={; }}%
\citep[see][and other standard Bayesian texts for concrete examples]{gelman2013,ohagan2010}.%
\setcitestyle{notesep={, }}%
Specifically, we first study Markov chains associated with a Bayesian analysis of the standard regression model and extensions thereof. These extensions include the Bayesian lasso, the Bayesian elastic net, and the spike-and-slab approach. We demonstrate that for Markov chains associated with standard regression-type models (and extensions thereof), the apparent theoretical safeguard of geometric ergodicity is merely an illusion if the dimension~$p$ grows faster than the sample size~$n$. More precisely, although the chain is indeed geometrically ergodic for any fixed $n$ and~$p$, we show that the rate constant~$r\equiv r_{n,p}$ tends to~$1$ if $p$~grows faster than~$n$. Thus, the convergence of these Markov chains may still be quite slow in modern high-dimensional settings. Our results also carry over directly to graphical models. We then contrast this convergence complexity with that of chains of other popular models, including the class of hierarchical models and the multivariate mean model. 
We demonstrate that
fortunately and contrary to perception,
convergence behavior seen in high-dimensional regression models is not inherent to many commonly used high-dimensional Markov chains. Though it is not possible to analyze all models and various prior specifications, the spectrum of models we do consider gives general and compelling insights into convergence behavior.

In all the models we consider, we are able to obtain exact or sharp convergence rates for various Markov chains using novel technical approaches. The significance of doing so is better understood by first recognizing that establishing geometric ergodicity itself is considered a challenging task
and
is often undertaken on a case-by-case basis for various MCMCs. Thus we believe that the ability to obtain sharp results for the geometric convergence rate in terms of $n$ and $p$ constitutes a significant step forward
in understanding
the behavior of high-dimensional MCMCs. 
It also simultaneously delivers novel theoretical methods for deriving such convergence rates.

The remainder of the paper is organized as follows. Section~\ref{sec:preliminaries} contains a discussion of known results for Markov chains and considers these results in high-dimensional settings. Section~\ref{sec:rr} provides a rigorous consideration of high-dimensional convergence problems in the Bayesian regression framework. In Section~\ref{sec:lasso}, we investigate
extensions of standard Bayesian regression, including
the Bayesian lasso, Bayesian elastic net, and spike-and-slab regression.
In Section~\ref{sec:mean} we consider the multivariate Gaussian mean model. In Section~\ref{sec:nh} we consider normal hierarchical models with known and unknown variances.
Additionally, we propose a diagnostic tool for assessing convergence in various $n$ and~$p$ regimes. 
Section~\ref{sec:ir} demonstrates how convergence rates that are uniformly bounded away from~$1$ may be obtained theoretically for high-dimensional Markov chains. Further discussion and conclusions are presented in Section~\ref{sec:conclusion}.

\section{Preliminaries}
\label{sec:preliminaries}

In this section, we present some preliminary results on the behavior of Markov chains.
First, we review notions of Markov chain convergence and associated convergence rates, along with methods by which such properties can be rigorously established.
We then consider Gibbs sampling and relevant properties of the joint and marginal chains that arise from such schemes.
Next, we discuss the role of autocorrelation in Gibbs sampling and its relationship to a chain's overall convergence behavior.
Finally, we introduce the
concept
of \emph{convergence complexity},
by which we mean the 
dependence of the chain's geometric convergence rate~$r$ on
the sample size~$n$ and the dimension of the parameter~$p$.
We introduce examples to illustrate this
concept
and to motivate the work in the remainder of the paper.

\subsection{Convergence Rates and Geometric Ergodicity}

The
total variation distance between two probability measures $P$ and~$Q$ defined on the same $\sigma$-algebra~$\mathcal F$ is
$d_{\tv}(P,Q)=\sup_{A\in\mathcal F}\,\left|P(A)-Q(A)\right|$.
%
In terms of Markov chains, if $P^k_{\bm x_0}$ denotes the distribution of the $k$th iterate of
a Markov
chain with starting point~$\bm x_0$ and $\Pi$ denotes the chain's stationary distribution (i.e., the target posterior), then we are typically interested in 
$d_{\tv}(P^k_{\bm x_0},\Pi)$.
It is typically desirable for the distance
$d_{\tv}(P^k_{\bm x_0},\Pi)$
to converge to zero at a geometric rate, i.e., that there exist $M_{\bm x_0}>0$ and $0<r<1$ such that
\begin{align}
d_{\tv}(P^k_{\bm x_0},\Pi)
\le M_{\bm x_0} r^k
\label{ge}
\end{align}
for every $k\ge1$.
When such constants exist
(and provided certain other regularity conditions hold), the Markov chain is said to be \emph{geometrically ergodic}.

An active area of current research is the establishment of geometric ergodicity for Markov chains
commonly
used
in applied Bayesian statistics.
Rigorous proofs of such results can be challenging to obtain, and different models and sampling schemes must often be handled on a case-by-case basis
(see in particular the rich array of results established by the work of J.~Hobert and co-authors).
Although a variety of methods may be used to prove geometric ergodicity \citep[see, e.g.,][]{meyn1993}, these methods often
establish the
\emph{existence} of a constant $0<r<1$
satisfying the geometric bound in~(\ref{ge}).
More sophisticated techniques are typically needed to find
quantitative 
bounds on the geometric convergence rate.
The most widely employed approach for
finding such bounds
has been the method set forth by \citet{rosenthal1995}.  This method proceeds by establishing a \emph{drift condition} and an associated \emph{minorization condition} for the Markov chain in question.
Let $(\bm X_k:k\ge0)$ be a Markov chain with state space $\mathcal X\subseteq\reals^p$ and associated Borel $\sigma$-algebra~$\mathcal B$.  We assume the Markov chain satisfies certain regularity conditions, e.g., those of \citet{jones2001}.
Let $P_{\bm x}$ denote its transition kernel, i.e., $P_{\bm x}(A)$ is the probability that $\bm X_{i+1}\in A{
{}\in\mathcal B
}$ given that $\bm X_i=\bm x$.  Let $\Pi$ denote the stationary distribution of the chain.
The chain satisfies a drift condition if there exist a function $V:\mathcal X\to[0,\infty)$ and constants $0<\lambda<1$ and $b<\infty$ such that
\basn
\int V\,dP_{\bm x}=E\left[V\left(\bm X_{i+1}\right)\mid\bm X_i=\bm x\right]\le\lambda\, V(\bm x)+b
\quad\text{ for all }\bm x\in\mathcal X.
\label{drift}
\easn
The chain satisfies a minorization condition if there exist a probability measure~$Q$ on~$\mathcal B$, a set $C$ with $\Pi(C)>0$, and a constant $\epsilon>0$ such that
\basn
P_{\bm x}(A)\ge\epsilon\,Q(A)\quad\text{ for all }\bm x\in C\text{ and all }A\in\mathcal B.
\label{minorization}
\easn
The establishment of geometric ergodicity requires that the set~$C$ be chosen specifically as $C=\{\bm x\in\mathcal X: V(\bm x)\le d\}$ for some $d>2b/(1-\lambda)$.  \citet{jones2001} provide an accessible conceptual discussion of the connections between these conditions and geometric ergodicity.

The convergence rates of Markov chains can also be investigated using tools and techniques from functional analysis.%
\setcitestyle{notesep={; }}%
\citep[See][and the references therein for further details.]{liu1994w,liu1994l}%
\setcitestyle{notesep={, }}%
Let $\mathcal X$ denote the state space of a Markov chain $(\bm X_k:k\ge0)$ with stationary distribution~$\Pi$, and let $L_0^2(\Pi)$ denote the space of all functions $h:\mathcal X\to\reals$ such that $E[h(\bm X)]=0$ and $\var[h(\bm X)]<\infty$
where
$\bm X\sim\Pi$.
For any function $g\in L_0^2(\Pi)$, its norm $\|g\|$ is defined as the square root of $\|g\|^2=E\left\{\left[g(\bm X)\right]^2\right\}$ with $\bm X\sim\Pi$.
Now define
the
\emph{forward operator}~$\bm F$ mapping $L_0^2(\Pi)$ to itself by
\bas
{
\bm Fg(\bm x)=E\left[g(\bm X_1)\mid \bm X_0=\bm x\right].
}
\eas
The norm of the operator $\bm F$ is defined as
$\|\bm F\|=\sup_{\|g\|=1}\|\bm Fg\|$,
and its spectral radius is
$r_{\bm F}=\lim_{k\to\infty}\|\bm F^k\|^{1/k}$%
,
noting that the $k$-step forward operator $\bm F^k$ is simply
$\bm F^k g(\bm x)=E\left[g(\bm X_k)\mid \bm X_0=\bm x\right]$.
If the Markov chain is reversible, then $\bm F$ is self-adjoint.
It follows that the norm~$\|\bm F\|$, spectral radius~$r_{\bm F}$, and largest eigenvalue of~$\bm F$ all share a common value~$r$.
Moreover, under certain regularity conditions, the chain is geometrically ergodic with geometric rate constant~$r$ if $r<1$
\citep{liu1994w,liu1995,liu2004}.

\subsection{Gibbs Sampling and Marginal Chains}
\label{subsec:marginal}

Many general techniques have been developed for
constructing
Markov chains to sample from a target posterior, such as the accept--reject algorithm and the Metropolis--Hastings algorithm \citep{metropolis1953,hastings1970}.
Many of these methods are based on proposing a new point and then either accepting or rejecting it with some probability.
For such methods to obtain reasonably large acceptance probabilities in high-dimensional settings, they must propose points that are very close to the chain's current state,
which
in turn
limits their ability to quickly traverse the state space%
\setcitestyle{notesep={; }}%
\citep[see, e.g., the work on optimal scaling of][and the references therein]{roberts2001o,beskos2009}.%
\setcitestyle{notesep={, }}%

However, one special case of the Metropolis--Hastings algorithm that is quite useful in high dimensions is known as the Gibbs sampler \citep{geman1984}.
By construction, Gibbs samplers propose a new point in such a way that the acceptance probability is~$1$.  Thus, they
are very useful for tractably sampling from the posterior in high-dimensional
settings.
Moreover,
a preponderance of theoretical convergence results
establishing geometric ergodicity
for specific MCMC schemes are for Gibbs samplers.  Indeed, the machinery by which these 
theoretical
results are established \citep[such as the method of][]{rosenthal1995} is inherently better suited
to Gibbs sampling than to
other approaches%
\setcitestyle{notesep={; }}%
\citep[see, e.g.,][and the references therein]{choi2013e,khare2013,roman2015}.%
\setcitestyle{notesep={, }}%
Thus,
let $\{(\bm X_k,\bm Y_k):k\ge0\}$ be a Markov chain constructed as a Gibbs sampler that alternates between drawing $\bm X$ and~$\bm Y$
and has
(joint) stationary distribution~$\Pi$.
It is well known that the marginal sequences $\{\bm X_k:k\ge0\}$ and $\{\bm Y_k:k\ge0\}$ are
reversible Markov chains \citep[e.g.,][]{liu1994w}.
Moreover, it can be shown that
either all three chains are geometrically ergodic
with the same rate
or none of the chains are geometrically ergodic
\citep{liu1994w}.
Thus, to establish the geometric convergence rate of the joint chain, it suffices to find the largest eigenvalue of the forward operator of either marginal chain (or to find the convergence rate of the marginal chain by some other method).
This approach can simplify proofs of geometric convergence rates if one of the marginal chains is more analytically tractable than the joint chain.

\subsection{Autocorrelation Structure}
\label{subsec:autocorrelation}

Even if a Markov chain is approximately sampling from its stationary distribution, the draws are \emph{not} approximately independent (in general).  The autocorrelation strutcture
between successive iterates
can thus be of great importance to the MCMC practicitioner when considering questions such as the amount of error inherent to the MCMC samples, i.e., how much an MCMC approximation can be expected to differ from the corresponding ``true'' result.  From a more theoretical perspective, the autocorrelation structure of the chain is also of interest due to its connections to other properties of the chain, including its convergence properties.  Indeed, it is intuitively
clear
that the greater the correlation between successive iterates, the more iterations it should take for the effects of the starting point (or starting distribution) to
``wash out.''

To properly state results
on the autocorrelation structure of Markov chains,
we first
introduce a slightly more general notion of correlation.  The \emph{maximal correlation} between two random variables
$\bm X_1$ and $\bm X_2$
(with some joint distribution) is defined as
\bas
\gamma({
\bm X_1,\bm X_2
})=\sup_{g_1,g_2}\;\corr\left[g_1({
\bm X_1
}),g_2({
\bm X_2
})\right],
\eas
where the supremum is taken over all functions $g_1$ and $g_2$ such that the variances $\var[g_1({
\bm X_1
})]$ and $\var[g_2({
\bm X_2
})]$ are finite and nonzero.
If
$({
\bm Y_k
}:k\ge0)$
is a stationary Markov chain with
${
\bm Y_k
}\sim\Pi$,
then
the norm of its forward operator $\bm F$
can be shown to be
equal to the maximal correlation between successive iterates, i.e.,
\bas
\|\bm F\|=\gamma({
\bm Y_k,\bm Y_{k+1}
})
\eas
\citep{liu1994l}.
Now consider the specific case of a two-step Gibbs sampler
to draw from some posterior distribution $\pi(\bm\theta,\bm\phi\mid\bm Z)$,
where $\bm\theta$ and $\bm\phi$ represent unknown parameters and $\bm Z$ represents observed data.
This Gibbs sampler draws a sequence of iterates $(\bm\theta_k,\bm\phi_k)$ by drawing alternately from the conditional posterior distributions $\pi(\bm\theta\mid\bm\phi,\bm Z)$ and $\pi(\bm\phi\mid\bm\theta,\bm Z)$.
Suppose the chain is stationary,
and let $\gamma(\bm\theta,\bm\phi\mid\bm Z)$ denote the maximal correlation between $\bm\theta$ and~$\bm\phi$ under the joint posterior.
Then the forward operators of the joint and marginal Gibbs sampling chains all have spectral radius equal
to the square of the maximum posterior correlation as given by
$[\gamma(\bm\theta,\bm\phi\mid\bm Z)]^2$
\citep{liu1994w}.

\subsection{Convergence Complexity}
\label{subsec:cvgc-cplx}

It is
obviously extremely
useful to show that any given Markov chain is geometrically ergodic.  Still, a full characterization of the behavior of the chain cannot be reduced to simply the binary question of whether a chain does or does not have this property.  Even if a chain is geometrically ergodic, the specific value of the
geometric
rate constant~$r$
in the bound in~(\ref{ge})
can be of great practical
importance, especially in ultra-high-dimensional applications.
More specifically, if $r$ is very close to~$1$, then a chain may still converge quite slowly despite the fact that it is geometrically ergodic,
a fact that has been noted
in the literature
\citep{papaspiliopoulos2007,papaspiliopoulos2008,woodard2013}.
Of course, a value of~$r$ close to~$1$ would immediately raise the question of the sharpness of the associated inequality, i.e.,
whether the bound in~(\ref{ge})
could be satisfied with some smaller choice of~$r$.
However, such questions
regarding the
\emph{convergence complexity}
of~$r$
may be difficult to answer when
existing
methods provide only upper bounds.

More generally, in modern applications, various notions of complexity are often of interest.
Practical limitations of computing time have motivated the consideration of
\emph{computational complexity},
and fundamental questions of signal recovery have led to the investigation of
\emph{sample complexity}
\citep[see][and the references therein]{hero2015,hero2016}.
For MCMC-based inferential procedures, the convergence complexity of the Markov chain in various $n$ and~$p$ regimes is
an important
issue that warrants attention.
In the context of Markov chain convergence, some authors have investigated the relationship between a chain's convergence behavior and the sample size of the data on which the target posterior is conditioned%
\setcitestyle{notesep={; }}%
\citep[see][and the references therein]{mossel2006,papaspiliopoulos2007,woodard2013}.%
\setcitestyle{notesep={, }}%
However, in modern high-dimensional statistics,
there is also great interest in
the behavior of the chain
as the
dimension of
the unknown parameter vector
grows without bound.
If an MCMC scheme for a Bayesian method is based on a Markov chain that is geometrically ergodic, then a
key question
of practical significance is how the associated rate constant~$r_{n,p}$ behaves in
various $n$ and~$p$ regimes.
More specifically, a key question for any particular asymptotic regime is whether $r_{n,p}\to1$,
or equivalently, whether the number of iterations required for approximate convergence
(to within some fixed distance~$\epsilon$ of the stationary distribution)
tends to infinity
as
$n$ or~$p$
tends to infinity.
An answer in the affirmative
would suggest that the Markov chain could converge quite slowly in such a regime despite the apparent theoretical safeguard of geometric ergodicity. Note also that the constant~$M_{\bm{x}_0}$ in~(\ref{ge}) can be disregarded in an asymptotic analysis as the geometric component~$r^k$ drives the convergence rates. It is clear that $r^k$ is the leading term in the bound.

It may seem that
an approach
to
answering
the question of convergence complexity may be provided by the method of \citet{rosenthal1995}.
Since this method
is commonly used to obtain
an upper bound
for the geometric convergence
rate~$r$,
it is natural to ask whether this bound
can be directly analyzed in various $n$ and~$p$ regimes.
Somewhat problematically,
such upper bounds
may tend to~$1$ as
$n$ or~$p$ tends
to
infinity.
We illustrate the behavior of these upper bounds in the following examples.

\begin{ex}
\label{ex:logistic}
Consider a Bayesian analysis of the logistic regression model
\bas
Y_i\mid\bm\beta&\sim\ind\bin\left[1,\,\logit^{-1}\left(\bm x_i^T\bm\beta\right)\right]\qquad\text{ for each }i\in\{1,\ldots,n\},\\
\bm\beta&\sim N_p(\bm0_p,\,\lambda^{-1}\bm I_p),
\eas
where $\bm x_1,\ldots,\bm x_n\in\reals^p$ and $\lambda>0$ are
known, and
where $\logit^{-1}(u)=e^u/(1+e^u)$.
A Gibbs sampler to
draw
from the posterior of a slightly more general version of this
construction
was developed by \citet{polson2013}.
\citet{choi2013e} used
the method of \citet{rosenthal1995} to prove
that this Gibbs sampler is geometrically ergodic (in fact, uniformly so) with a convergence rate bounded above by the quantity
$\tilde r=1-\delta$,
where
\bas
\delta
&=\lambda^{p/2}\left(\det\bm A\right)^{-1/2}\,2^{-n}\exp\left(-\frac n4-\frac1{4\lambda}\left\|\bm X\bm A^{-1/2}\bm X^T\tilde{\bm Y}\right\|_2^2\right),
\eas
where $\bm A=\frac12\bm X^T\bm X+\lambda\bm I_p$
and $\tilde{\bm Y}=\bm Y-\frac12\bm1_n$
\citep[see Proposition~3.1 of][]{choi2013e}.
Thus, the results of \citet{choi2013e} essentially establish the upper bound
\basn
d_{\tv}(G_k,G)\le M \tilde r^k=M(1-\delta)^k
\label{bad-bound-lasso}
\easn
for some $M>0$, where $G_k$ denotes the distribution of the $k$th iterate of the joint chain and $G$ denotes the corresponding stationary distribution.
The following lemma establishes
the behavior of
$\tilde r=1-\delta$, and hence the behavior of the upper bound in~(\ref{bad-bound-lasso}),
as $n$ or~$p$ grows.  Its proof and all subsequent proofs are provided in the
supplemental sections.
\end{ex}

\begin{lem}
\label{lem:ex-logistic}
Consider the upper bound $\tilde r=1-\delta$ provided by \citet{choi2013e} for the convergence rate of the logistic regression Gibbs sampler in Example~\ref{ex:logistic}.
Then $\tilde r\to1$ exponentially fast as $n\to\infty$ (for fixed~$p$) and as $p\to\infty$ (for fixed~$n$).
\end{lem}

Thus, if either $n$ or $p$ tends to infinity, then the upper bound on the convergence rate tends to~$1$,
and it does so exponentially fast.
The apparent safeguard of geometric ergodicity is therefore misleading in high-dimensional applications when either $n$ or~$p$ is very large.

\begin{ex}
\label{ex:lasso}
Consider the Bayesian lasso framework of \citet{park2008}:
\bas
\bm Y\mid\bm\beta,\sigma^2,\bm\tau&\sim N_n(\bm X\bm\beta,\,\sigma^2\bm I_n),\\
\bm\beta\mid\sigma^2,\bm\tau&\sim N_p(\bm0_p,\sigma^2\bm D_{\bm\tau}),\\
\pi(\sigma^2)&\propto1/\sigma^2,\\
\tau_j&\sim\iid\expd(\lambda/2)\qquad\text{ for each }j\in\{1,\ldots,p\},
\eas
where $\bm D_{\bm\tau}=\diag(\tau_1,\ldots,\tau_p)$.
\citet{park2008} provide a Gibbs sampler to draw from the posterior corresponding to the Bayesian lasso.
\citet{khare2013} demonstrated a useful result that this Gibbs sampler is geometrically ergodic
with an upper bound~$\tilde r$ for its geometric rate constant.
The following lemma establishes the asymptotic behavior of~$\tilde r$.
\end{ex}

\begin{lem}
\label{lem:ex-lasso}
Consider the upper bound~$\tilde r$ provided by \citet{khare2013} for the convergence rate of the Bayesian lasso Gibbs sampler in Example~\ref{ex:lasso}.  Then $\tilde r\to1$ exponentially fast as $n\to\infty$ (for fixed~$p$) and as $p\to\infty$ (for fixed~$n$).
\end{lem}

Using the bounds from
Examples~\ref{ex:logistic}~and~\ref{ex:lasso}, the number of iterations required to obtain convergence to a desired tolerance in total variation norm grows at least exponentially fast in~$p$.
Note that
it is an upper bound and not a lower bound.
To gain some insight into the
possible disadvantages
of these upper bounds
in high-dimensional settings, consider the dimensions of the various quantities that appear
in the proof of geometric ergodicity.
More specifically, if the dimension  of the distributions $P$ and~$Q$ in 
the minorization condition in~(\ref{minorization})
is~$p$,
consider the constant $\epsilon\equiv\epsilon_{n,p}$ that appears 
in the minorization condition in~(\ref{minorization}),
noting that $1-\epsilon$ essentially corresponds to the upper bound for the geometric convergence rate.
This constant will often take the form
$\epsilon_{n,p}=(\epsilon_\star)^p$
for some $0<\epsilon_\star<1$.  For example, if $P$ is expressable as a product of $p$ independent marginal distributions, then it will often be necessary to find
a bound
akin to
the minorization condition in~(\ref{minorization})
for each such marginal
distribution.
Thus,
the ``overall'' $\epsilon$ will be a product of $p$ ``individual'' $\epsilon$-type quantities.
Hence, it is often the case that $\epsilon_{n,p}\to0$ as $p\to\infty$.
If indeed $\epsilon_{n,p}\to0$, then the resulting bound on the geometric convergence rate (namely, $1-\epsilon_{n,p}$ or some power thereof) tends to~$1$ and hence is not useful.
Such problems have hampered attempts to combine \citeauthor{rosenthal1995}'s method with dimensional asymptotics
\citep[see, for instance,][]{hu2012}.
Thus, alternative strategies may be required
if we wish to obtain convergence rates that do not tend to~$1$ as the dimension grows.
On the other hand, it may instead be asked whether there exist settings in which \citeauthor{rosenthal1995}'s technique can overcome these high-dimensional obstacles.
We show later in Section~\ref{sec:ir} that a specifically tailored application of \citeauthor{rosenthal1995}'s approach may still lead to a convergence rate that is bounded away from~$1$.

\section{%
Regression Models
\&
Graphical
Models}
\label{sec:rr}

We now begin to pursue our goal of obtaining precise bounds for the geometric convergence rate of high-dimensional MCMCs. To this end, we first undertake a thorough investigation of the behavior of the Gibbs sampler for a 
Bayesian analysis of the standard regression model.
The
properties of this basic model are essential for illuminating the problems that certain types of Gibbs samplers encounter in high-dimensional regimes.  These results also lead directly to corresponding results for Gibbs samplers in
an important class of graphical models.


Consider
a Bayesian analysis of the standard regression model
\begin{align}
\bm Y\mid\bm\beta,\sigma^2&\sim N_n(\bm X\bm\beta,\,\sigma^2\bm I_n),\notag\\
\bm\beta\mid\sigma^2&\sim N_{p}(\bm0_{p},\,\lambda^{-1}\sigma^2\bm I_{p}),\label{model-rr}\\
\pi(\sigma^2)&\propto1/\sigma^2,\quad\sigma^2>0,\notag
\end{align}
where $\bm X$ is a known $n\times p$ matrix of covariate values and $\lambda>0$ is a known regularization parameter.
We assume $n\ge5$ to
facilitate the technical analysis.
%
Then a Gibbs sampler to draw from the joint posterior under~(\ref{model-rr}) may be constructed by taking an initial value $\sigma^2_0>0$ and then drawing (for every $k\ge1$)
\basn
\left.\bm\beta_k\given\sigma^2_{k-1},\bm Y\right.&\sim N_p\left(\tilde{\bm\beta},\,\sigma^2_{k-1}\bm A^{-1}\right),\notag\\
\left.\sigma^2_k\given\bm\beta_k,\bm Y\right.&\sim\inversegamma\left\{\frac{n+p}2,\,\frac12\left[\left(\bm\beta_k-\tilde{\bm\beta}\right)^T\bm A\left(\bm\beta_k-\tilde{\bm\beta}\right)+C\right]\right\},\label{gibbs-rr}
\easn
where $\bm A=\bm X^T\bm X+\lambda\bm I_p$ (which is positive-definite), $\tilde{\bm\beta}=\bm A^{-1}\bm X^T\bm Y$, and $C=\bm Y^T(\bm I_n-\bm X\bm A^{-1}\bm X^T)\bm Y$.

\subsection{Convergence Rates}
\label{subsec:cvgc-rr}

In order to understand the convergence behavior of the Gibbs sampler in~(\ref{gibbs-rr}) corresponding to a standard regression model, we proceed to undertake a fundamental analysis of this MCMC scheme.


We now establish 
sharp bounds for the
geometric convergence rate of the
standard Bayesian regression Gibbs sampler in~(\ref{gibbs-rr})
in total variation
norm in terms of the dimension~$p$ and sample size~$n$.
For every $k\ge0$, let $F_k(\sigma^2_0)$ denote the joint distribution of $(\bm\beta_k,\sigma^2_k)$ for the chain in~(\ref{gibbs-rr}) started with initial value~$\sigma^2_0$, and let $F$ denote the stationary distribution of this chain, i.e.,
the true joint posterior of 
 $(\bm\beta,\sigma^2)$.
Then we have the following
result.

\begin{thm}
\label{thm:tv-rr}
For the standard Bayesian regression Gibbs sampler in~(\ref{gibbs-rr}), there
exist $0<M_1\le M_2$ such that
\bas
M_1\left(\frac{p}{n+p-2}\right)^k\le d_{\tv}\left[F_k\left(\sigma^2_0\right),\,F\right]\le
M_2\left(\frac{p}{n+p-2}\right)^k
\eas
for every~$k\ge0$.
\end{thm}

Note that if $p\equiv p_n$ grows faster than~$n$, then the 
sharp
bound provided by Theorem~\ref{thm:tv-rr} tends to~$1$.  Hence, Theorem~\ref{thm:tv-rr} provides our first theoretical indication of
the precise nature of the
convergence problem in high-dimensional Markov chains.
In Supplemental Section~\ref{sec:supp-rr}, we also
obtain similar rates in terms of Wasserstein distance~$d_W$,
including expressions for the multiplicative constants in the bounds (i.e., the equivalent of $M_1$ and~$M_2$ in Theorem~\ref{thm:tv-rr}).  These results allow us to derive
expressions
for the number of iterations required for convergence of the chain to within a given 
tolerance~$\epsilon>0$.
We show that the number of iterations required for convergence to within~$\epsilon$ grows only linearly in~$p$ and not exponentially.
This is an encouraging result.
A complete discussion of convergence rates in terms of Wasserstein distance can be found in
Supplemental Section~\ref{sec:supp-rr}.

We also note that \citet{roman2012} establish a useful result concerning geometric ergodicity of a Gibbs sampler for the linear mixed model. Their results, however, do not provide a quantitative bound on the geometric convergence rate itself. Thus it is not clear how the geometric rate behaves as a function of the sample size~$n$ and the dimension~$p$. In contrast our analysis obtains sharp quantitative bounds for this geometric convergence rate in terms of $n$ and~$p$ for standard regression.

\subsection{%
Characterization of Convergence
}
\label{subsec:further}

The behavior of the Gibbs sampler in
Subsection~\ref{subsec:cvgc-rr}
in
various $n$ and~$p$ regimes
can be further examined
by considering the nature of the joint posterior distribution
itself.
The following lemma provides insight
regarding the posterior correlation between~$\sigma^2$ and a particular function of~$\bm\beta$.
Specifically, let $\bm\theta=\bm A^{1/2}(\bm\beta-\tilde{\bm\beta})$,
and note that
$\|\bm\theta\|_2$
represents a
Mahalanobis-type
distance between $\bm\beta$ and the posterior mean~$\tilde{\bm\beta}$.
Then we have the following result.

\begin{lem}\label{lem:post-corr-rr}
For
the posterior of
the standard Bayesian regression framework in~(\ref{model-rr}),
\bas
\corr\left(\sigma^2,\,
\left\|\bm\theta\right\|_2^2
\given\bm Y\right)=\sqrt{\frac p{n+p-2}}.
\eas
\end{lem}

Thus, by Lemma~\ref{lem:post-corr-rr}, the posterior correlation of $\sigma^2$ and
$\|\bm\theta\|_2^2
$
tends to~$1$ asymptotically if $p_n\ne O(n)$.  This behavior is a consequence of the way in which the prior on $\bm\beta$ and~$\sigma^2$ is specified
under the Bayesian regression framework in~(\ref{model-rr}).
Specifically, observe that under
this prior,
we have
\basn
\left.\frac1p\|\bm\beta\|_2^2\given\sigma^2\right.\sim\gammad\left(\frac p2,\frac{p\lambda}{2\sigma^2}\right),\qquad
\sigma^2\mid\bm\beta\sim\inversegamma\left(\frac p2,\frac\lambda2\|\bm\beta\|_2^2\right).
\label{prior-informative}
\easn
Observe from~(\ref{prior-informative}) that
\bas
E\left(\frac1p\|\bm\beta\|_2^2\given\sigma^2\right)=\frac{\sigma^2}\lambda,
\qquad
\var\left(\frac1p\|\bm\beta\|_2^2\given\sigma^2\right)=\frac2p\left(\frac{\sigma^2}\lambda\right)^2.
\eas
If $p$ is large, the prior distribution of $p^{-1}\|\bm\beta\|_2^2\mid\sigma^2$ is highly concentrated around $\sigma^2/\lambda$.
Similarly, observe from~(\ref{prior-informative}) that
\bas
E(\sigma^2\mid\bm\beta)=\frac{\lambda\|\bm\beta\|_2^2}{p-2},
\qquad
\var(\sigma^2\mid\bm\beta)=\frac2{p-4}\left(\frac{\lambda\|\bm\beta\|_2^2}{p-2}\right)^2.
\eas
If $p$ is large, the prior distribution of $\sigma^2\mid\bm\beta$ is highly concentrated around $\lambda\|\bm\beta\|_2^2/(p-2)$.
Thus, for large~$p$,
the prior is highly informative about the relationship between $\|\bm\beta\|_2^2$ and~$\sigma^2$.
It can be shown that this
high dependence
carries over to the posterior in the regime where $p\gg n$
because the data is overwhelmed by the prior.
The posterior dependence between the parameters manifests itself through the conditionals that are used in the Gibbs sampler.  The value of $\|\bm\theta_k\|_2^2=(\bm\beta_k-\tilde{\bm\beta})^T\bm A(\bm\beta_k-\tilde{\bm\beta})$ is heavily dependent on the value of $\sigma^2_{k-1}$, and in turn
the
value
of $\sigma^2_{k-1}$ is heavily dependent on the value of $\|\bm\theta_{k-1}\|_2^2$.
Thus, each iteration of the
joint and marginal Gibbs sampling chains
is highly dependent on the previous
iteration.
The same concept may be alternatively expressed by stating that the chain \emph{mixes} poorly.  More specifically, one manifestation of this poor mixing behavior is high autocorrelation between successive values $\sigma^2_k$ and~$\sigma^2_{k+1}$, even if the chain is in its stationary state.
This property is formalized
in
the following
lemma.
To state the result, let $G$ denote the stationary distribution of the marginal $\sigma^2_k$ chain for the chain in~(\ref{gibbs-rr}), i.e., the true marginal posterior of~$\sigma^2$ under~(\ref{model-rr}), which is the $\inversegamma(n/2,\,C/2)$ distribution.

\begin{lem}\label{lem:autocorr-rr}
Consider the standard Bayesian regression Gibbs sampler in~(\ref{gibbs-rr}).
If $\sigma^2_0\sim G$, then
for every $k\ge0$,
$\corr(\sigma^2_k,\sigma^2_{k+1})=\corr(\|\bm\theta_k\|_2^2,\|\bm\theta_{k+1}\|_2^2)=p/(n+p-2)$.
\end{lem}



Although Lemma~\ref{lem:autocorr-rr} asserts that the \emph{norms} of the $\bm\theta_k$ chain are highly dependent, 
their \emph{directions} are independent.  This fact is established
in
the following lemma.

\begin{lem}
\label{lem:directions-iid}
Consider the standard Bayesian regression Gibbs sampler in~(\ref{gibbs-rr}).
If $\sigma^2_0\sim G$, then
the vectors $\bm\theta_k/\|\bm\theta_k\|_2$ are independent for all $k\ge1$.
\end{lem}

The
behavior of the
standard Bayesian regression Gibbs sampler
in~(\ref{gibbs-rr}) as
described by Lemmas~\ref{lem:autocorr-rr}~and~\ref{lem:directions-iid} can be
interpreted geometrically.
For any $t>0$, the set $\{\bm\theta\in\reals^p:\|\bm\theta\|_2^2=t\}$ defines a hypersphere in $\bm\theta$-space, which corresponds to a hyperellipsoid in $\bm\beta$-space.
As discussed previously, the value of $\|\bm\theta_k\|_2^2$ is very highly dependent on the value of $\|\bm\theta_{k-1}\|_2^2$ when $p\gg n$.
Hence, in the $p\gg n$ regime, $\bm\theta_k$ is likely to fall on a hypersphere very close to the hypersphere on which $\bm\theta_{k-1}$ falls.
It then follows
that
$\bm\beta_k$ and $\bm\beta_{k-1}$ are also likely to fall on hyperellipsoids that are very close together.
Note that the center of the hyperspheres in $\bm\theta$-space corresponds to the posterior mean~$\tilde{\bm\beta}$ in $\bm\beta$-space.
Thus, the $\bm\beta_k$ chain has difficulty moving to points ``closer to'' or ``farther from'' the posterior mean~$\tilde{\bm\beta}$ as measured by the Mahalanobis distance
$\|\bm\theta\|_2=\|\bm A^{1/2}(\bm\beta-\tilde{\bm\beta})\|_2$.
This behavior is illustrated in Figure~\ref{fig:shell}.  (It should be noted, however, that this behavior only arises when $p$ is large, so an illustration with $p=2$ should be interpreted merely as a conceptual representation of the behavior in question.)
Meanwhile, the behavior of the marginal $\sigma^2_k$ chain as described by Lemma~\ref{lem:autocorr-rr} is somewhat simpler.
The $\sigma^2_k$ simply exhibits a high autocorrelation, i.e., it has difficulty moving at all.


\begin{figure}[htbp]
\floatbox[{\capbeside\thisfloatsetup{capbesideposition={right,center},capbesidewidth=0.7213\textwidth}}]{figure}[\FBwidth]{\caption{%
Geometric representation of the behavior of the standard Bayesian regression Gibbs sampler in~(\ref{gibbs-rr}) in terms of $\bm\theta=\bm A^{1/2}(\bm\beta-\tilde{\bm\beta})$.  If the point on the solid circle represents the value of $\bm\theta_k$, then the value of $\bm\theta_{k+1}$ falls with high probability in the shell of values (region between dashed circles) where $\|\bm\theta_{k+1}\|_2$ is approximately equal to $\|\bm\theta_k\|_2$ (solid circle).
}\label{fig:shell}}
{\includegraphics[scale=0.42]{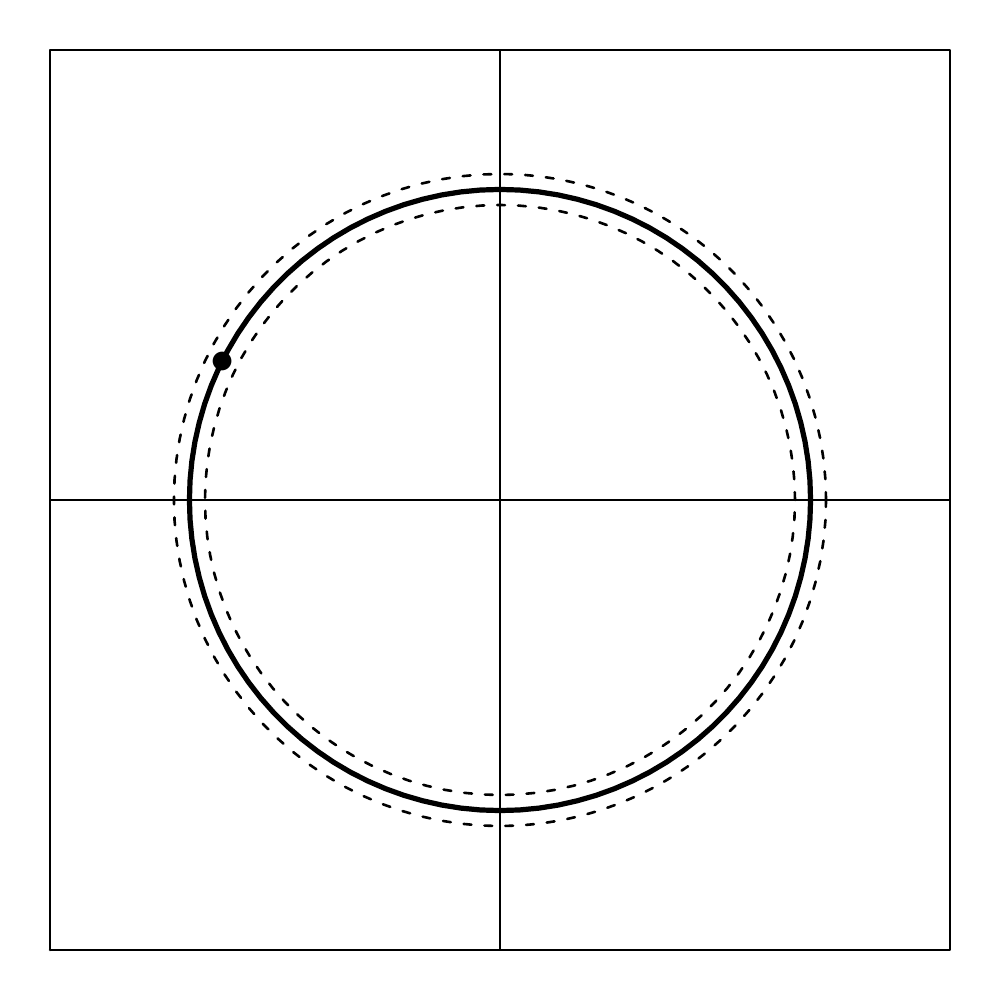}}
\end{figure}

In a practical sense, it is important to understand the manner in which
the convergence problems discussed in the previous paragraph
would affect inference based on Gibbs samples that have been ostensibly (but not actually) drawn from the approximate posterior.
First, note that the aforementioned autocorrelation phenomenon
occurs between $\sigma^2$ and the \emph{norm} of~$\bm\theta$, not the \emph{direction} of~$\bm\theta$.
Thus, even if the chain mixes slowly, the average of the $\bm\theta_k$ iterates should be close to the origin in $\bm\theta$-space.
It follows that the average of the $\bm\beta_k$ iterates should be close to the posterior mean~$\tilde{\bm\beta}$ in $\bm\beta$-space.
Thus, even when $p/(n+p-2)$ is
close to~$1$,
the chain can still yield a good approximation of the posterior mean~$\tilde{\bm\beta}$ of the regression coefficients~$\bm\beta$.
However, there may be substantial error when using the Gibbs sampling output to approximate either the posterior variance of $\bm\beta$ or the posterior mean of~$\sigma^2$.
It is indeed possible for the $\bm\beta_k$ iterates to be distributed
closer
to the posterior mean than they should
be,
in which case approximate credible intervals based on the Gibbs sampling output will be too narrow.
On the other
hand, if
the $\bm\beta_k$ iterates are distributed too far from the posterior mean, then the approximate credible intervals will be too wide.
Thus, uncertainty quantification, one of the fundamental advantages of posterior inference, may be seriously compromised.

\subsection{%
Convergence Diagnostics%
}
\label{subsec:diagnostics}

In practice, MCMC convergence behavior is often assessed through
convergence diagnostics.
These methods can be useful for identifying various kinds of convergence problems in some settings,
though it is well understood that they
do not establish convergence in any rigorous sense.

Consideration of convergence diagnostics is
especially important in high-dimensional settings because the types of convergence problems described in
this subsection and the previous subsection
may not be detectable if only certain convergence diagnostics are considered.
More specifically, note that the primary parameter of interest
in the regression model
is~$\bm\beta$.
As discussed in
Subsection~\ref{subsec:further}
and illustrated in Figure~\ref{fig:shell},
the draws of the $\bm\beta_k$ iterates in the standard Bayesian regression Gibbs sampler can be highly dependent
in terms of their respective \emph{distances}
from their distribution's center, but their \emph{directions} from the center are independent and identically distributed.
Then any convergence diagnostic that focuses on plotting, testing, or otherwise analyzing only linear functions of the~$\bm\beta_k$ is likely to fail to identify any problem.
For example, a trace plot of the marginal chain for any component of~$\bm\beta_k$ will likely appear to have converged and to be uncorrelated.
Similarly, the Geweke diagnostic \citep{geweke1992}, which compares the means of the iterates from earlier versus later portions of the chain, is also likely to fail to detect any problem.

On the other hand,
the above
high-dimensional convergence problems
can be
detected by certain convergence diagnostics
when
applied to certain parameters.  For instance, in standard Bayesian regression, a trace plot of the marginal chain for the nuisance parameter (i.e.,~$\sigma^2_k$) will indeed reveal if the chain is slow to
converge or is
highly autocorrelated.  Slow convergence can similarly be detected by the Geweke diagnostic as applied to the nuisance parameter.
(These same diagnostics can also reveal the problem if they are applied to quadratic, rather than linear, functions of~$\bm\beta_k$
since the behavior of such quadratic functions of~$\bm\beta_k$ may indeed be similar to that of~$\sigma^2_k$, noting the result of Lemma~\ref{lem:autocorr-rr}
).
Hence,
the essence of the message
from the above analysis
is that
it is important to examine convergence diagnostics for \emph{all} parameters, not merely the parameter of interest, if high-dimensional convergence problems are to be detected.
Moreover,
certain
other convergence diagnostics exist that may be more readily able to detect problematic phenomena when
they do occur.
For the regression setting,
diagnostics based on the relative variability within and between various $\bm\beta_k$ chains may indeed be useful.
Examples in this regard include the
Gelman--Rubin
diagnostic \citep{gelman1992}.
Such an approach may not be feasible under the limited computational budget that is often present in high-dimensional settings.

\subsection{%
Convergence Rates for Graphical Models
}
\label{subsec:graphical}

The results obtained for the Gibbs sampler
for
the standard Bayesian regression
framework
can also be applied in the context of a
Bayesian analysis of a
class of Gaussian graphical models.
We first define some notation.
Let $\mathcal G=(V,E)$ be a directed acyclic graph (DAG) with vertex set $V=\{1,\ldots,m\}$ and edge set
$E\subseteq V\times V$.
Assume that
$i>j$
for all $(i,j)\in E$, i.e., assume that $\mathcal G$ is parent-ordered.  Let
\bas
\pa(j)=\{i\in V:{
(i,j)\in E
}
\},\qquad
\fa(j)=\pa(j)\cup\{j\},\qquad
\eas
denote the parents and family of vertex~$j$.
Let $\delta_j=|\pa(j)|$ denote the cardinality of the set of parents of vertex~$j$, which we shall call the degree of vertex~$j$.
Now consider
the Gaussian DAG model in Cholesky form
\basn
\bm X_1,\ldots,\bm X_n\mid\bm D,\bm L&\sim\iid N_m\left[\bm0_m,\,\left(\bm L^{-1}\right)^T\bm D\bm L^{-1}\right],
\label{dag-cholesky}
\easn
where $\bm D=\diag(\sigma^2_{1\mid\pa(1)},\ldots,\sigma^2_{m\mid\pa(m)})$, and where the elements of $\bm L$ are
\bas
L_{ij}=\begin{cases}1&\text{ if }i=j,\\-\beta_{i,j}&\text{ if }i\in\pa(j),\\0&\text{ if }i\notin\fa(j).\end{cases}
\eas
Now suppose we take the prior on $\bm D$ and $\bm L$ to be
$(\bm D,\bm L)\sim\pi_{\bm U,\bm\alpha}$,
that is, the DAG-Wishart prior as defined by
\citet{bendavid2015}.
Combining this prior with the DAG model in~(\ref{dag-cholesky}) yields the \emph{Bayesian DAG framework}.
The posterior distribution of $(\bm D,\bm L)$ then factorizes as
\bas
\pi(\bm D,\bm L\mid\bm X_1,\ldots,\bm X_n)=\prod_{j=1}^m\pi(D_{jj},\bm L_{\pa(j),j}\mid\bm X_1,\ldots,\bm X_n),
\eas
i.e., the posterior distributions of $(D_{jj},\bm L_{\pa(j),j})$ are
mutually
independent for each $j\in\{1,\ldots,m\}$
\citep{bendavid2015}.
Then we can execute separate Gibbs samplers for each of these posterior distributions for each $j\in\{1,\ldots,m\}$ and combine them to yield samples from the overall joint posterior of~$(\bm D,\bm L)$.

To state the form of these Gibbs samplers, we first define an additional item of notation.
For any $m\times m$ matrix~$\bm H$ and any two index subsets $A,B\subseteq\{1,\ldots,m\}$, we write $\bm H_{A,B}$ to denote the submatrix of $\bm H$ formed by retaining the $j$th row if and only if $j\in A$ and the $k$th column if and only if $k\in B$.  (Note that if $A$ or $B$ is a singleton set, i.e., $A=\{a\}$ or $B=\{b\}$, then we will write simply $\bm H_{a,B}$, $\bm H_{A,b}$, or $H_{ab}$.)

Now suppose that
for each $j\in\{1,\ldots,m\}$, we set an initial value $D_{jj;0}>0$.  Then a Gibbs sampler for drawing from the posterior
of~$(D_{jj},\bm L_{\pa(j),j})$
takes the form
\basn
\bm L_{\pa(j),j;k}&=\bm\mu_j+\sqrt{D_{jj;\,k-1}}
\bm W_j^{-1/2}
\bm Z_{j,k},&&\text{ where }
\bm Z_{j,k}\sim N_{\delta_j}(\bm0_{\delta_j},\bm I_{\delta_j}),\notag\\
D_{jj;k}&=\frac1{V_{j,k}}
\left[\left\|\bm W_j^{1/2}\left(\bm L_{\pa(j),j;k}-\bm\mu_j\right)\right\|_2^2+\bm C_{j}\right]
&&\text{ where }V_{j,k}\sim\chi^2_{n+\alpha_j-2},\label{gibbs-dag}
\easn
where $\bm W_j=\bm U_{\pa(j),\pa(j)}+n\bm S_{\pa(j),\pa(j)}$, 
$\bm\mu_j=\bm W_j^{-1}(\bm U_{\pa(j),j}+n\bm S_{\pa(j),j})$, and
$\bm C_j=\bm U_{j,j}+n\bm S_{j,j}-\bm\mu_j^T\bm W_j\bm\mu_j$, and where all of the $\bm Z_{j,k}$ and $V_{j,k}$ are independent.
\citep[See][for the form of the relevant conditional distributions.]{bendavid2015}
Observe
that if we take $\alpha_j=\delta_j+2$, then this Gibbs sampler has the same form as the standard Bayesian regression Gibbs sampler in~(\ref{gibbs-rr}) with
$p=\delta_j$.
We can therefore use our convergence results for standard Bayesian regression to obtain convergence results in the Bayesian DAG
framework.
To state and prove the
said
result, let $\Xi_{j,k}(D_{jj;0})$ denote the distribution of~$(D_{jj;k},\bm L_{
\pa(j),j
;k})$ for the $j$th chain initialized at~$D_{jj;0}$, and let $\Xi_j$ denote the corresponding stationary distribution
for $j\in\{1,\ldots,m\}$.
Then let $\Xi_k(\bm D_0)$ denote the distribution of $(\bm D_k,\bm L_k)$ for the overall joint Gibbs sampler, and let $\Xi$ denote the corresponding stationary distribution.
Also let $\delta_{\max}=\max_{1\le j\le m}\delta_j$.
The following result now gives sharp bounds for the convergence rate of the DAG-Wishart Gibbs sampler.

\begin{thm}
\label{thm:tv-dag}
For the Bayesian DAG Gibbs sampler in~(\ref{gibbs-dag}),
there exist
$0\le\tilde M_1\le\tilde M_2$
such that
\bas
{
\tilde M_1
}
\left(\frac{\delta_{\max}}{n+\delta_{\max}-2}\right)^k
\le d_{
\tv
}\left[\Xi_k(\bm D_0),\,\Xi\right]\le
{
\tilde M_2
}
\left(\frac{\delta_{\max}}{n+\delta_{\max}-2}\right)^k
\eas
for all sufficiently large~$k$.
\end{thm}

Thus, the geometric rate constant of the Bayesian DAG Gibbs sampler in~(\ref{gibbs-dag}) is bounded away from~$1$ as $m$ and $n$ tend to infinity if and only if $\delta_{\max}=O(n)$, i.e., if and only if the
maximum degree of any vertex
grows no faster than the sample size.
Thus, the convergence complexity of the Bayesian DAG Gibbs sampler is closely related to its
sparsity.
This result provides yet another motivation for desiring sparsity in modern high-dimensional settings.

\section{%
Bayesian Model Selection
}
\label{sec:lasso}


We now turn our attention to Gibbs samplers for
Bayesian model selection.
Important contemporary cases include the Bayesian lasso of \citet{park2008}
and the spike-and-slab prior of \citet{mitchell1988}.
The form of the priors
we consider below
is general enough to accommodate other
``regularized''
Bayesian approaches to regression
as well.
They are also easily extended to model selection in other statistical models.


The
Bayesian analysis of the standard regression model in~(\ref{model-rr})
can be generalized by replacing the prior on $\bm\beta\mid\sigma^2$ with a scale mixture of normal distributions, i.e., by taking
\begin{align}
\bm Y\mid\bm\beta,\sigma^2&\sim N_n(\bm X\bm\beta,\,\sigma^2\bm I_n),\notag\\
\bm\beta\mid\sigma^2,\bm\tau&\sim N_{p}(\bm0_{p},\,\sigma^2\bm D_{\bm\tau}),\label{model-lasso}\\
\pi(\sigma^2)&\propto1/\sigma^2,\notag\\
\bm\tau&\sim\pi(\bm\tau),\notag
\end{align}
where $\bm\tau$ is a $p$-dimensional vector of positive hyperparameters and $\bm D_{\bm\tau}=\diag(\tau_1,\ldots,\tau_p)$.
A variety of priors for $\bm\beta\mid\sigma^2$ can be represented by
the hierarchical construction in~(\ref{model-lasso}) above,
as
will be
discussed in
Subsection~\ref{subsec:special-cases}
below.
We will use the term \emph{Bayesian model selection framework} to refer in general to the model and priors in~(\ref{model-lasso}) above.

Now suppose that we can sample from the conditional posterior $\pi(\bm\tau\mid\bm\beta,\sigma^2,\bm Y)$ for all $\bm\beta\in\reals^p$ and all $\sigma^2>0$,
as is often the case.  (See Subsection~\ref{subsec:special-cases} for examples.)
Then a Gibbs sampler to draw from the joint posterior
under~(\ref{model-lasso})
may be constructed by taking initial values $\bm\beta_0\in\reals^p$ and $\sigma^2_0>0$ and then drawing (for every $k\ge1$)
\basn
\left.\bm\tau_k\given\bm\beta_{k-1},\sigma^2_{k-1},\bm Y\right.&\sim\pi\left(\bm\tau\given\bm\beta=\bm\beta_{k-1},\,\sigma^2=\sigma^2_{k-1},\,\bm Y\right),\notag\\
\bm\beta_k\mid\sigma^2_{k-1},\bm\tau_k,\bm Y&\sim N_p\left(\tilde{\bm\beta}_{\bm\tau_k},\,\sigma^2_{k-1}\bm A_{\bm\tau_k}^{-1}\right),\label{gibbs-lasso}\\
\sigma^2_k\mid\bm\beta_k,\bm\tau_k,\bm Y&\sim\inversegamma\left[\frac{n+p}2,\,\frac{\left\|\bm A_{\bm\tau_k}^{1/2}\left(\bm\beta_k-\tilde{\bm\beta}_{\bm\tau_k}\right)\right\|_2^2+C_{\bm\tau_k}}2\right],\notag
\easn
where $\bm A_{\bm\tau}=\bm X^T\bm X+\bm D_{\bm\tau}^{-1}$ (which is positive-definite), $\tilde{\bm\beta}_{\bm\tau}=\bm A_{\bm\tau}^{-1}\bm X^T\bm Y$, and  $C_{\bm\tau}=\bm Y^T(\bm I_n-\bm X\bm A_{\bm\tau}^{-1}\bm X^T)\bm Y$.

\subsection{%
Special Cases for Model Selection: Bayesian Lasso,
Bayesian Elastic Net,
\& Spike-and-Slab
}
\label{subsec:special-cases}

Suppose that $\tau_1,\ldots,\tau_p$ are assigned independent $\expd(\lambda/2)$ priors,
where $\lambda>0$.
Then the resulting marginal prior on the regression coefficients $\bm\beta\mid\sigma^2$ is a product of Laplacian (double exponential) distributions:
\bas
\pi(\bm\beta\mid\sigma^2)=\prod_{j=1}^p\frac12\sqrt{\frac{\lambda}{\sigma^2}}\exp\left(-\sqrt{\frac\lambda{\sigma^2}}\,\left|\beta_j\right|\right).
\eas
The conditional posterior of $\bm\tau$ is then
\bas
\left.\frac1{\tau_j}\given\bm\beta,\sigma^2,\bm Y\right.\sim\ind\inversegaussian\left(\sqrt{\frac{\lambda\sigma^2}{\beta_j^2}},\,\lambda\right).
\eas
This particular hierarchical representation is the original formulation of what is typically called the Bayesian lasso~\citep{park2008}.

Suppose instead that $\tau_1,\ldots,\tau_p$ are assigned the prior
\bas
\pi(\tau_1,\ldots,\tau_p)=\prod_{j=1}^p\frac{\lambda_1}{2(1-\lambda_2\tau_j)^2}\,\exp\left[-\frac{\lambda_1\tau_j}{2(1-\lambda_2\tau_j)}\right]\,\bm1_{(0,\,1/\lambda_2)}(\tau_j),
\eas
where $\lambda_1,\lambda_2>0$.
Then the resulting marginal prior on the regression coefficients $\beta\mid\sigma^2$ has the form
\bas
\pi(\bm\beta\mid\sigma^2)\propto\prod_{j=1}^p\exp\left(-\sqrt{\frac{\lambda_1}{\sigma^2}}\,\left|\beta_j\right|-\frac{\lambda_2}{2\sigma^2}\sum_{j=1}^p\beta_j^2\right).
\eas
The conditional posterior of~$\bm\tau$ is then
\bas
\left.\left(\frac1{\tau_j}-\lambda_2\right)\given\bm\beta,\sigma^2,\bm Y\right.\sim\ind\inversegaussian\left(\sqrt{\frac{\lambda_1\sigma^2}{\beta_j^2}},\,\lambda_1\right).
\eas
This particular hierarchical representation is known as the Bayesian elastic net \citep{li2010,kyung2010}.

As another example, suppose
instead that the priors on $\tau_1,\ldots,\tau_p$ are again taken to be independent, but for all~$j\in\{1,\ldots,p\}$, take $P(\tau_j=\kappa_j\zeta_j)=w_j=1-P({
\tau_j=\zeta_j
})$, where $\zeta_j>0$ is small, $\kappa_j>0$ is large, and $0<w_j<1$.
This is a slight variant of the prior proposed by \citet{george1993}
to approximate the spike-and-slab prior of \citet{mitchell1988}.
The prior on~$\bm\beta$ is specified conditionally on~$\sigma^2$,
with $\var(\bm\beta\mid\sigma^2,\bm\tau)\propto\sigma^2$.
Then $\tau_j\mid\bm\beta,\sigma^2,\bm Y$ are conditionally independent a~posteriori with
\bas
P\left(\tau_j=\kappa_j\zeta_j\given\bm\beta,\sigma^2,\bm Y\right)
&=1-P\left(\tau_j=\zeta_j\given\bm\beta,\sigma^2,\bm Y\right)\\
&=\left\{1+\displaystyle\frac{(1-w_j)\sqrt{\kappa_j}}{w_j}\exp\left[-\frac{\beta_j^2}{2\sigma^2}\left(\frac{\kappa_j-1}{\kappa_j\zeta_j}\right)\right]\right\}^{-1}
\eas
by straightforward modification of the results of \citet{george1993}.

\subsection{Convergence Properties}

The Gibbs sampler in~(\ref{gibbs-lasso}) for Bayesian model selection is easily executed in practice.
In comparison to standard regression,
the additional step in the Gibbs sampling cycle makes it less tractable in the context of analyzing convergence rates in various $n$ and~$p$ regimes.
Nevertheless, geometric ergodicity has
been obtained for important special cases in
Subsection~\ref{subsec:special-cases}.
The Gibbs sampler for the modified spike-and-slab model was shown by \citet{diebolt1990} to be geometrically ergodic, but without quantitative bounds on the geometric convergence rate.
For the Bayesian lasso Gibbs sampler,
\citet{khare2013} used the method of \citet{rosenthal1995} to establish geometric ergodicity and
derive
a quantitative bound~$\tilde r$ on the geometric convergence rate~$r$.
In Example~\ref{ex:lasso} and Lemma~\ref{lem:ex-lasso}, we showed that this bound~$\tilde r$ tends to~$1$
exponentially fast
as either $p$ or~$n$ tends to infinity.
As this result is an upper bound based on \citeauthor{rosenthal1995}'s method, it is not clear whether it is
sharp in high-dimensional regimes.
Thus, it does not answer the question of the chain's actual convergence rate or the rate at which it may tend to~$1$ as $n$ or $p$ grows without bound.
To address this question for
Gibbs samplers for
the Bayesian lasso, Bayesian elastic net, spike-and-slab priors, and other Bayesian model selection frameworks,
we now provide an autocorrelation result that is similar to that of
Lemma~\ref{lem:autocorr-rr}.

\begin{thm}\label{thm:autocorr-lasso}
Consider the Gibbs sampler in~(\ref{gibbs-lasso}) for Bayesian model selection.
Suppose
that $(\bm\beta_k,\sigma^2_k)\sim\pi(\bm\beta,\sigma^2\mid\bm Y)$.  Then
\bas
\corr(\sigma^2_k,\sigma^2_{k+1})\ge\frac p{n+p-2}
\left[1-\frac{\bm Y^T\bm Y}{p\,\sqrt{\var\left(\sigma^2\mid\bm Y\right)}}\right].
\eas
\end{thm}

Suppose we
make the mild assumption
that $\|\bm Y_n\|_2^2=O(n)$,
and suppose also that $\var(\sigma^2\mid\bm Y)=O(1/n)$, as is commonly the case.
Then it is clear that
the lower bound in Theorem~\ref{thm:autocorr-lasso} tends to~$1$ in the limit as $p_n/n^{3/2}\to\infty$.
Hence, the
MCMC convergence
problems seen in Section~\ref{sec:rr}
for the standard regression model
occur once again
for the Gibbs sampler for Bayesian model selection
if $p_n$ grows too fast relative to~$n$.
This phenomenon
is of course concerning as the lasso is specifically designed for
high-dimensional settings
where $p\gg n$.
The sharpness of the
above lower bound for the autocorrelation
is further investigated numerically in
Subsection~\ref{subsec:numerical-lasso}
below.

\subsection{%
Numerical Results for Bayesian Model Selection
}
\label{subsec:numerical-lasso}

Theorem~\ref{thm:autocorr-lasso}
provides
a lower bound for the autocorrelation between successive iterates of the $\sigma^2_k$ chain
of the
Gibbs sampler in~(\ref{gibbs-lasso}) for Bayesian model selection.
It is not immediately clear whether this bound is sharp,
so it is also instructive to use numerical
approaches
to understand the high-dimensional convergence behavior of Gibbs samplers of
this form.
The left side of Figure~\ref{fig:autocorr} plots the autocorrelation in the $\sigma^2_k$ chain versus $p/(n+p-2)$ for various values of $n\in\{10,30,100\}$ and $p\in\{10,30,100\}$ as observed from runs of the
Gibbs sampler for the Bayesian lasso.
The center and right side of Figure~\ref{fig:autocorr} are similar plots for the Gibbs samplers for
the Bayesian elastic net and the spike-and-slab prior (respectively).
(The exact details of these runs can be found
in Supplemental Section~\ref{sec:supp-numerical-details}.)
Such plots can be useful tools when sharp theoretical bounds are not available.
The strength of the linear relationship in Figure~\ref{fig:autocorr} strongly suggests that the ratio $p/(n+p-2)$ governs the convergence behavior of the Gibbs samplers for a variety of Bayesian regression approaches that can be written in the form specified by~(\ref{model-lasso}).
Thus, although the theoretical result of Theorem~\ref{thm:autocorr-lasso} is
slightly
less refined than those obtained
for the standard regression model
in Section~\ref{sec:rr}, it is clear from Figure~\ref{fig:autocorr} that these Markov chains exhibit the same convergence complexity as in the standard regression setting.

\begin{figure}[htbp]
\centering
\includegraphics[scale=0.42]{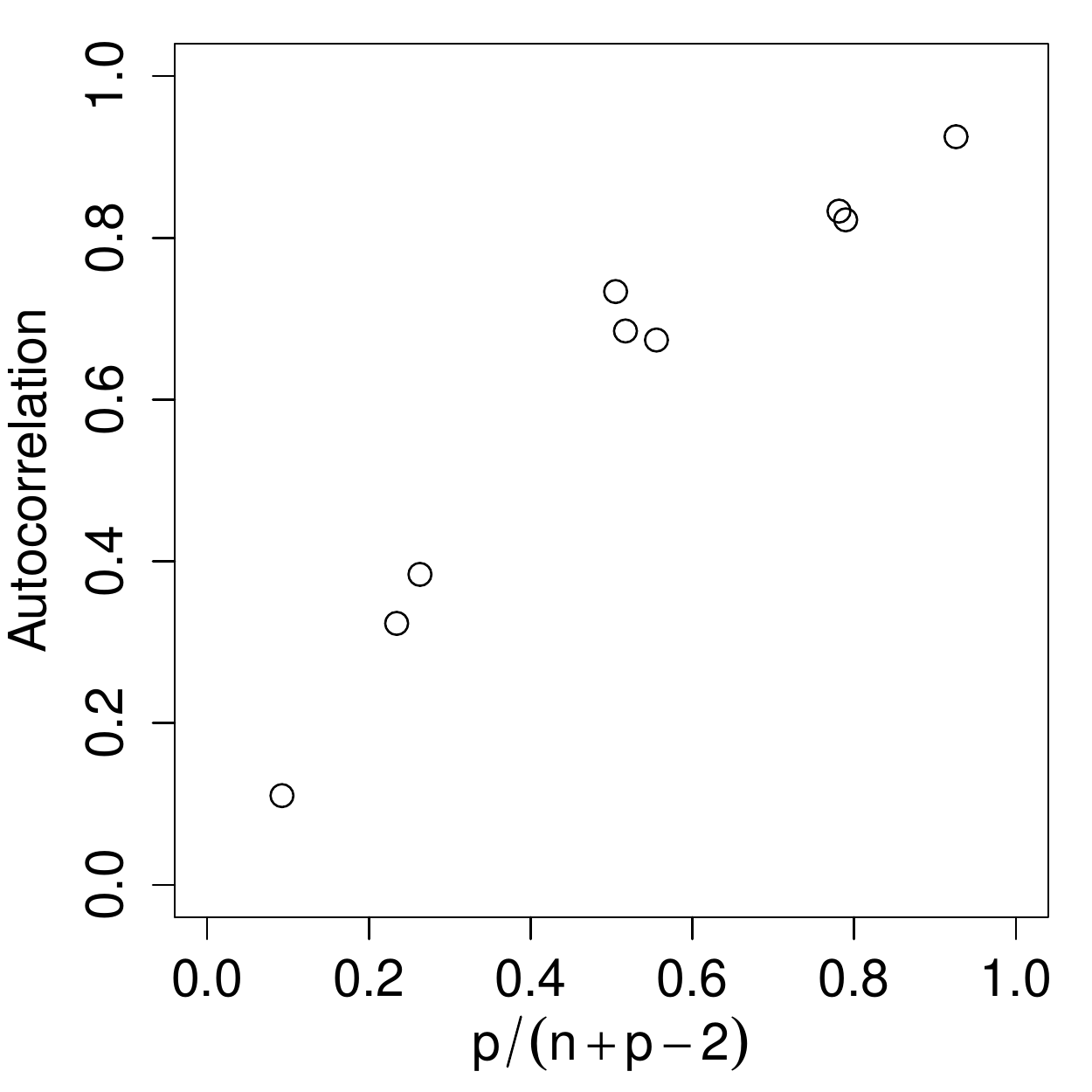}
\includegraphics[scale=0.42]{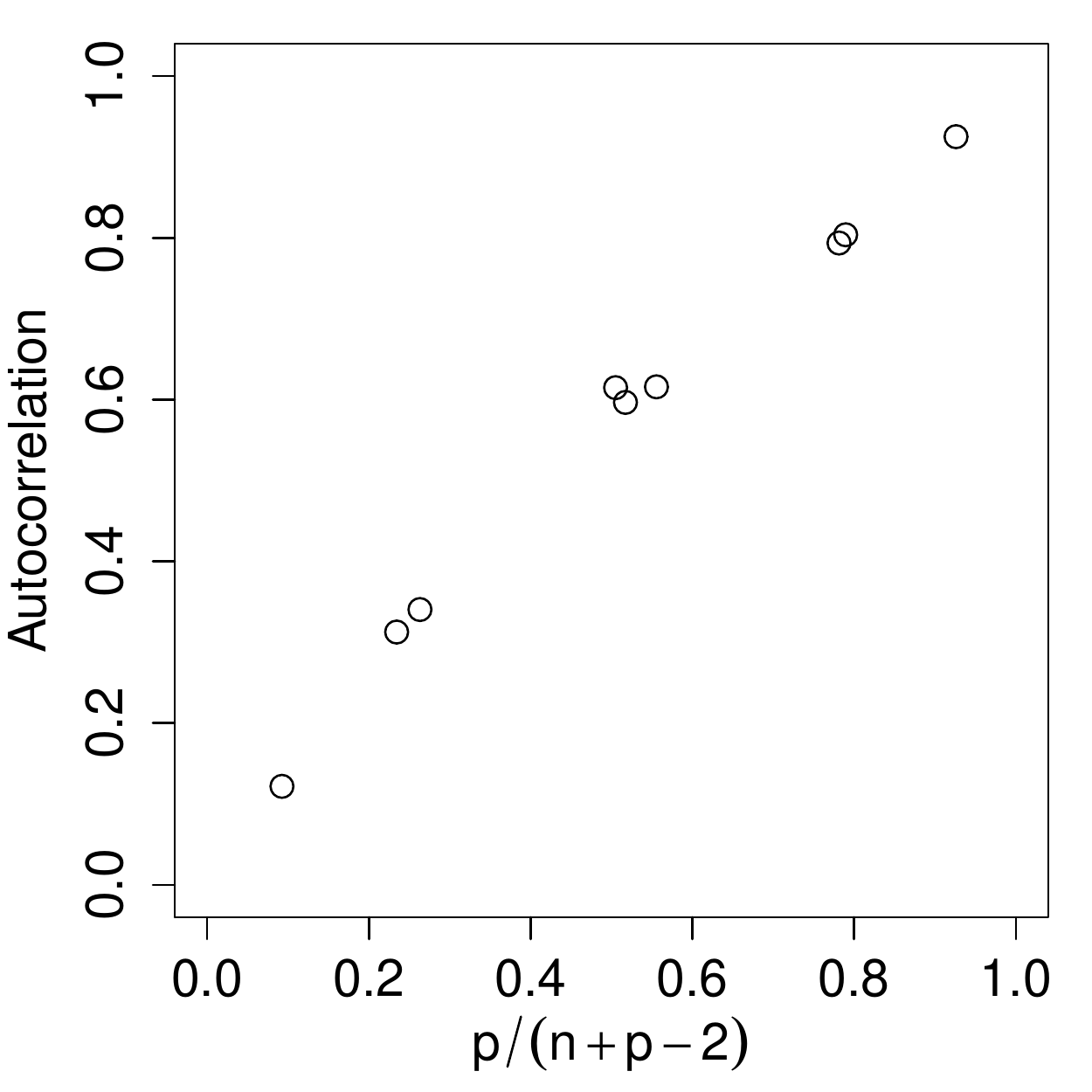}
\includegraphics[scale=0.42]{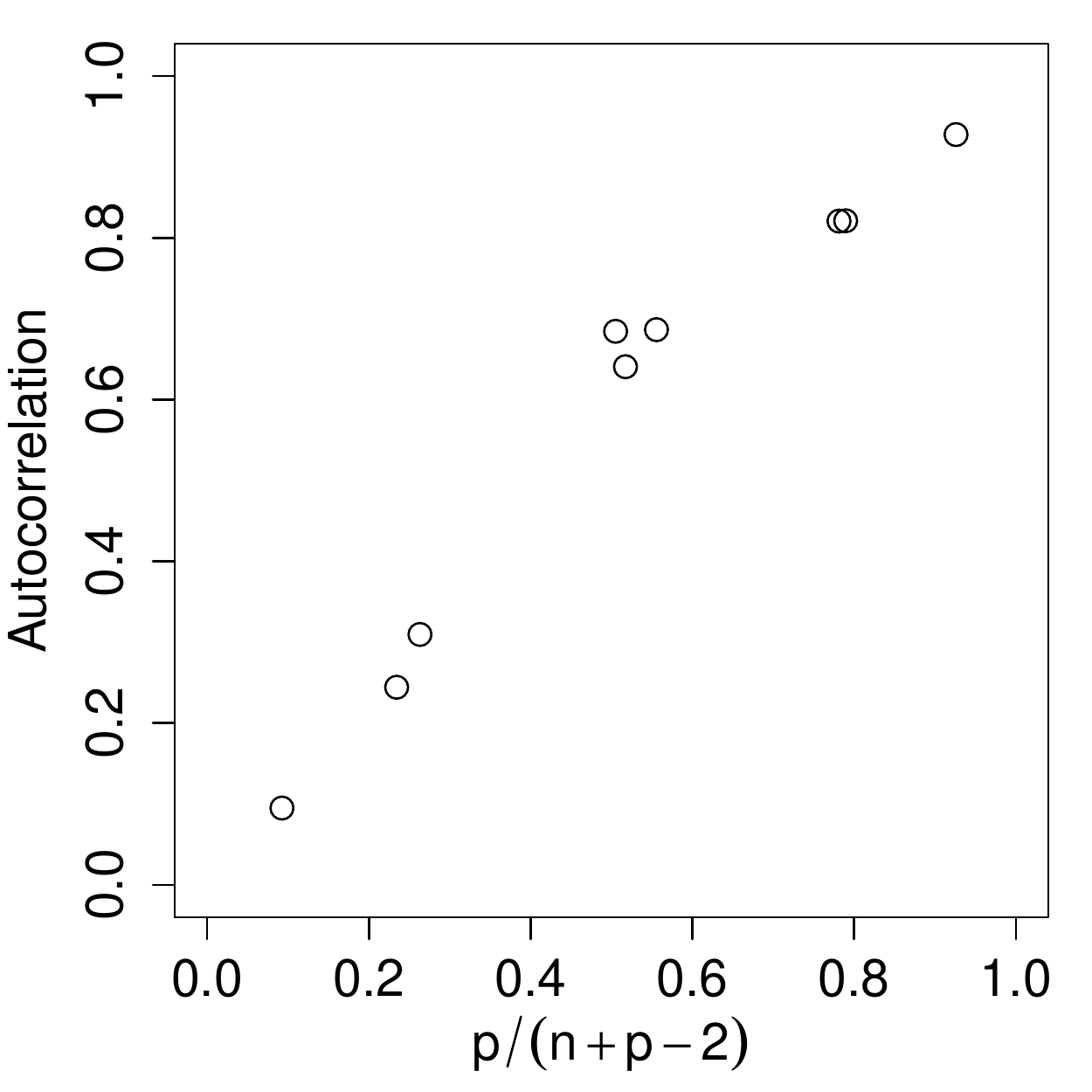}
\vspace*{-0.5em}
\caption{%
Autocorrelation of the $\sigma^2_k$ chain versus $p/(n+p-2)$ for the
Gibbs sampler for the Bayesian lasso (left), Bayesian elastic net (center), and the spike-and-slab prior (right).
See Supplemental Section~\ref{sec:supp-numerical-details}
for details
of the generation of the various numerical
quantities, vectors, and matrices
that were used in the execution of these chains.
}
\label{fig:autocorr}
\end{figure}

In summary, our theoretical and numerical analysis above indicates that regardless of the type or form of regression
(standard regression, lasso, elastic net, or spike-and-slab),
there is a universal geometric convergence rate of the form $r=p/(n+p-2)$.

\section{Multivariate
Location Models
}
\label{sec:mean}

A conceptually simple but centrally important class of models is the
class of
multivariate mean or location
models.
We now consider the convergence complexity of Markov chains associated with
a
Bayesian analysis of such models.
We shall see that though there are some similarities between location models and regression, convergence complexities are vastly different.


Let
$\bm X_1,\ldots,\bm X_n$ be observed data vectors taking values in~$\reals^p$, and
consider the
multivariate mean model
and priors
\begin{align}
\bm X_i\mid\bm\mu,\sigma^2&\sim\iid N_p(\bm\mu,\,\sigma^2\bm I_p),\notag\\
\bm\mu\mid\sigma^2&\sim\iid N_p(\bm0_p,\,
{
\lambda^{-1}
}
\sigma^2\bm I_p),\label{model-mean}\\
\pi(\sigma^2)&\propto1/\sigma^2\;\text{ for all }\sigma^2>0,\notag
\end{align}
where $i\in\{1,\ldots,n\}$ with
$n\ge3$,
and where
$\lambda>0$
is known.
Then a Gibbs sampler to draw from the joint posterior under~(\ref{model-mean}) may be constructed by taking an initial value $\sigma^2_0>0$ and then drawing (for every $k\ge1$)
\basn
\bm\mu_k\mid\sigma^2_{k-1},\bm X_1,\ldots,\bm X_n&\sim N_p\left(
\tilde{\bm\mu},\,\frac{\sigma^2_{k-1}}{n+\lambda}\bm I_p\right),\notag\\
\sigma^2_k\mid\bm\mu_k,\bm X_1,\ldots,\bm X_n&\sim\inversegamma\left[\frac{np+p}2,\,\frac{(n+\lambda)\left\|\bm\mu_k-\tilde{\bm\mu}\right\|_2^2+C}2\right],\label{gibbs-mean}
\easn
where $\tilde{\bm\mu}=(n+\lambda)^{-1}\sum_{i=1}^n\bm X_i$ and $C=\sum_{i=1}^n\|\bm X_i\|_2^2-(1+n^{-1}\lambda)\|\tilde{\bm\mu}\|_2^2$.

\subsection{Convergence Properties}
\label{subsec:cvgc-mean}

The
convergence properties of the
Gibbs sampler in~(\ref{gibbs-mean}) for the multivariate mean model
can be obtained using the results
previously
established in Section~\ref{sec:rr}
for the standard regression Gibbs sampler
in~(\ref{gibbs-rr}).
For every $k\ge0$, let $F_k(\sigma^2_0)$ denote the joint distribution of $(\bm\mu_k,\sigma^2_k)$ for the
Gibbs sampler of the multivariate mean model in~(\ref{gibbs-mean})
started with initial value~$\sigma^2_0$.  Let $F$ denote the stationary distribution of this chain, i.e., the true
joint posterior of $(\bm\mu,\sigma^2)$.
Then we have the following result.

\begin{thm}
\label{thm:tv-mean}
Consider the Gibbs sampler for the multivariate mean model in~(\ref{gibbs-mean}).
Then there exist $0<M_1\le M_2$ such that
\bas
M_1\left(\frac{p}{np+p-2}\right)^k\le d_{\tv}\left[F_k\left(\sigma^2_0\right),\,F\right]\le M_2\left(\frac{p}{np+p-2}\right)^k
\eas
for every $k\ge0$.
\end{thm}

Despite the apparent similarities between the standard regression model and the multivariate location model, it
is clear from
Theorem~\ref{thm:tv-mean}
that
the respective Gibbs samplers display different convergence complexities.
In particular, as $p\to\infty$ with $n$ fixed, the geometric convergence rate of the standard Bayesian regression Gibbs sampler tends to~$1$.
However, the
convergence rate of the
Gibbs sampler for the multivariate mean model
tends to $1/(n+1)$.
Moreover, $r\le 1/(n-1)$ for
any~$p$.
Thus, the convergence rate of the
Gibbs sampler for the multivariate mean model
is bounded away from~$1$ in all $n$ and $p$ regimes.


As was the case in Bayesian regression,
we can again establish sharp results in terms of Wasserstein distance.
These results can be found in
Supplemental Section~\ref{sec:supp-mean}.

\section{Normal Hierarchical Model}
\label{sec:nh}

Hierarchical models are an important class of models that have found widespread applications in many fields.  They have thus become a staple in contemporary Bayesian inference.  Their flexibility and ability to avoid overfitting makes them ideal for modern high-dimensional settings.  Hierarchical models are also ideally suited for Bayesian analysis since they are readily amenable to posterior inference using Gibbs samplers.
To further investigate notions of convergence complexity, we
thus
turn our attention to Markov chains associated with
a
Bayesian analysis of a common type of model:
the normal hierarchical model.
We begin by first considering a simplified version of such a model in which the variance components are known.  We subsequently investigate the unknown-variance version of this 
hierarchical
model as well.


Let
$\bm X_1,\ldots,\bm X_n$ be observed data vectors taking values in~$\reals^p$, and
consider the
following hierarchical model and priors:
\begin{align}
\bm X_i\mid\bm\psi_i&\sim\ind N_p(\bm\psi_i,\sigma^2\bm I_p),\notag\\
\bm\psi_i\mid\bm\mu&\sim\iid N_p(\bm\mu,\tau^2\bm I_p),\label{model-nh}\\
\pi(\bm\mu)&\propto1\;\text{ for all }\bm\mu\in\reals^p,\notag
\end{align}
where $i\in\{1,\ldots,n\}$ and where $\sigma^2>0$ and $\tau^2>0$ are known.
Then a Gibbs sampler to draw from the joint posterior under~(\ref{model-nh}) may be constructed by taking an initial value $\bm\mu_0\in\reals^p$ and then drawing (for every $k\ge1$)
\basn
\bm\psi_{k,i}\mid\bm\mu_{k-1},\bm X_1,\ldots,\bm X_n&\sim\ind N_p\left[
(1-r)\bm X_i+r\bm\mu_{k-1},\;\tau^2r\bm I_p\right],\notag\\
\bm\mu_k\mid\bm\psi_{k,1},\ldots,\bm\psi_{k,n},\bm X_1,\ldots,\bm X_n&\sim N_p\left(\frac1n\sum_{i=1}^n\bm\psi_{k,i},\;\frac{\tau^2}n\bm I_p\right),\label{gibbs-nh}
\easn
for each $i\in\{1,\ldots,n\}$, where $r=\sigma^2/(\sigma^2+\tau^2)$.

\subsection{Convergence Properties}
\label{subsec:cvgc-nh}

We now establish sharp bounds
for the geometric
convergence
rate of the
Gibbs sampler in~(\ref{gibbs-nh}) for the normal hierarchical model.
For
every $k\ge0$, let $H_k(\bm\mu_0)$ denote
the distribution of $\bm\mu_k$ for
the
normal hierarchical model Gibbs sampler in~(\ref{gibbs-nh})
started with initial value~$\bm\mu_0$,
and let $H$ denote
the
true marginal posterior of~$\bm\mu$. 
Then we have the following result.


\begin{thm}
\label{thm:tv-nh}
Consider the Gibbs sampler for the normal hierarchical model in~(\ref{gibbs-nh}).
Then
\bas
\sqrt{\frac n{2(\sigma^2+\tau^2)}}\left\|\bm\mu_0\right\|_2\,r^k
\le
d_{\tv}\left[H_k(\bm\mu_0),H\right]
\le\sqrt{\frac n{\sigma^2+\tau^2}}\left\|\bm\mu_0\right\|_2\,r^k
\eas
for all sufficiently large~$k$,
where $r=\sigma^2/(\sigma^2+\tau^2)$.
\end{thm}

Note in particular that Theorem~\ref{thm:tv-nh} provides
an expression
for the geometric convergence rate~$r$ that does not depend on $n$ or~$p$.
Thus, the
Gibbs sampler specified in~(\ref{gibbs-nh}) for the normal hierarchical model
does not exhibit the same high-dimensional convergence problems that were seen in Sections~\ref{sec:rr}~and~\ref{sec:lasso} for the regression model and its extensions.
More precisely, the geometric convergence rate~$r_{n,p}$ of the Gibbs sampler for the normal hierarchical model is (trivially) bounded away from~$1$.
In this respect, the convergence complexity of the normal hierarchical model is similar to that of the location model in
Section~\ref{sec:mean}.
(Note also that this result shows that the Gibbs sampler converges faster
when the population variance~$\tau^2$ takes larger values.
)

Similarly to the autocorrelation result in Section \ref{sec:rr}, it can be shown that if $\bm\mu_0\sim H$ (i.e., if the chain is stationary), then
\basn
\corr\left(\mu_{k,j},\mu_{k+1,j}\right)=\frac{\sigma^2}{\sigma^2+\tau^2}
\label{autocorr-nh}
\easn
for each $j\in\{1,\ldots,p\}$.
The autocorrelation result in~(\ref{autocorr-nh}) above contrasts with the autocorrelation result for standard regression in Lemma~\ref{lem:autocorr-rr} in the same way that
the convergence rate in Theorem~\ref{thm:tv-nh} constrasts with the convergence rate for standard regression in Theorem~\ref{thm:tv-rr}.

We can once again establish sharp results in terms of Wasserstein distance as well.
These results can be found in
Supplemental Section~\ref{sec:supp-nh}.


\subsection{%
Unknown Variances
\& Convergence Rates
}
\label{subsec:nhuv}

The normal hierarchical model in~(\ref{model-nh}),
in which the variances are known, 
is
simple enough to yield a Gibbs sampler that
permits the derivation of sharp bounds for the geometric convergence rate.
It is also of interest to consider a more complex model in which the variances are unknown.
Thus, suppose we have
\basn
\bm X_i\mid\bm\psi_i,\sigma^2&\sim\ind N_p(\bm\psi_i,\sigma^2\bm I_p),\notag\\
\bm\psi_i\mid\bm\mu,\tau^2&\sim\iid N_p(\bm\mu,\tau^2\bm I_p),\notag\\
\sigma^2&\sim\inversegamma(a_\sigma/2,\,s_\sigma/2),\label{model-nhuv}\\
\tau^2&\sim\inversegamma(a_\tau/2,\,s_\tau/2),\notag\\
\pi(\bm\mu)&\propto1\;\text{ for all }\bm\mu\in\reals^p,\notag
\easn
where $a_\sigma,s_\sigma,a_\tau,s_\tau>0$ are all
known.
The posterior
for the above setup
is
less tractable, and MCMC is indeed required to sample from the posterior.  A Gibbs sampler to draw from
this posterior
takes
initial values $\bm\mu_0\in\reals^p$ and $\sigma^2_0,\tau^2_0>0$ and then draws (for every $k\ge1$)
\basn
\bm\psi_{k,i}\mid\bm\mu_{k-1},\sigma^2_{k-1},\tau^2_{k-1},
\bm X
&\sim\ind N_p\left[
(1-\rho_{k-1})\bm X_i+\rho_{k-1}\bm\mu_{k-1}{
,
}\;\frac{\sigma^2_{k-1}\tau^2_{k-1}}{\sigma^2_{k-1}+\tau^2_{k-1}}\bm I_p\right],\notag\\
\bm\mu_k\mid\bm\psi_{k,1},\ldots,\bm\psi_{k,n},\sigma^2_{k-1},\tau^2_{k-1},
\bm X
&\sim N_p\left(\frac1n\sum_{i=1}^n\bm\psi_{k,i},\;\frac{\tau^2_{k-1}}n\bm I_p\right),\label{gibbs-nhuv}\\
\sigma^2_k\mid\bm\psi_{k,1},\ldots,\bm\psi_{k,n},\bm\mu_k,\tau^2_{k-1},
\bm X
&\sim\inversegamma\left[\frac{a_\sigma+np}2,\,\frac12\left(s_\sigma+\sum_{i=1}^n\left\|\bm X_i-\bm\psi_{k,i}\right\|_2^2\right)\right],\notag\\
\tau^2_k\mid\bm\psi_{k,1},\ldots,\bm\psi_{k,n},\bm\mu_k,\sigma^2_k,
\bm X
&\sim\inversegamma\left[\frac{a_\tau+np}2,\,\frac12\left(s_\tau+\sum_{i=1}^n\left\|\bm\psi_{k,i}-\bm\mu_k\right\|_2^2\right)\right],\notag
\easn
for each $i\in\{1,\ldots,n\}$, where $\rho_{k-1}=\sigma^2_{k-1}/(\sigma^2_{k-1}+\tau^2_{k-1})$
and $\bm X=(\bm X_1,\ldots,\bm X_n)$.

Since sharp theoretical results for
the convergence rate of the above Gibbs sampler are not as readily quantifiable,
we now provide a numerical demonstration of the behavior of these chains in relation to $n$ and~$p$.  Figure~\ref{fig:autocorr-nhuv} plots the autocorrelation in the $\tau^2_k$ chain for various values of $n\in\{10,30,60,100{
,150,210
}\}$ and $p\in\{3,10,30,100{
,300
}\}$
for the Gibbs sampler in~(\ref{gibbs-nhuv}) for the unknown-variance normal hierarchical model.
(The exact details of these runs can be found
in Supplemental Section~\ref{sec:supp-numerical-details}.)
Figure~\ref{fig:autocorr-nhuv} also
plots the same information
as a three-dimensional autocorrelation surface and an autocorrelation contour plot.
These plots, which we call
dimensional autocorrelation function (DACF) plots, depict
the autocorrelation as a function of the sample size~$n$ and the dimension~$p$.
It is clear from Figure~\ref{fig:autocorr-nhuv} that the convergence behavior of the 
chain corresponding to the unknown-variance hierarchical model
is dramatically different
from the known-variance case in high dimensions.
In fact, both large sample sizes and high dimensions affect the autocorrelation adversely.
The
seemingly harmless practice of putting a prior on a difficult-to-specify quantity in fact leads to a chain that converges slowly for
large
$n$ or~$p$.
Thus, in the unknown-variance setting, it appears that a sample-starved high-dimensional hierarchical model enjoys better convergence than a sample-rich high-dimensional hierarchical model.
In this sense the hierarchical model is better suited to ``large~$p$, small~$n$'' applications than
to ``large~$p$, large~$n$'' applications.


\begin{figure}[htbp]
\centering
\includegraphics[scale=0.42]{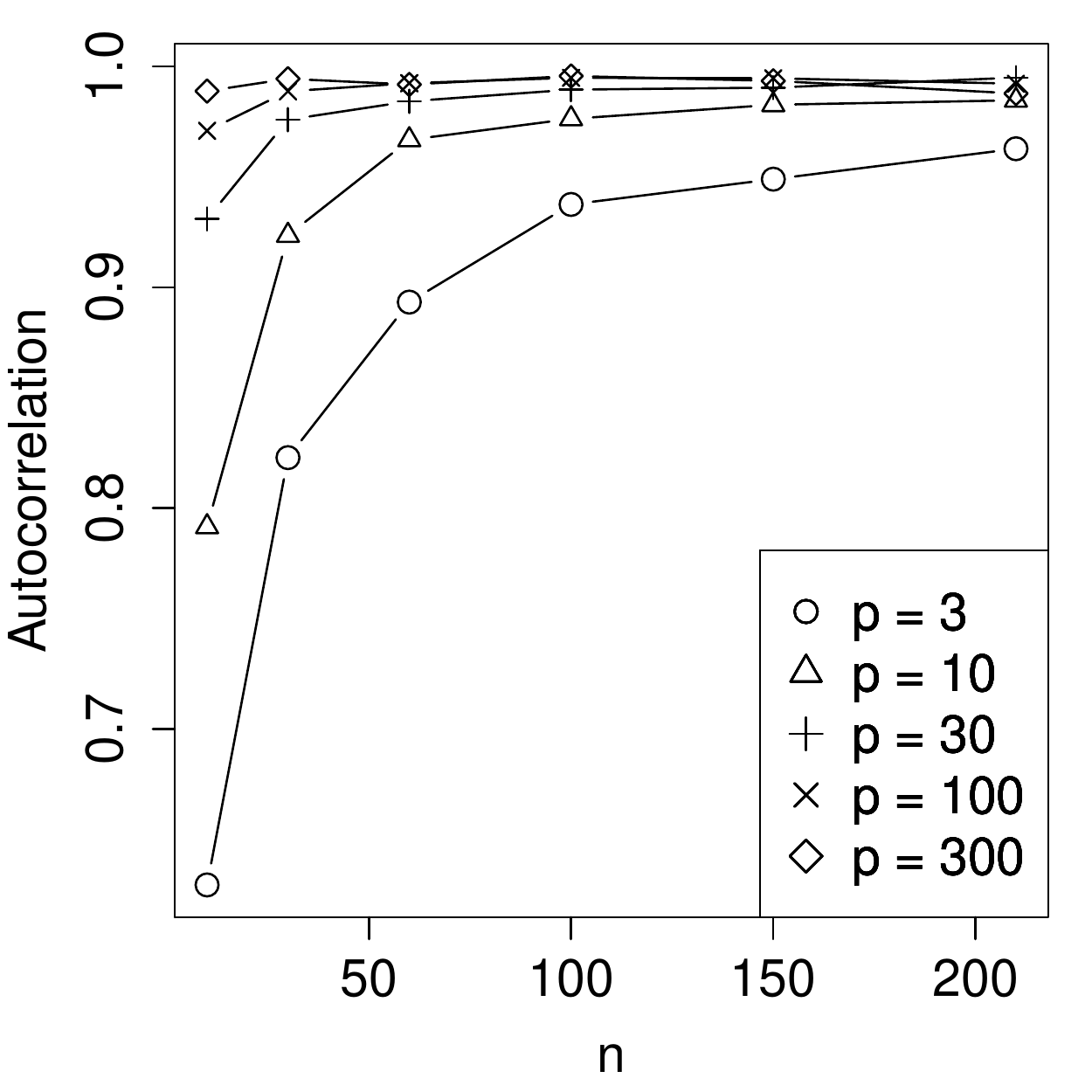}
\includegraphics[scale=0.42]{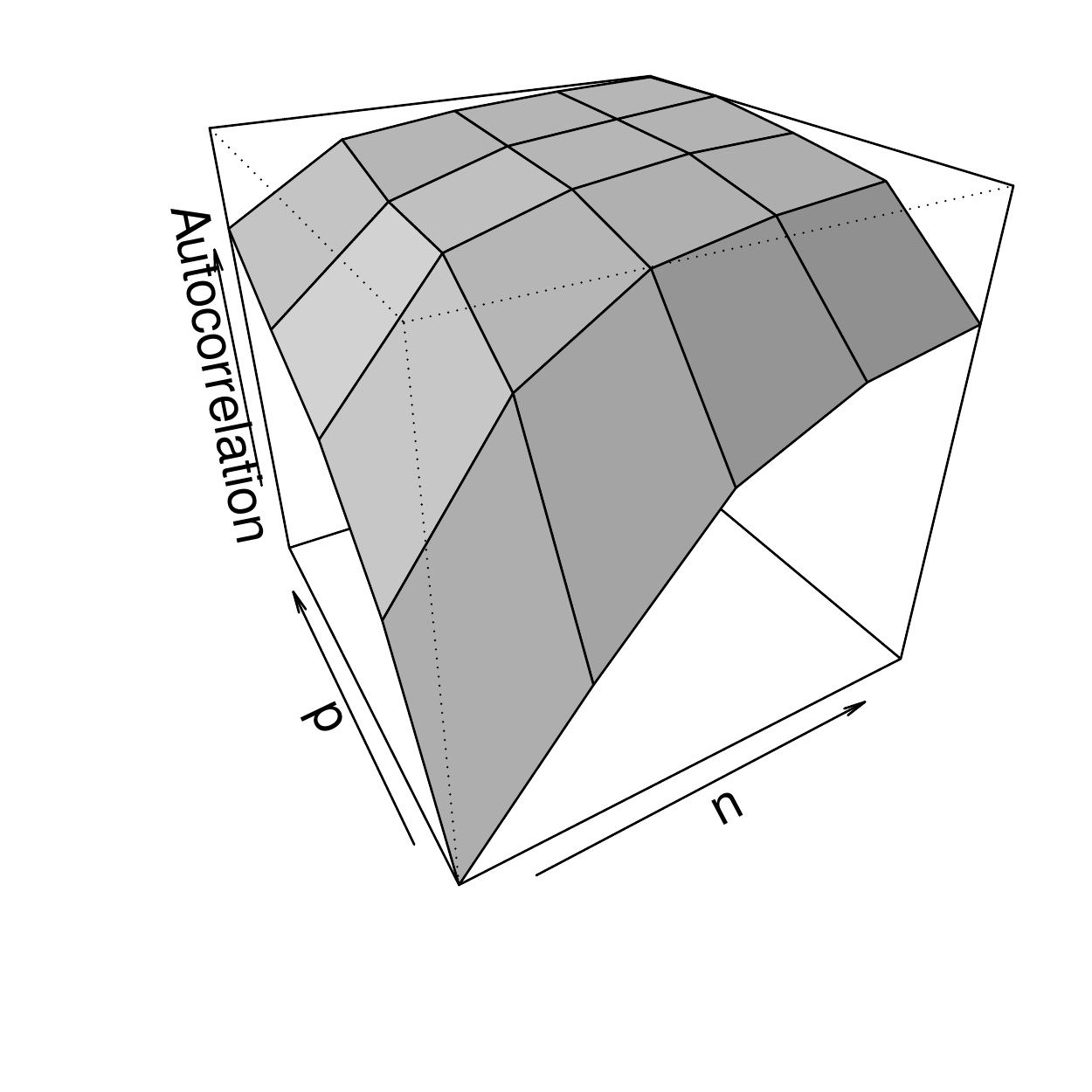}
\includegraphics[scale=0.42]{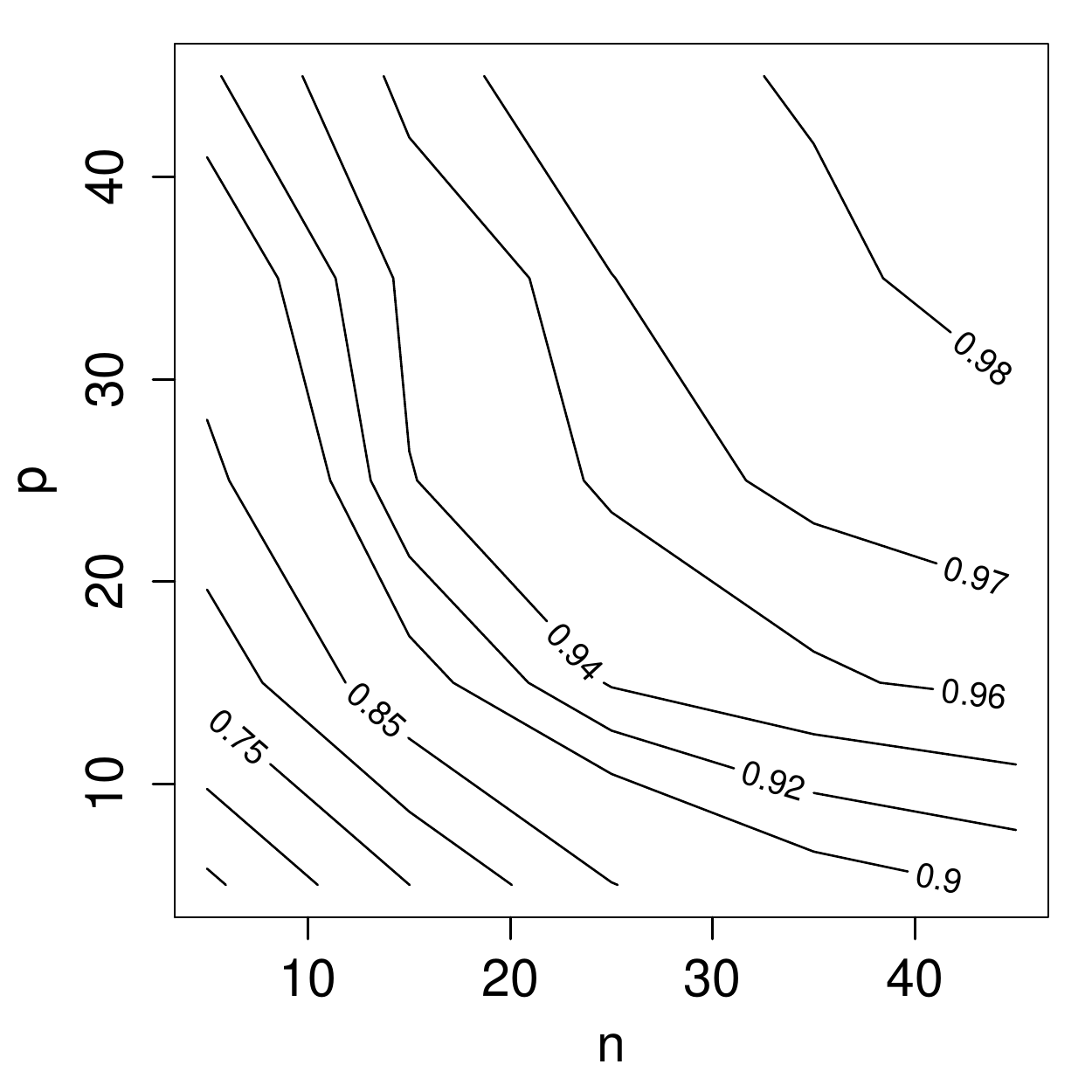}
\vspace*{-0.5em}
\caption{%
Autocorrelation of the $\tau^2_k$ chain versus
the sample size~$n$ (left)
for various values of~$p$.
Dimensional
autocorrelation function (DACF) plots in both surface (center) and contour (right) forms
for the $\tau^2_k$ chain relative to $n$ and~$p$ for the unknown-variance normal hierarchical model.
See Supplemental Section~\ref{sec:supp-numerical-details}
for details
of the generation of the various numerical
quantities, vectors, and matrices
that were used in the execution of these chains.
}
\label{fig:autocorr-nhuv}
\end{figure}


\subsection{%
Unknown Variances \& Bounded Convergence Rates
}
\label{subsec:nhdd}

Empirical Bayesian methods provide
one straightforward
way to obtain bounded convergence rates for the normal hierarchical model with unknown variances.
If the values of $\sigma^2$ and $\tau^2$ are set by an empirical Bayesian approach (e.g., by taking the values that maximize the marginal likelihood), then these values $\hat\sigma^2_{\eb}$ and $\hat\tau^2_{\eb}$ may simply be inserted into the Gibbs sampler for the known-variance model.  It then follows immediately from Theorem~\ref{thm:tv-nh} that the convergence rate is $\hat\sigma^2_{\eb}/(\hat\sigma^2_{\eb}+\hat\tau^2_{\eb})$.

Even when variances are unknown,
we now show that
it is still possible to achieve bounded convergence rates for the normal hierarchical model.
Some insight into a possible solution may be gained by observing that a known-variance approach is simply the limit of an unknown-variance approach as the priors on the variances tend to degeneracy at particular points (the ``known'' values).
More precisely, suppose that we retain the independent
inverse-gamma
priors for $\sigma^2$ and~$\tau^2$ as specified in~(\ref{model-nhuv}), but suppose we take $a_\sigma$, $s_\sigma$, $a_\tau$, and $s_\tau$ to grow proportionally to $np$.
Figure~\ref{fig:nhdd} shows
plots analogous to
Figure~\ref{fig:autocorr-nhuv} in which we have taken $a_\sigma=s_\sigma=a_\tau=s_\tau=np$,
i.e., dimensionally-dependent.
It is clear from
Figure~\ref{fig:nhdd}
that the convergence rate remains bounded away from~$1$ for all $n$ and~$p$.
Thus, the dimensionally-dependent prior for the unknown variances yields dramatically improved convergence complexity.  We will revisit this idea in Subsection~\ref{subsec:dd} for the regression setting as well.

\begin{figure}[htbp]
\centering
\includegraphics[scale=0.42]{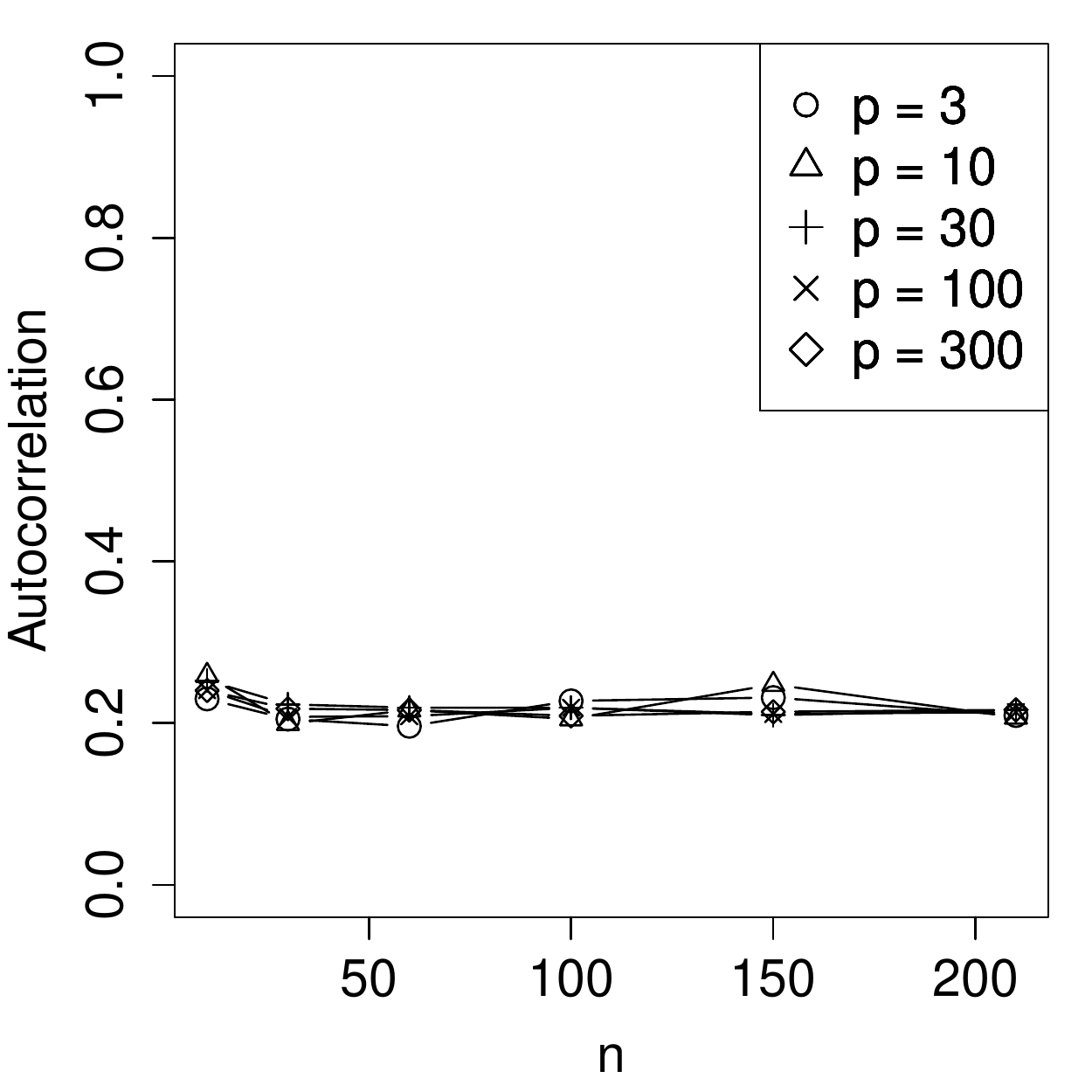}
\includegraphics[scale=0.42]{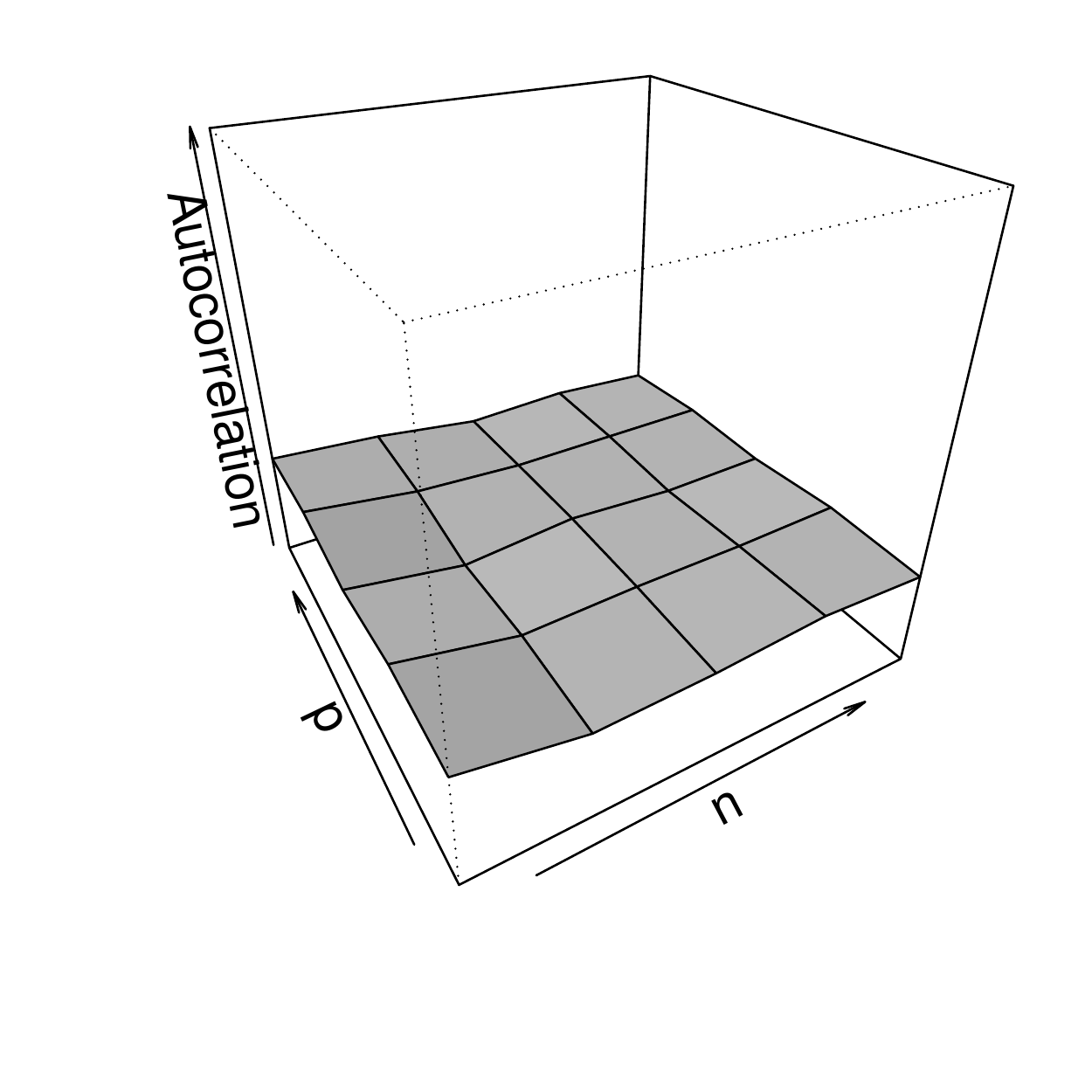}
\vspace*{-0.5em}
\caption{%
Autocorrelation of the $\tau^2_k$ chain versus
the sample size~$n$ (left)
for various values of~$p$.
Dimensional
autocorrelation function (DACF) surface plot for the $\tau^2_k$ chain relative to $n$ and~$p$ (right) for the unknown-variance normal hierarchical
model.
Both plots take $a_\sigma=s_\sigma=a_\tau=s_\tau=np$.
All other settings are the same as in Figure~\ref{fig:autocorr-nhuv}.
}
\label{fig:nhdd}
\end{figure}

It is also possible to use the empirical Bayesian approach to set the values of the points to which the aforementioned dimensionally-dependent priors converge.
Such a hybrid approach would enjoy bounded convergence rates while
both allowing the variances to remain stochastic 
and permitting sensible choices of the corresponding hyperparameters.


\section{%
Bounded Geometric Convergence Rates for High-Dimensional Regression
}
\label{sec:ir}

Recall that as discussed in Subsection~\ref{subsec:cvgc-cplx}, many applications of the method of \citet{rosenthal1995} yield an upper bound for the geometric convergence rate that tends to~$1$ as the dimension~$p$ tends to infinity.
It was demonstrated in Sections~\ref{sec:rr}~and~\ref{sec:lasso} that in the important regression setting, the actual convergence rate (as opposed to merely a bound) tends to~$1$ if the dimension~$p$ grows faster than the sample size~$n$.
Thus, MCMC-based inference for regression when $n=o(p)$ (i.e., in modern high-dimensional settings) remains a critical hurdle.
On the other hand, Sections~\ref{sec:mean}~and~\ref{sec:nh} demonstrated that in important classes of models like location models and hierarchical models, Gibbs sampling--type MCMC enjoys bounded convergence rates.
Then it may be asked (i)~whether bounded convergence rates can
nevertheless be attained for regression models,
and (ii)~whether such bounded convergence rates can be rigorously established by the method of \citet{rosenthal1995}.
In this section, we
present two approaches to address these
issues in the regression setting.
First, we propose a concrete framework in
which \citeauthor{rosenthal1995}'s method can still be used to obtain bounds on the convergence rate that do not tend to~$1$ as $p\to\infty$.
We apply this technique to a regression model with independent priors on $\bm\beta$ and~$\sigma^2$ to obtain the aforementioned bounded convergence rate.
Second, we propose an alternative, dimensionally-dependent prior specification that immediately yields provable bounded convergence rates while still retaining the classical conditional prior specification. Thus, we show that these approaches yield
the theoretical safeguard of geometric ergodicity so that MCMCs are still effective as a means to sample from
modern high-dimensional posteriors, that is, even if $n=o(p)$.

Consider a Gibbs sampler for drawing from a posterior distribution $\pi(\theta,\bm\phi\mid\bm Z)$, where $\theta$ is low-dimensional (say, $\theta\in\reals$)
but
$\bm\phi$ is high-dimensional (say, $\bm\phi\in\reals^p$),
and where $\bm Z$ denotes the data.
A two-step Gibbs sampler proceeds by drawing alternately from $\pi(\theta\mid\bm\phi,\bm Z)$ and $\pi(\bm\phi\mid\theta,\bm Z)$.
Then the Markov transition density
for drawing the next point based on the previous point $(\theta_0,\bm\phi_0)$
has the form
$f(\theta,\bm\phi\mid\theta_0,\bm\phi_0)=f_1(\theta\mid\bm\phi_0)\,f_2(\bm\phi\mid\theta)$,
where for simplicity we suppress the dependence
on~$\bm Z$
in the notation.
Now suppose we wish to prove a minorization condition in order to apply the result of \citet{rosenthal1995}.
Then it suffices to find a density $g(\theta,\bm\phi)$ and $\epsilon>0$ such that
$f(\theta,\bm\phi\mid\theta_0,\bm\phi_0)>\epsilon\,g(\theta,\bm\phi)$ for all $(\theta_0,\bm\phi)$ in some small set.
Observe that $g(\theta,\bm\phi)$ may be constructed by first finding a density $g_1(\theta)$ and $\epsilon>0$ such that
\basn
f_1(\theta\mid\bm\phi_0)>\epsilon\,g_1(\theta)
\label{rosenthal-1d}
\easn
for
all $\bm\phi_0$ in some small set, and then defining $g(\theta,\bm\phi)=g_1(\theta)\,f_2(\bm\phi\mid\theta)$.
Thus, if the high-dimensional parameter is drawn in the last step of the Gibbs sampling cycle, then a minorization condition can be established by working only with the low-dimensional distribution of the other parameter.
More precisely, the quantity $\epsilon>0$
that appears
in~(\ref{rosenthal-1d}) is used to bound a low-dimensional distribution by another low-dimensional distribution.
Hence, the convergence issue
discussed in
Subsection~\ref{subsec:cvgc-cplx},
in which the quantity $\epsilon$ takes the form $\epsilon=(\epsilon_\star)^p$, is thereby avoided.
Note that the same principle applies for Gibbs samplers of more than two steps as long as the high-dimensional parameter is confined to the last step of the Gibbs sampling cycle.
We illustrate
the
general approach above
in the next subsection.

\subsection{%
Independent-Prior Regression Model
}
\label{subsec:model-ir}

As an example of
the above
technique, consider a modification of the standard Bayesian regression framework in~(\ref{model-rr})
in which the
joint
prior on
$\bm\beta$ and $\sigma^2$ is specified as independent, i.e., $\pi(\bm\beta,\sigma^2)=\pi(\bm\beta)\,\pi(\sigma^2)$, that is, it is not specified
conditionally:
\begin{align}
\bm Y\mid\bm\beta,\sigma^2&\sim N_n(\bm X\bm\beta,\sigma^2\bm I_n),\notag\\
\bm\beta&\sim N_{p}(\bm0_{p},\lambda^{-1}\bm I_{p}),\label{model-ir}\\
\sigma^2&\sim\inversegamma(a/2,\,s/2),\notag
\end{align}
where $\bm X$ is a known $n\times p$ matrix (again with $n\ge5$),
and where the hyperparameters have known values $\lambda>0$, $a>2$, and $s>0$.
Note that the improper prior $\pi(\sigma^2)\propto1/\sigma^2$ is not used
in the above formulation
since it leads to an improper posterior when $p>n$.
Then a Gibbs sampler to draw from the posterior under~(\ref{model-ir}) may be constructed by taking an initial value~$\bm\beta_0\in\reals^p$ and then drawing (for every $k\ge1$)
\basn
\sigma^2_k\mid\bm\beta_{k-1},\bm Y&\sim\inversegamma\left(\frac{n+a}2,\;\frac{\|\bm Y-\bm X\bm\beta_{k-1}\|_2^2+s}{2}\right),\notag\\
\bm\beta_k\mid\sigma^2_k,\bm Y&\sim N_p\left(\tilde{\bm\beta}_{\sigma^2_k},\,\sigma^2_k\bm A_{\sigma^2_k}^{-1}\right),\label{gibbs-ir}
\easn
where $\bm A_{\sigma^2}=\bm X^T\bm X+\lambda\sigma^2\bm I_p$ and $\tilde{\bm\beta}_{\sigma^2}=\bm A_{\sigma^2}^{-1}\bm X^T\bm Y$.

\subsection{%
Convergence Properties
\& Bounded Convergence Rates
}
\label{subsec:cvgc-ir}

We now derive a quantitative upper bound for the convergence rate of the
the independent-prior regression Gibbs sampler
in~(\ref{gibbs-ir}).
To do so, we use the aforementioned approach that allows us to focus on the distribution of the low-dimensional parameter~$\sigma^2$ when establishing the minorization condition 
of
\citet{rosenthal1995}.
For every $k\ge0$, let $F_k(\bm\beta_0)$ denote the distribution of $(\sigma^2_k,\bm\beta_k)$ for the
the independent-prior regression Gibbs sampler in~(\ref{gibbs-ir})
started with initial value~$\bm\beta_0$, and let $F$ denote the stationary distribution of this chain, i.e., the true marginal posterior of~$(\sigma^2,\bm\beta)$.
Then we have the result shown below.
(
We
will preserve the notation of \citet{rosenthal1995} as closely as possible with a subscript~$R$ added, e.g., $\lambda_R$ is the quantity called simply $\lambda$ by \citeauthor{rosenthal1995} and is unrelated to the quantity we have called $\lambda$ elsewhere in the paper.)



\begin{thm}\label{thm:rosenthal}
For any $0<\alpha<1$ and any $d_R>2b_R/(1-\lambda_R)$,
\bas
&d_{\tv}\left[F_k(\bm\beta_0),\,F\right]\\
&\qquad\le
\left(1-\epsilon_R\right)^{\alpha k}\\
&\qquad\qquad+\left[\left(1+2b_R+2\lambda_R d_R\right)^\alpha\left(\frac{1+2b_R+\lambda_R d_R}{1+d_R}\right)^{1-\alpha}\right]^k
\left(1+\frac{b_R}{1-\lambda_R}+\left\|\bm Y-\bm X\bm\beta_0\right\|_2^2\right)
\eas
for every $k\ge1$, where
\bas
\lambda_R=\frac n{n+a-2},\qquad b_R=\bm Y^T\bm Y+\frac{ns}{n+a-2},\qquad
\epsilon_R=\left(\frac{s}{d_R+s}\right)^{(n+a)/2}.
\eas
In particular, the constants $\lambda_R$, $b_R$, $d_R$, and $\epsilon_R$ are functionally independent of~$p$.
\end{thm}

Note that the quantities governing the convergence rate in Theorem~\ref{thm:rosenthal} do not depend on the design matrix~$\bm X$ or on the parameter dimension~$p$.
Thus,
for any given fixed sample size~$n$,
the convergence
complexity
of the Markov chain
is
not affected by taking $p\to\infty$.
This result
is stated formally in the corollary below.
Let $n$ be fixed, and suppose we have a sequence of $n\times p$ covariate matrices $\bm X_p$ and a sequence of $p\times1$ vectors $\bm\beta_{0,p}$.
Let $F_{p,k}(\bm\beta_{0,p})$ denote the distribution of~$(\sigma^2_k,\bm\beta_k)$ for the
the independent-prior regression Gibbs sampler in~(\ref{gibbs-ir})
with initial value~$\bm\beta_{0,p}$, and let $F_p$ denote the stationary distribution of this chain.
Then we have the following result.

\begin{cor}\label{cor:rates-ir}
For the independent-prior regression Gibbs sampler in~(\ref{gibbs-ir}),
there exist $m(\bm X_p\bm\beta_{0,p})>0$ and $0<r<1$ such that
\bas
d_{\tv}\left[F_{p,k}\left(\bm\beta_0\right),\,F_p\right]\le m\left(\bm X_p\bm\beta_{0,p}\right)\, r^k
\eas
for all $k$ and $p$.
In particular, the geometric rate constant~$r$ is functionally independent of~$p$.
\end{cor}

By Corollary~\ref{cor:rates-ir} above,
there exists 
a single geometric rate constant~$r$ that holds for all~$p$.
Moreover, note that if the sequence of $n\times1$ vectors $\bm X_p\bm\beta_{0,p}$ is bounded uniformly in~$p$ (as would be the case for the starting point $\bm\beta_{0,p}=\bm0_p$, for example), then the multiplicative factor $m(\bm X_p\bm\beta_{0,p})$ in Corollary~\ref{cor:rates-ir} is also bounded uniformly in~$p$.

\begin{rem}
Recall from Section~\ref{sec:rr} that the convergence rate of the Gibbs sampler for standard regression tends to~$1$ as $p/n\to\infty$.
The Gibbs sampler for independent-prior regression is therefore fundamentally better in this ``large~$p$, small~$n$'' regime.
In particular, the Gibbs sampler above
has two important consequences:
(i)~it yields a proof of concept in which we can prove a bounded convergence rate using the method of \citet{rosenthal1995},
and (ii)~it gives an example that establishes prior specification as a possible solution to convergence problems in high dimensions
(at least when aiming to prove bounded geometric convergence rates).
\end{rem}

\subsection{%
Dimensionally-Dependent Prior Specification
}
\label{subsec:dd}

The above independent-prior analysis leads to the question of whether bounded convergence rates can be obtained while retaining the conditional-prior specification.
In this subsection we show that this is indeed the case.
In particular, suppose we take the model and priors as follows:
\begin{align}
\bm Y\mid\bm\beta,\sigma^2&\sim N_n(\bm X\bm\beta,\sigma^2\bm I_n),\notag\\
\bm\beta\mid\sigma^2&\sim N_{p}(\bm0_{p},\lambda^{-1}\sigma^2\bm I_{p}),\label{model-dd}\\
\sigma^2&\sim\inversegamma(\lceil p\epsilon\rceil/2,\,s/2),\notag
\end{align}
where $\lceil\,\cdot\,\rceil$ denotes the ceiling function, $\bm X$ is a known $n\times p$ matrix (again with $n\ge5$),
and the hyperparameters have known values $\lambda>0$, $\epsilon>2$, and $s>0$.
Then a Gibbs sampler to draw from the joint posterior under~(\ref{model-dd}) may be constructed by taking an initial value~$\sigma^2_0>0$ and then drawing (for every $k\ge1$)
\basn
\bm\beta_k\mid\sigma^2_{k-1},\bm Y
&\sim N_p\left(\tilde{\bm\beta},\,\sigma^2_{k-1}\bm A^{-1}\right),\notag\\
\sigma^2_k\mid\bm\beta_k,\bm Y&\sim\inversegamma\left\{\frac{
n+p+\lceil p\epsilon\rceil
}2,\;\frac12\left[\left(\bm\beta_k-\tilde{\bm\beta}\right)^T\bm A\left(\bm\beta_k-\tilde{\bm\beta}\right)+C+s\right]\right\},\label{gibbs-dd}
\easn
where $\bm A=\bm X^T\bm X+\lambda\bm I_p$ (which is positive-definite), $\tilde{\bm\beta}=\bm A^{-1}\bm X^T\bm Y$, and $C=\bm Y^T(\bm I_n-\bm X\bm A^{-1}\bm X^T)\bm Y$.
The convergence properties of the Gibbs sampler
above
can be obtained using the results previously established in Section~\ref{sec:rr} for the standard regression Gibbs sampler in~(\ref{gibbs-rr}) since their basic form is the same.
For every $k\ge0$, let $F_k(\sigma^2_0)$ denote the joint distribution of $(\bm\beta_k,\sigma^2_k)$ for the Gibbs sampler in~(\ref{gibbs-dd}) for regression under the dimensionally-dependent prior started with initial value~$\sigma^2_0$.  Let $F$ denote the stationary distribution of this chain, i.e., the true joint posterior of $(\bm\beta,\sigma^2)$. Then we have the following result.


\begin{thm}
\label{thm:tv-dd}
Consider the Gibbs sampler in~(\ref{gibbs-dd}) for
the regression model
under the dimensionally-dependent prior.
Then there exist $0<M_1\le M_2$ such that
\bas
M_1\left(\frac{p}{n+p+\lceil p\epsilon\rceil-2}\right)^k\le d_{\tv}\left[F_k\left(\sigma^2_0\right),\,F\right]\le M_2\left(\frac{p}{n+p+\lceil p\epsilon\rceil-2}\right)^k
\eas
for every $k\ge0$.
Moreover,
the geometric rate constant $r=p/(n+p+\lceil p\epsilon\rceil-2)$
is bounded above by $1/(1+\epsilon)$
for all $n$ and~$p$.
\end{thm}

Thus,
the dimensionally-dependent prior specification in~(\ref{model-dd}) provides an alternative approach to obtaining bounded convergence rates that preserves the conditional prior specification in which $\var(\bm\beta\mid\sigma^2)\propto\sigma^2$.

We now show that
the dimensionally-dependent prior specification in~(\ref{model-dd}) can also be used
where it is very relevant, that is, in high-dimensional Bayesian model selection
(see Section~\ref{sec:lasso}).
Figure~\ref{fig:autocorr-dd} shows
analogous
results of Gibbs sampling runs
as in Figure~\ref{fig:autocorr}
for the Bayesian lasso (left), Bayesian elastic net (center), and spike-and-slab prior
(right). Here the prior on $\sigma^2$ has been
taken as $\inversegamma(p/2,\,1/2)$, with all other settings the same as in Figure~\ref{fig:autocorr}.
Remarkably,
the dimensionally-dependent prior specification does yield bounded
autocorrelations,
thus yielding a workable solution for the use of model selection priors in high dimensions.

\begin{figure}[htbp]
\centering
\includegraphics[scale=0.42]{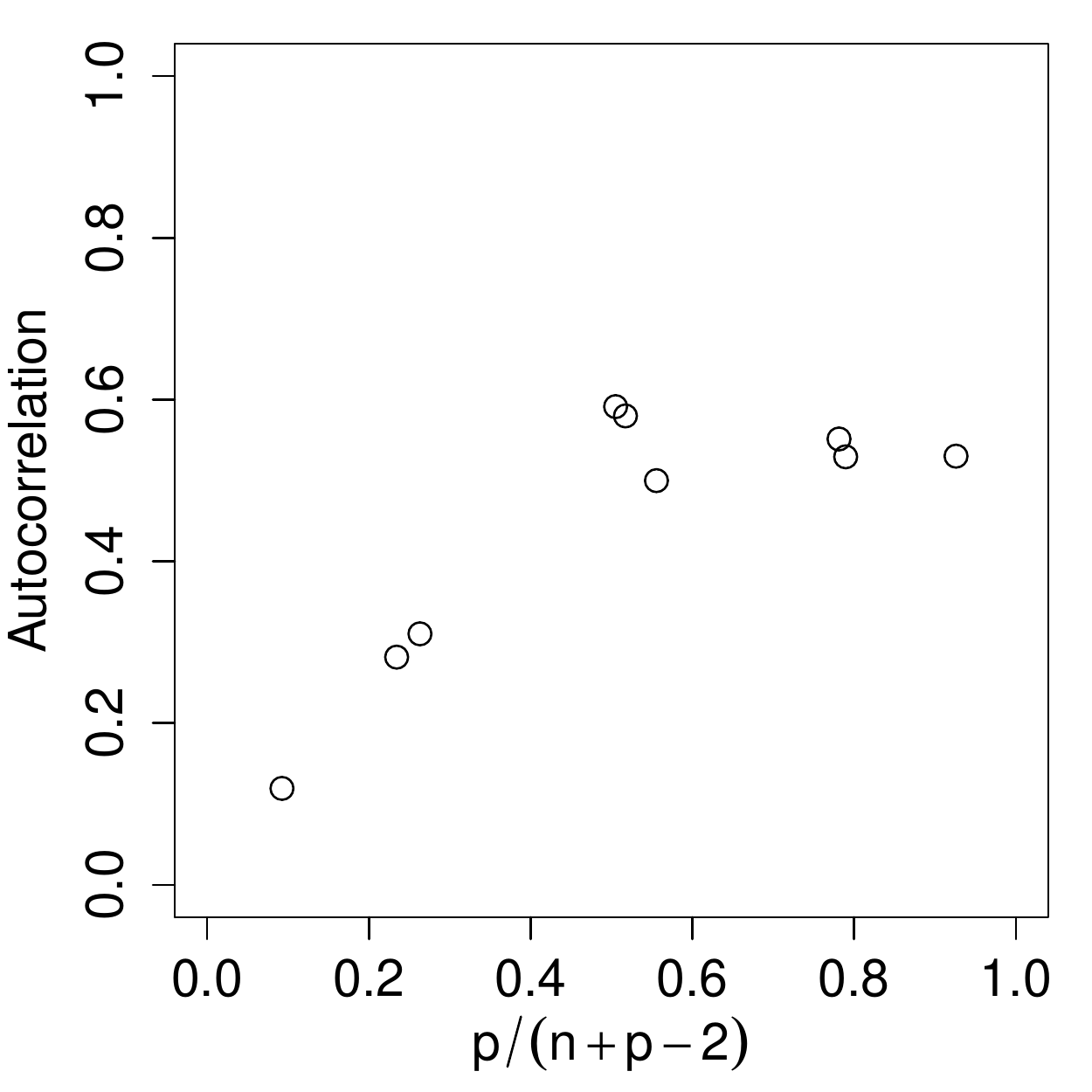}
\includegraphics[scale=0.42]{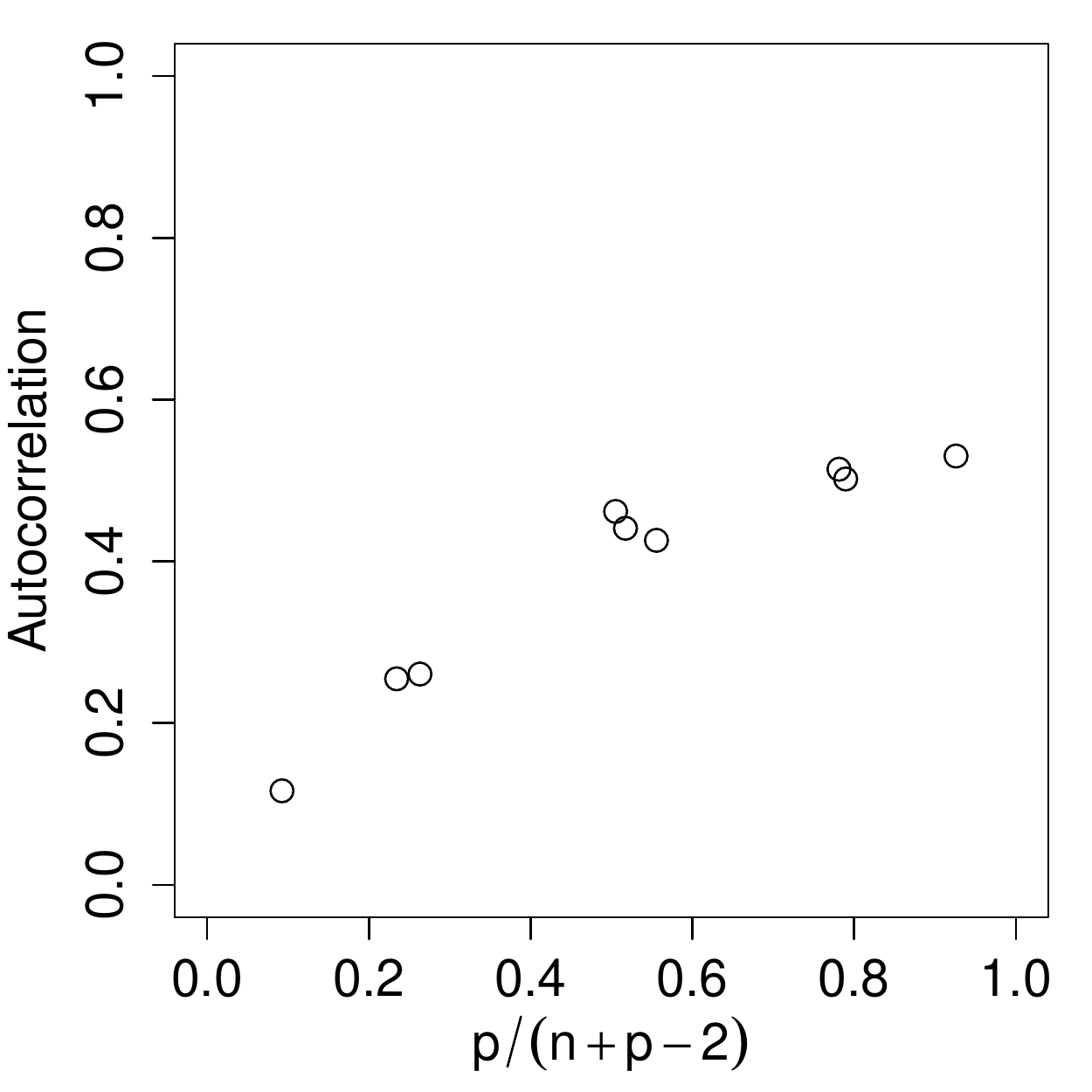}
\includegraphics[scale=0.42]{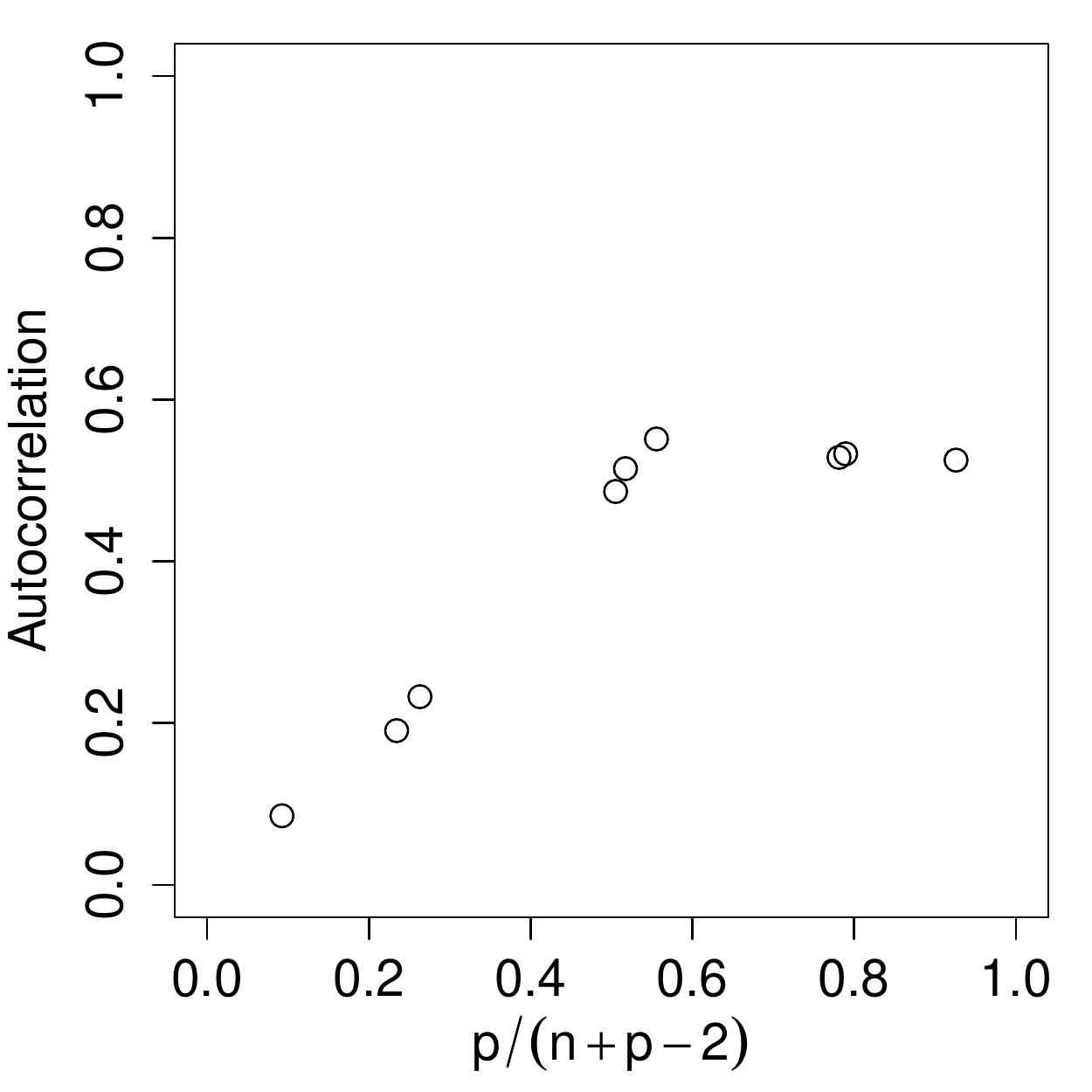}
\vspace*{-0.5em}
\caption{%
Autocorrelation of the $\sigma^2_k$ chain versus $p/(n+2p-2)$ for the
Gibbs sampler for the Bayesian lasso (left), Bayesian elastic net (center), and the spike-and-slab prior (right)
under the dimensionally-dependent prior specification where $\sigma^2\sim\inversegamma(p/2,\,1/2)$.  All other settings are the same as in Figure~\ref{fig:autocorr}.
}
\label{fig:autocorr-dd}
\end{figure}

Note that the above analysis shows that in principle one could choose the degrees of freedom in the prior specification of~$\sigma^2$ in order to obtain a desired geometric convergence rate.
Hence, the prior can be specified in such a way that that the resulting Markov chain achieves convergence to within a given tolerance~$\epsilon$ in a desired number of iterations.

\section{Discussion and Conclusions}
\label{sec:conclusion}

The preceding sections presented results on the convergence properties of various Gibbs samplers in high-dimensional regimes for important classes of statistical models. Although in some cases we consider simplified model and prior combinations, convergence rates can be extended to more sophisticated models, as demonstrated in Section \ref{sec:lasso}. We now summarize the results in the paper and discuss their implications for high-dimensional  MCMC and Bayesian inference in both theory and practice.

\subsection{Summary of Results on Convergence Rates}
\label{subsec:summary}

Sections~\ref{sec:rr},~\ref{sec:lasso},~\ref{sec:mean},~\ref{sec:nh},~and~\ref{sec:ir} considered the convergence complexity of the Gibbs samplers for several key models. The results are summarized in
the table
in Supplemental Section~\ref{sec:supp-table}.
In particular, there are three important conclusions that
can be drawn from our analysis in this paper.
First, many MCMC schemes for popular models enjoy bounded geometric convergence rates. This property gives safeguards regarding the effectiveness of using
standard
MCMCs in modern high-dimensional settings
and is a welcome message.
Important examples include
multivariate mean models, hierarchical models with known variances, regression models when $p=O(n)$,
and graphical models with bounded vertex degree.
Thus, by and large, and contrary to what is generally perceived, convergence of high-dimensional MCMCs is achieved in many models, and even when not, there are possible solutions.
Indeed, even in problematic cases, we have been able to resolve the convergence issue.
Second, the Gibbs samplers for Bayesian analysis of some commonly used models have a convergence rate that can be arbitrarily close to~$1$ in high-dimensional regimes. An important case of this phenomenon is the class of regression-type models
when $n=o(p)$ and when the usual (conditional) prior specification is used.
Third, the convergence complexity of the Gibbs sampler corresponding to a particular model can differ substantially from that of a similar one, i.e., slight changes to the model or prior can lead to very different convergence behavior. A case in point is the normal hierarchical model, in which the known-variance version enjoys convergence rates bounded away from~$1$ while the unknown-variance version does not.

Though the mechanics of the convergence behavior of these chains are difficult to predict beforehand, there are nevertheless some patterns that can be observed.
Models often feature one or more nuisance parameters that tie together a large number of other parameters.
Typical Bayesian practice would often be to take such parameters as unknown with some uninformative prior.
However, in high dimensions, such an approach can lead to extreme posterior dependence between parameters due to the structure of the likelihood
or
conditional priors of other parameters.
This dependence can dramatically worsen the convergence rate of associated Gibbs samplers.

\subsection{Remedies for Potential Convergence Problems}
\label{subsec:remedies}

When convergence problems arise due to stochasticity of the nuisance parameters in the manner discussed in the previous susbection, the most straightforward solution is simply to take these nuisance parameters as known.
Using empirical Bayes to specify the nuisance parameters is a viable option.
However, if such parameters must be taken as unknown, then an intermediate approach is to take strongly informative priors for these parameters.
Of course, such specifications still require the practitioner to supply prior knowledge of these parameters' (approximate) values.
Empirical Bayes can once more be very useful in such instances.
Specifically,
hybrid methods that combine empirical Bayes with dimensionally-dependent hyperparameters
can enjoy the superior convergence complexity of the known-parameter approach while retaining the obvious inferential benefits of allowing these parameters to be
stochastic.
Clearly there is often a trade-off between convergence complexity and other goals of Bayesian inference.


More generally, a variety of practical methods
have been proposed for improving the convergence rate of Markov chains used in Bayesian inference in the classical regime where $n$ and~$p$ are fixed and $p$ is small. Such methods may involve reorganization of the structure of the actual sampling steps by grouping or collapsing \citep[see, e.g.,][]{liu1994w}, reparametrization of the model \citep[see, e.g.,][]{gelfand1995,roberts1997,papaspiliopoulos2007,yu2011}, or expansion of the parameter space by methods such as PX-DA \citep{liu1999}. It is well established that these methods can indeed reduce the value of the constant associated with the geometric convergence in some settings for any particular fixed values of $n$ and~$p$. However, what is less clear is whether such techniques can qualitatively alter the behavior of a chain in terms of convergence complexity in various $n$ and~$p$ regimes. More precisely, it is essential to determine whether there are settings and regimes where $r_{n,p}\to1$ for some basic chain, but where the use of one of these convergence acceleration
methods
can instead yield a chain for which $r_{n,p}$ is bounded away from~$1$. Such questions are topics to be investigated in forthcoming work. Another potentially useful tool for diagnosing and understanding convergence complexity from a practical point of view is the dimensional autocorrelation function, or DACF, plots introduced in Figure~\ref{fig:autocorr-nhuv}. These plots can potentially provide insight into the convergence complexity of a Markov chain as a function of dimension and sample size. Additionally, when slow convergence is encountered by MCMC practitioners in any given application, a DACF plot could potentially aid in determining whether the problems are due to convergence complexity issues or some other cause.

From a theoretical standpoint, we were able to establish convergence rates that are bounded away from~$1$ in important classes of models. Even in the problematic regression setting, the ability of MCMC as an efficient tool to sample from the posterior was demonstrated,
including in settings where
the dimension is larger than the sample size.
Novel approaches are nevertheless
required when the dimension grows faster than the sample size. Section~\ref{sec:ir} used the Gibbs sampler for an independent-prior regression approach as an example of the specific way in which the result of \citet{rosenthal1995} can still be used to obtain bounded convergence rates in  sample-starved high-dimensional settings.
However,
it seems that there may only be certain cases in which such an approach can provide a bound that is sharp enough to permit analysis in various $n$ and~$p$ regimes.  In this paper, we have tried to overcome this problem by using a ``first principles'' approach by considering various classes of important models and analyzing the convergence behavior of the corresponding Gibbs samplers. It would be useful to generalize this strategy. We therefore hope that one consequence of our work will be to motivate the proposal and development of new ideas analogous to those of \citeauthor{rosenthal1995} that are suitable for high-dimensional settings.

\newpage
\pagestyle{empty}

\bibliographystyle{ims}
\bibliography{references}

\section*{Acknowledgments}

We thank Peter Craigmile for kindly agreeing to read the paper once it was completed.
This work was partially funded by the US National Science Foundation under grants
DMS-CMG-1025465, AGS-1003823, DMS-1106642, DMS-CAREER-1352656,
and the US Air Force Office
of Scientific Research grant award FA9550-13-1-0043.

\newpage

\appendix

\pagestyle{plain}
\setcounter{page}{1}

\part*{%
Supplemental Sections
}

\section{%
Preliminaries
}
\label{sec:supp-preliminaries}

If
$P$ and $Q$ are both distributions on~$\reals^D$, then
the Wasserstein distance between $P$ and~$Q$ is
\bas
d_W(P,Q)=\inf\,E\left(\|\bm X-\bm Y\|_1\right)=\inf\,E\left(\sum_{j=1}^D\left|X_j-Y_j\right|\right),
\eas
where the infimum is taken over all joint distributions of
$D$-dimensional
random vectors $\bm X$ and~$\bm Y$ with respective marginal distributions $P$ and~$Q$ \citep{wasserstein1969}.
The Wasserstein distance may be equivalently defined as
\bas
d_W(P,Q)=\sup\,\left|\int h\,dP-\int h\,dQ\right|
\eas
where the supremum is taken over all functions~$h$ such that $h(\bm x,\bm y)\le\|\bm x-\bm y\|_1$ for all $\bm x,\bm y\in\reals^D$
\citep{szulga1983}.
Note that this distance is sometimes called the first Wasserstein distance since it can also be generalized
by replacing the
$\ell_1$
norm (in either definition) with some other norm.
\citet{gibbs2002} provide further discussion of the relationships between
Wasserstein distance, total variation distance, and other distances between distributions.

\begin{proof}[%
Proof of Lemma~\ref{lem:ex-logistic}%
]
We consider the two asymptotic regimes separately.

\emph{Case~I: Fixed~$p$, increasing~$n$.}
Since $\det\bm A\ge\lambda^p$, it is clear that
$\delta\le2^{-n}e^{-n/4}$.
Thus, as $n\to\infty$,\, 
$\delta\to0$,
so
$\tilde r=1-\delta\to1$.
In fact,
$\tilde r\to1$
exponentially fast as $n\to\infty$.
Now suppose that the bound in~(\ref{bad-bound-lasso})
is used to calculate a number of iterations~$K_{n,\epsilon}$ that will yield convergence to within a given tolerance~$\epsilon>0$.
If $n$ is large, then $\delta$ is small, and hence
\bas
K_{n,\epsilon}=\frac{\log(\epsilon/M)}{\log(1-\delta)}\approx\frac1\delta\log(M/\epsilon)\ge(2e^{1/4})^n\log(M/\epsilon)=\exp\left[n\left(\frac14+\log2\right)\right]\log(M/\epsilon).
\eas
Thus, the number of iterations required for convergence based on the bound in~(\ref{bad-bound-lasso})
grows exponentially in the sample size.

\emph{Case~II: Fixed~$n$, increasing~$p$.}
Suppose $p\ge n$,
and let $\bm X=\bm U\bm\Omega\bm V^T$,
where
the $n\times n$ matrix~$\bm U$ and the $p\times p$ matrix~$\bm V$
are orthogonal with columns
$\bm u_1,\ldots,\bm u_n$ and
$\bm v_1,\ldots,\bm v_{
p
}$ (respectively),
and where $\bm\Omega$ is $n\times p$ rectangular-diagonal with $\bm\Omega=\diag_{n\times p}(\omega_1,\ldots,\omega_n)$.
(Note that these matrices depend on~$p$, although we do not indicate this dependence explicitly in the notation.)
Then
\bas
\bm A^{-1/2}
=\left(\frac12\bm V\bm\Omega^T\bm\Omega\bm V^T+\lambda\bm V\bm V^T\right)^{-1/2}
&=\bm V\left(\frac12\bm\Omega^T\bm\Omega+\lambda\bm I_p\right)^{-1/2}\bm V^T.
\eas
Then
an alternative upper bound for $\delta$ is
\bas
\delta&\le\exp\left(-\frac1{4\lambda}\left\|\bm X\bm A^{-1/2}\bm X^T\tilde{\bm Y}\right\|_2^2\right)\\
&=\exp\left[-\frac1{4\lambda}\left\|\bm U\bm\Omega\bm V^T\bm V\left(\frac12\bm\Omega^T\bm\Omega+\lambda\bm I_p\right)^{-1/2}\bm V^T\bm V\bm\Omega^T\bm U^T\tilde{\bm Y}\right\|_2^2\right]\\
&
=\exp\left[-\frac1{4\lambda}\tilde{\bm Y}^T\bm U\bm\Omega
\diag_{p\times p}\left(\frac{2\omega_1^2}{\omega_1^2+2\lambda},\ldots,\frac{2\omega_n^2}{\omega_n^2+2\lambda},0,\ldots,0\right)
\bm\Omega^T\bm U^T\tilde{\bm Y}\right]\\
&=\exp\left[-\frac1{2\lambda}\sum_{i=1}^n\frac{\omega_i^4}{\omega_i^2+2\lambda}\left(\bm u_i^T\tilde{\bm Y}\right)^2\right]
\le\exp\left[-\frac1{2\lambda}\left(\max_{1\le i\le n}\omega_i^2\right)\left\|\tilde{\bm Y}\right\|_2^2\right]{
.
}
\eas
Now observe that $\omega_1^2,\ldots,\omega_n^2$ are the eigenvalues of $\bm X\bm X^T$. 
The largest of these eigenvalues is bounded below by the largest diagonal element of $\bm X\bm X^T$,
i.e.,
$\max_{1\le i\le n}\omega_i^2\ge\max_{1\le i\le n}\sum_{j=1}^p X_{ij}^2$, which is of order~$p$.
Thus, $\max_{1\le i\le n}\omega_i^2\to\infty$ as $p\to\infty$.
Since $\|\tilde{\bm Y}\|_2^2$ does not depend on~$p$, it follows
that $\delta\to0$ as $p\to\infty$, so again $1-\delta\to1$.
Furthermore, $1-\delta\to1$ exponentially fast
as $p\to\infty$.
\end{proof}

\begin{proof}[%
Proof of Lemma~\ref{lem:ex-lasso}%
]
First, note that the prior $\pi(\sigma^2)\propto1/\sigma^2$ above corresponds to $\alpha=\xi=0$ in the notation of \citet{khare2013}.  Now define $\gamma$ and $b$ as in equations (3.9) and~(3.10) of \citet{khare2013}:
\bas
\gamma&=\max\left\{\frac p{n+p-2},\frac12\right\},\qquad
b=\bm Y^T\bm Y+\frac{p(n+2p)}{2\lambda^2}+\frac p{\lambda^2}.
\eas
Next, let $d>2b/(1-\gamma)$ as required by Proposition~4 of \citet{khare2013},
and then define~$\originalepsilon$ as in equation~(3.16) of \citet{khare2013}:
\bas
\originalepsilon=e^{-1/2}\left\{\frac{\bm Y^T\left[\bm I_n-\bm X\left(\bm X^T\bm X+d^{-1}\bm I_p\right)^{-1}\bm X^T\right]\bm Y}{d\left(1+p^2\lambda^2d\right)}\right\}^{(n+p)/2}.
\eas
Observe that $d>2b/(1-\gamma)>2b>2\bm Y^T\bm Y$, and thus
$\originalepsilon\le(\bm Y^T\bm Y/d)^{(n+p)/2}\le2^{-(n+p)/2}\to0$ as either $n$ or~$p$ tends to~$\infty$.
Now observe that Proposition~4 of \citet{khare2013} establishes a bound~$\tilde r$ for the geometric rate constant that is at least as large as $1-\originalepsilon$.
Then this bound $\tilde r$ tends to~$1$ exponentially fast as either $n$ or~$p$ tends to $\infty$.
\end{proof}

\section{%
Regression Models
}
\label{sec:supp-rr}

\begin{proof}[%
Proof of Theorem~\ref{thm:tv-rr}%
]
Begin by writing the Gibbs sampler in~(\ref{gibbs-rr}) as (for every $k\ge1$)
\bas
\bm\beta_k&=\tilde{\bm\beta}+\sqrt{\sigma^2_{k-1}}\;\bm A^{-1/2}\bm Z_k,
&&\text{where }\bm Z_k\sim N_p(\bm0_p,\bm I_p),\\
\sigma^2_k&=\frac1{V_k}\left[\left(\bm\beta_{k}-\tilde{\bm\beta}\right)^T\bm A\left(\bm\beta_{k}-\tilde{\bm\beta}\right)+C\right],
&&\text{where }V_k\sim\chi^2_{n+p},
\eas
and where all of the $\bm Z_k$ and $V_k$ are 
independent.
Substituting for $\bm\beta_k$
yields (for every $k\ge1$)
\basn
\label{marginal-rr}
\sigma^2_k=\frac1{V_k}\left(\sigma^2_{k-1}U_k+C\right),
\qquad\text{ where }U_k\sim\chi^2_p,\;V_k\sim\chi^2_{n+p},
\easn
and where the $U_k$ and $V_k$ are all independent.
Note that the marginal posterior of~$\sigma^2$, and hence the stationary distribution of the marginal chain in~(\ref{marginal-rr}) above, is
\bas
\sigma^2\mid\bm Y&\sim\inversegamma\left(\frac n2,\frac{
C
}2\right)
\eas
by
elementary Bayesian computations \citep[see, e.g.,][]{ohagan2010}.

We first establish the lower bound.
By the results of \citet{liu1994w}, it suffices to show that
$r=p/(n+p-2)$ is an
eigenvalue of~$\bm F_{\sigma^2}$, the forward operator of
the marginal chain in~(\ref{marginal-rr})
on the space of mean-zero, finite-variance functions on the positive half-line.
Let $\psi(\sigma^2)=\sigma^2-C/(n-2)$, which has mean zero and finite variance under the stationary distribution (i.e., the true marginal posterior).
Then it is clear
from the form of the marginal chain in~(\ref{marginal-rr})
that
\bas
E\left[\psi\left(\sigma^2_k\right)\given\sigma^2_{k-1}\right]=\frac p{n+p-2}\,\psi\left(\sigma^2_{k-1}\right)
\eas
for every $k\ge1$.  Thus, $\psi$ is an eigenfunction of the marginal forward operator~$\bm F_{\sigma^2}$ with eigenvalue
$p/(n+p-2)$,
so the largest eigenvalue of~$\bm F_{\sigma^2}$ is at least $p/(n+p-2)$.
Thus, if the joint and marginal chains are geometrically ergodic, then their convergence rate is at least $p/(n+p-2)$, and the result
follows.

\bc{81}
We now establish the upper bound.
Make the transformation $\bm\theta=\bm A^{1/2}(\bm\beta-\tilde{\bm\beta})$ and (for every $k\ge1$) $\bm\theta_k=\bm A^{1/2}(\bm\beta_k-\tilde{\bm\beta})$, so that the Gibbs sampler in~(\ref{gibbs-rr}) may be written as
\bas
\bm\theta_k&=\sqrt{\sigma^2_{k-1}}\,\bm Z_k,&&\text{where }\bm Z_k\sim N_p(\bm0_p,\bm I_p),\\
\sigma^2_k&=\frac{\|\bm\theta_k\|_2^2+C}{V_k},
&&\text{where }V_k\sim\chi^2_{n+p},
\eas
for every $k\ge1$,
where all of the $\bm Z_k$ and $V_k$ are  independent.
Substituting for $\sigma^2_{k-1}$ yields
\basn
\bm\theta_k&=\sqrt{\frac{\|\bm\theta_{k-1}\|_2^2+C}{V_{k-1}}}\,\bm Z_k,&&\text{where }\bm Z_k\sim N_p(\bm0_p,\bm I_p),\;V_{k-1}\sim\chi^2_{n+p},
\label{theta-rr}
\easn
for every $k\ge2$,
where all of the $\bm Z_k$ and $V_k$ are independent.  (The fact that $\bm\theta_1$ differs will ultimately be irrelevant.)
Note that the marginal posterior of~$\bm\theta$, and hence the stationary distribution of the marginal chain in~(\ref{theta-rr}) above, is
\bas
\left.\sqrt{\frac nC}\;\bm\theta\given\bm Y\right.\sim t_{p,n},
\eas
where $t_{p,n}$ denotes the $p$-variate $t$~distribution with $n$ degrees of freedom,
by elementary Bayesian computations \citep[see, e.g.,][]{ohagan2010}.
By the results of \citet{liu1994w}, it suffices to work with the marginal chain in~(\ref{theta-rr}).
Although the specification of~$\bm\theta_1$ differs from that of~$\bm\theta_k$ for $k\ge2$, this issue clearly has no effect on the convergence rate.
Thus, we henceforth work with a modified version of the marginal $\bm\theta_k$ chain in which some initial value $\bm\theta_0\in\reals^p$ is provided and the specification in~(\ref{theta-rr}) is extended to $k=1$ as well (and where we take $V_0\sim\chi^2_{n+p}$ to be independent of the other $\bm Z_k$ and $V_k$).
For every $k\ge1$, let $G_k(\bm\theta_0)$ denote the distribution of the $k$th iterate of this modified version of the marginal $\bm\theta_k$ chain with starting value~$\bm\theta_0$.
Let $G$ denote the corresponding stationary distribution (i.e., $G$ is a scaled $t_{p,n}$ distribution).
\ec{81}


\bc{81}
Now define $\eta=\|\bm\theta\|_2^2/(\|\bm\theta\|_2^2+C)$ and (for every $k\ge0$) $\eta_k=\|\bm\theta_k\|_2^2/(\|\bm\theta_k\|_2^2+C)$.  Then
\basn
\eta_k=\frac{W_k}{W_k+1-\eta_{k-1}},&&\text{ where }W_k=\frac{\|\bm Z_k\|_2^2}{V_{k-1}}\sim\betaprime\left(\frac p2,\,\frac{n+p}2\right),
\label{eta-rr}
\easn
for every $k\ge1$, where all of the $W_k$ are independent and where $\betaprime(\cdot,\cdot)$ denotes the beta prime distribution or beta distribution of the second kind.
For every $k\ge1$, let $H_k(\eta_0)$ denote the distribution of the $k$th iterate of the $\eta_k$ chain in~(\ref{eta-rr}) with starting value $\eta_0\in[0,1)$, noting that this distribution depends on~$\bm\theta_0$ only through~$\|\bm\theta_0\|_2^2$ and hence only through~$\eta_0$.
Let $H$ denote the corresponding stationary distribution, noting that $H$ is the $\betad(p/2,\,n/2)$ distribution by elementary results.
Observe that for every $k\ge1$, the distribution of $\|\bm\theta_k\|_2^{-1}\bm\theta_k$ is uniform on the unit sphere in~$\reals^p$ and is therefore independent of~$\eta_k$.
Similarly, the posterior distribution of $\|\bm\theta\|_2^{-1}\bm\theta\mid\bm Y$ is uniform on the unit sphere in~$\reals^p$ and hence is independent (a~posteriori) of~$\eta\mid\bm Y$.
Then it follows that
$d_{\tv}[G_k(\bm\theta_0),G]=d_{\tv}[H_k(\eta_0),H]$
by the properties of total variation distance.
Thus, it suffices to show that the $\eta_k$ chain in~(\ref{eta-rr}) converges geometrically in total variation distance with rate constant no larger than $p/(n+p-2)$.
\ec{81}

\bc{81}
Now suppose that $\eta_0\sim H$, i.e., suppose that the chain is stationary.  Then the geometric convergence rate of the $\eta_k$ chain in general (i.e., not under stationarity) is equal to $\gamma(\eta_0,\eta_1)$, the maximal correlation of successive iterates under stationarity \citep{liu1994w}.
The application of standard transformational techniques to~(\ref{eta-rr}) shows that
the joint distribution of $\eta_0$ and $\eta_1$ under stationarity has density 
\bas
f^{(\eta_0,\eta_1)}(\eta_0,\eta_1)=\frac{\Gamma[(n+2p)/2]}{\left[\Gamma(p/2)\right]^2\,\Gamma(n/2)}\;
\frac{\eta_0^{(p-2)/2}(1-\eta_0)_{\phantom0}^{(n+p-2)/2}\eta_1^{(p-2)/2}(1-\eta_1)_{\phantom0}^{(n+p-2)/2}}{(1-\eta_0\eta_1)^{(n+2p)/2}}
\eas
with respect to Lebesgue measure on $(0,1)\times(0,1)$.
Then the joint distribution of $\eta_0$ and~$\eta_1$ is the bivariate beta distribution of \citet{olkin2003} with parameters $p/2$, $p/2$, and $n/2$.
By equation~(1.6) of \citet{olkin2003}, the distribution of $\eta_0$ and~$\eta_1$ may be written as
\bas
\eta_0,\eta_1\mid Q\sim\iid\betad\left(\frac p2+Q,\,\frac{n+p}2\right),\qquad
Q\sim\betanegbin\left(\frac p2,\frac p2,\frac n2\right),
\eas
where $\betanegbin(\cdot,\cdot,\cdot)$ denotes the beta--negative binomial distribution, i.e., the distribution of~$Q$ has density
\bas
f^{(Q)}(q)=\left\{\frac{\Gamma[(n+p)/2]\,\Gamma[(p+2q)/2]}{\Gamma(p/2)}\right\}^2\frac1{\Gamma(n/2)\,q!\;\Gamma[(n+2p+2q)/2]}
\eas
with respect to counting measure on~$\mathbb N_0$, where $\mathbb N_0$ denotes the set of nonnegative integers.
Then $\gamma(\eta_0,\eta_1)=[\gamma(\eta_0,Q)]^2$ \citep[][Lemma~2.1]{yu2008}.
We now write $\eta_0$ as $\tilde\eta$ to avoid an eventual conflict in notation. The joint distribution of $\tilde\eta$ and~$Q$ has density
\basn
f^{(\tilde\eta,Q)}(\tilde\eta,q)=\frac{\Gamma[(n+p)/2]}{\Gamma(n/2)\,[\Gamma(p/2)]^2}\,(\tilde\eta)^{(p+2q-2)/2}(1-\tilde\eta)^{(n+p-2)/2}\,\frac{\Gamma[(p+2q)/2]}{q!}
\label{joint-eta-q}
\easn
with respect to the product of Lebesgue measure on~$(0,1)$ and counting measure on~$\mathbb N_0$.  Since $\tilde\eta\sim\betad(p/2,\,n/2)$ marginally, the conditional density of $Q\mid\tilde\eta$ with respect to counting measure on~$\mathbb N_0$ is
\bas
f^{(Q\mid\tilde\eta)}(q\mid\tilde\eta)=\frac{(1-\eta)^{p/2}}{\Gamma(p/2)}\,\eta^q\,\frac{\Gamma[(p+2q)/2]}{q!},
\eas
i.e., $Q\mid\tilde\eta\sim\negbin(p/2,\,\tilde\eta)$, where $\negbin(\cdot,\cdot)$ denotes the negative binomial distribution.
Now consider an entirely separate Gibbs sampler that draws from the joint distribution of $\tilde\eta$ and~$Q$ in~(\ref{joint-eta-q}) by fixing a starting point $Q_0\in\mathbb N_0$ and drawing (for every $k\ge1$)
\basn
\tilde\eta_k\mid Q_{k-1}\sim\betad\left(\frac p2+Q_{k-1},\,\frac{n+p}2\right),
\qquad
Q_k\mid\tilde\eta_k\sim\negbin\left(\frac p2,\,\tilde\eta_{k}\right).
\label{gibbs-eta-q}
\easn
Then $[\gamma(\tilde\eta,Q)]^2$ under the joint distribution in~(\ref{joint-eta-q}) is equal to the convergence rate of the Gibbs sampler in~(\ref{gibbs-eta-q})
\citep{liu1994w}.
Therefore the original standard regression Gibbs sampler has the same convergence rate as the Gibbs sampler in~(\ref{gibbs-eta-q}).
Integrating out $\tilde\eta_k$ from the joint chain in~(\ref{gibbs-eta-q}) yields
\basn
Q_k\mid Q_{k-1}\sim\betanegbin\left(\frac p2,\,\frac p2+Q_{k-1},\,\frac{n+p}2\right)
\label{chain-q}
\easn
for every $k\ge1$.
The convergence rate of this marginal $Q_k$ chain equals that of the joint chain and hence equals that of the original standard regression Gibbs sampler.
Observe that
\bas
E(Q_1\mid Q_0=q_0)=\frac p2\left(\frac p2+q_0\right)\left(\frac{n+p}2-1\right)^{-1}
=\left(\frac p{n+p-2}\right)q_0+\frac{p^2}{2(n+p-2)}
\eas
by the basic properties of the beta--negative binomial distribution.
Then since the marginal $Q_k$ chain in~(\ref{chain-q}) is stochastically monotone, we may apply Theorem~2.1 of \citet{diaconis2010} with their monotone function~$f$ taken as the identity function, their $\lambda\in(0,1)$ taken as $\lambda=p/(n+p-2)$, and their $c>0$ taken as $c=1$.
It follows immediately that an upper bound for the convergence rate of the chain in~(\ref{chain-q}), and hence an upper bound for the convergence rate of the standard regression Gibbs sampler, is $p/(n+p-2)$.
\ec{81}
\end{proof}

\bc{81}
We
\ec{81}
now establish a sharp bound for the $d_W$-convergence rate of
the marginal chain in~(\ref{marginal-rr}).
Note that once more
we will use the behavior of the marginal chain in~(\ref{marginal-rr}) as a surrogate for the behavior of the overall chain in~(\ref{gibbs-rr}).
For
every $k\ge0$, let $G_k(\sigma^2_0)$ denote the distribution of $\sigma^2_k$ for the marginal chain in~(\ref{marginal-rr}) started with initial value~$\sigma^2_0$, and let $G$ denote the stationary distribution of this chain, i.e., the true marginal posterior of~$\sigma^2$. Note that $G$ is simply the $\inversegamma(n/2,\,C/2)$
distribution.
Then we have the following result.

\begin{thm}\label{thm:wasserstein-rr}
For the
marginal chain in~(\ref{marginal-rr}) of the standard Bayesian regression Gibbs sampler 
specified in~(\ref{gibbs-rr}),
\bas
M_1\left(\sigma^2_0\right)\,\left(\frac p{n+p-2}\right)^k
\le d_W\left[G_k(\sigma^2_0),G\right]\le
M_2\left(\sigma^2_0\right)\,\left(\frac p{n+p-2}\right)^k
\eas
for every $k\ge0$, where
\bas
M_1\left(\sigma^2_0\right)=\left|\sigma^2_0-\frac C{n-2}\right|,\qquad
M_2\left(\sigma^2_0\right)=\left|\sigma^2_0-\frac C{n-2}\right|+\frac C{n-2}\sqrt{\frac{2}{n-4}}.
\eas
\end{thm}

\begin{proof}
We first establish the lower bound.
Let $\psi(\sigma^2)=\sigma^2-C/(n-2)$.
Then it is clear
from the form of the marginal chain in~(\ref{marginal-rr})
that
\bas
E\left[\psi\left(\sigma^2_k\right)\given\sigma^2_{k-1}\right]=\frac p{n+p-2}\,\psi\left(\sigma^2_{k-1}\right)
\eas
for every $k\ge1$, which by repeated application yields
\bas
E\left[\psi\left(\sigma^2_k\right)\right]=\left(\frac p{n+p-2}\right)^k\psi(\sigma^2_0)
\eas
for every $k\ge0$.
Now note that $\psi$ has Lipschitz constant~$1$.
Then
\bas
d_W\left[G_k(\sigma^2_0),G\right]\ge\left|E\left[\psi\left(\sigma^2_k\right)\right]-E\left[\psi\left(\sigma^2\right)\given\bm Y\right]\vphantom{\frac{}{}}\right|
&=\left|\sigma^2_0-\frac C{n-2}\right|\left(\frac p{n+p-2}\right)^k
\eas
for every $k\ge0$ since $E[\psi(\sigma^2)\mid\bm Y]=0$, which establishes the lower bound.

To establish the upper bound, let $\xi$ be a random variable such that $\xi\sim G$.
Then
\bas
d_W\left[G_0\left(\sigma^2_0\right),G\right]=E\left(\left|\xi-\sigma^2_0\right|\right)
&\le\left|\sigma^2_0-\frac C{n-2}\right|+E\left(\left|\xi-\frac C{n-2}\right|\right)\\
&\le\left|\sigma^2_0-\frac C{n-2}\right|+\sqrt{\var(\xi)}
=M_2\left(\sigma^2_0\right),
\eas
noting that $E(\xi)=C/(n-2)$.
Hence
the upper bound is established for $k=0$.  Now assume
as an inductive hypothesis
that the upper bound holds for some arbitrary $k\ge0$.  Then there exists a random variable~$\xi_k$ such that $\xi_k\sim G_k(\sigma^2_0)$ and
\bas
E\left(\left|\xi_k-\xi\right|\right)=d_W\left[G_k\left(\sigma^2_0\right),G\right]
\le M_2\left(\sigma^2_0\right)\,\left(\frac p{n+p-2}\right)^k,
\eas
noting that the existence of a coupling that attains the Wasserstein distance is well known \citep[e.g.,][]{rachev1984,givens1984}.
Now let $U\sim\chi^2_p$ and $V\sim\chi^2_{n+p}$ be independent of $\xi$, $\xi_k$, and each other, and define random variables $\xi_{k+1}$ and $\xi_\star$ according to
\bas
\xi_{k+1}=\frac1{V}\left(U\xi_k+C\right),\qquad
\xi_\star=\frac1{V}\left(U\xi+C\right),
\eas
noting that $\xi_{k+1}\sim G_{k+1}(\sigma^2_0)$ and $\xi_\star\sim G$ by construction.
Then
\bas
d_W\left[G_{k+1}\left(\sigma^2_0\right),G\right]\le
E\left(\left|\xi_{k+1}-\xi_\star\right|\right)=E\left(\frac UV\left|\xi_k-\xi\right|\right)
&=\frac p{n+p-2}\,d_W\left[G_k\left(\sigma^2_0\right),G\right]\\
&\le M_2\left(\sigma^2_0\right)\,\left(\frac p{n+p-2}\right)^{k+1},
\eas
establishing the upper bound for every $k\ge0$ by induction.
\end{proof}

\begin{rem}
The method of proof of the upper bound in Theorem~\ref{thm:wasserstein-rr} relies upon the establishment of a \emph{coupling} between the distributions of the iterates of the Markov chain and the stationary distribution.  (Note that the word \emph{coupling} here refers to a joint distribution that yields some specified distributions as its marginals, which differs from its usual meaning in the context of Markov chain analysis.)
To our knowledge, this approach has not been
previously
used to construct quantitative results for Markov chain convergence.
\end{rem}

\begin{rem}
Note that the expression for the geometric rate constant~$r$ is sharp since the upper and lower bounds both lead to $r=p/(n+p-2)$.
\end{rem}

Theorem~\ref{thm:wasserstein-rr} applies for any particular fixed values of $n$ and $p$.  
This sharp result allows us to
analyze the
standard Bayesian regression Gibbs sampler in~(\ref{gibbs-rr})
as the values of $n$ and~$p\equiv p_n$ grow.
Thus, we can understand the convergence of the chain in various $n$ and~$p$ regimes.
To do so, suppose we have a sequence of $n\times p_n$ covariate matrices~$\bm X_n$ and a sequence of $n\times1$ response vectors~$\bm Y_n$.
For the sake of complete rigor, we also impose the following
very mild
assumption
for the remainder of the results in this subsection.

\begin{assum}\label{assum:rr}
$\|\bm Y_n\|_2^2=O(n)$.
\end{assum}

Now
also note that we will write $\bm A_n$, $C_n$, $M_{1,n}(\sigma^2_0)$, and $M_{2,n}(\sigma^2_0)$ to denote the dependence of these quantities on~$n$.
Finally, for every $n\ge5$ and $k\ge0$, let $G_{n,k}(\sigma^2_0)$ denote the distribution of~$\sigma^2_k$ for the marginal chain in~(\ref{marginal-rr}) started with initial value~$\sigma^2_0$, and let $G_n$ denote the stationary distribution of this chain.
The following result now
allows us to understand the convergence behavior of the Markov chain corresponding to the Bayesian analysis of the classical regression model in~(\ref{model-rr}) in various $n$ and~$p$ regimes.

\begin{cor}\label{cor:rates-rr}
For the marginal chain in~(\ref{marginal-rr}) of the standard Bayesian regression Gibbs sampler
specified in~(\ref{gibbs-rr}),
\bas
m_1(\sigma^2_0)\,r_n^k\le d_W\left[G_{n,k}\left(\sigma^2_0\right),G_n\right]\le m_2(\sigma^2_0)\,r_n^k
\eas
for all $k$ and~$n$,
where
\bas
r_n=\frac{p_n}{n+p_n-2},\qquad
m_1(\sigma^2_0)=\inf_{n\ge5}M_{1,n}(\sigma^2_0)\ge0,\qquad
m_2(\sigma^2_0)=\sup_{n\ge5}M_{2,n}(\sigma^2_0)<\infty.
\eas
Moreover,
$r_n$ is bounded away from~$1$ if and only if $p_n=O(n)$.
\end{cor}

\begin{proof}
Note that
$0\le C_n=\bm Y_n^T(\bm I_n-\bm X_n\bm A_n^{-1}\bm X_n^T)\bm Y_n\le\|\bm Y_n\|_2^2=O(n)$
by Assumption~\ref{assum:rr}
and since the matrix $\bm X_n\bm A_n^{-1}\bm X_n^T$ is positive semidefinite.
Then clearly $m_2(\sigma^2_0)<\infty$, and
the rest follows immediately from Theorem~\ref{thm:wasserstein-rr}.
\end{proof}

The
message of Corollary~\ref{cor:rates-rr} is that although the convergence to the stationary distribution is geometric, the rate constant of the geometric convergence tends to~$1$ if the number of
parameters (i.e., the number of regression
coefficients, or equivalently,
the number of predictor variables)
grows faster than the sample size.
%
%
%
More practically, the results of Theorem~\ref{thm:wasserstein-rr} and Corollary~\ref{cor:rates-rr} may be understood by considering the number of iterations
required for approximate convergence.
Let $\epsilon>0$ be given, and let
$K_{n,\epsilon}\left(\sigma^2_0\right)$
denote the number of iterations
required for the Markov chain to be within~$\varepsilon$ of the stationary distribution in Wasserstein distance,
i.e.,
\bas
K_{n,\epsilon}\left(\sigma^2_0\right)=\inf\left\{K\ge0:d_W\left[G_{n,k}\left(\sigma^2_0\right),G_n\right]\le\epsilon\text{ for every }k\ge K\right\}.
\eas
The
following result asserts
that
if the dimension of the parameter grows faster than the sample size,
then the number of iterations required (to obtain approximate convergence to within some desired distance~$\epsilon>0$) grows without bound.

\begin{cor}\label{cor:rates-rr-alt}
For the marginal chain in~(\ref{marginal-rr}) of the standard Bayesian regression Gibbs sampler
specified in~(\ref{gibbs-rr}),
suppose that $0<\epsilon<m_1(\sigma^2_0)$.  Then
$K_{n,\epsilon}(\sigma^2_0)=O(1)$ as $n\to\infty$
if and only if $p_n=O(n)$.
\end{cor}

\begin{proof}
Assume that $p_n=O(n)$. Then there exists $r<1$ such that $r_n\le r$ for all~$n$.
Now let $K=\log[\epsilon/m_2(\sigma^2_0)]/\log(r)$, and note that for every $k\ge K$ and all~$n$,
\bas
d_W\left[G_{n,k}\left(\sigma^2_0\right),G_n\right]\le m_2(\sigma^2_0)\,r_n^k\le m_2(\sigma^2_0)\,r^k\le m_2(\sigma^2_0)\,r^K=\epsilon.
\eas
Thus, $K_{n,\epsilon}(\sigma^2_0)\le K$ for all~$n$, so $K_{n,\epsilon}(\sigma^2_0)=O(1)$.

Now assume instead that $K_{n,\epsilon}(\sigma^2_0)=O(1)$.  Then there exists an integer $K\ge0$ such that
\bas
m_1(\sigma^2_0)\,r_n^K\le d_W\left[G_{n,K}\left(\sigma^2_0\right),G_n\right]\le\epsilon\quad\text{ for all }n,
\eas
which implies that $r_n\le[\epsilon/m_1(\sigma^2_0)]^{1/K}<1$ for all~$n$.
It follows that
\bas
\frac{p_n}n\le\frac{p_n}{n-2}=\frac{r_n}{1-r_n}\le\frac{[\epsilon/m_1(\sigma^2_0)]^{1/K}}{1-[\epsilon/m_1(\sigma^2_0)]^{1/K}}
\eas
for all~$n$, and thus $p_n=O(n)$.
\end{proof}

To express
the idea of Corollary~\ref{cor:rates-rr-alt}
somewhat more finely, we can consider the rate at which
$K_{n,\epsilon}(\sigma^2_0)$ grows with $p_n$ and~$n$.
Note from Corollary~\ref{cor:rates-rr} and the definition of $K_{n,\epsilon}(\sigma^2_0)$ that
\bas
\left\lceil\frac{\log m_1\left(\sigma^2_0\right)-\log\epsilon}{\log(n+p_n-2)-\log p_n}\right\rceil
\le K_{n,\epsilon}\left(\sigma^2_0\right)\le
\left\lceil\frac{\log m_2\left(\sigma^2_0\right)-\log\epsilon}{\log(n+p_n-2)-\log p_n}\right\rceil,
\eas
where $\lceil\,\cdot\,\rceil$ denotes the ceiling function.
Thus,
$K_{n,\epsilon}(\sigma^2_0)$ is proportional to
\bas
\left[\log(n+p_n-2)-\log p_n\right]^{-1}=\left[\log\left(1+\frac{n-2}{p_n}\right)\right]^{-1}\approx\left(\frac{n}{p_n}\right)^{-1}=\frac{p_n}n
\eas
for large $n$ and~$p_n$
with $n<p_n$.
Thus, the rate of growth of $K_{n,\epsilon}(\sigma^2_0)$ is asymptotically linear in the ratio $p_n/n$.
In particular, an increase in the parameter dimension~$p$ increases the number of iteratons required for approximate convergence, while an increase in the sample size~$n$ reduces it.
As a concrete example,
a hundredfold increase in the parameter dimension implies a hundredfold increase in the required number of iterations (holding the sample size constant).
Hence, an especially large number of iterations may be required in the
modern ``large~$p$, small~$n$'' setting.
The above result thus
questions the validity of
high-dimensional Bayesian inference that relies on regression-type Gibbs
samplers.

\begin{proof}[%
Proof of Lemma~\ref{lem:post-corr-rr}%
]
The result is obtained by a straightforward calculation of the posterior correlation using
the conditional and marginal
posteriors
as well as standard properties of the $\chi^2$, $F$, and
inverse-gamma
distributions.
\end{proof}

\begin{proof}[%
Proof of Lemma~\ref{lem:autocorr-rr}%
]
Note that $\sigma^2_k\sim G$ and $\sigma^2_{k+1}\sim G$ since $\sigma^2_0\sim G$.
Then $\var(\sigma^2_k)=\var(\sigma^2_{k+1})$, and
\bas
\cov\left(\sigma^2_k,\sigma^2_{k+1}\right)
=\cov\left[\sigma^2_k,\,\frac1{V_{k+1}}\left(\sigma^2_kU_{k+1}+C\right)\right]
=E\left(\frac{U_{k+1}}{V_{k+1}}\right)
\var\left(\sigma^2_k\right)
=\frac{p}{n+p-2}\var\left(\sigma^2_k\right).
\eas
Similarly,
$\var(\|\bm\theta_k\|_2^2)=\var(\|\bm\theta_{k+1}\|_2^2)$, and
\bas
\cov\left(\|\bm\theta_k\|_2^2,\|\bm\theta_{k+1}\|_2^2\right)
=\cov\left[\|\bm\theta_k\|_2^2,\,\frac{U_{k+1}}{V_k}\left(\|\bm\theta_k\|_2^2+C\right)\right]
&=E\left(\frac{U_{k+1}}{V_k}\right)\,\var\left(\|\bm\theta_k\|_2^2\right)\\
&=\frac{p}{n+p-2}\var\left(\|\bm\theta_k\|_2^2\right).
\eas
The result then follows immediately.
\end{proof}

\begin{proof}[%
Proof of Lemma~\ref{lem:directions-iid}%
]
For all $k\ge1$, we have
$\bm\theta_k/\|\bm\theta_k\|_2=
\bm Z_k
/\|\bm Z_k\|_2$,
and the $\bm Z_k$ are independent.
\end{proof}

\begin{proof}[%
Proof of Theorem~\ref{thm:tv-dag}%
]
For each $j\in\{1,\ldots,m\}$,
Theorem~\ref{thm:tv-rr}
immediately establishes the existence of constants
$0\le \tilde M_{j,1}\le\tilde M_{j,2}$
such that
\bas
{
\tilde M_{j,1}
}
\left(\frac{\delta_j}{n+\delta_j-2}\right)^k\le d_{
\tv
}\left[\Xi_{j,k}\left(D_{jj;0}\right),\,\Xi_j\right]\le
{
\tilde M_{j,2}
}
\left(\frac{\delta_j}{n+\delta_j-2}\right)^k
\eas
for every $k\ge0$.
Now observe that $d_{
\tv
}
[\Xi_k(\bm D_0),\,\Xi]\le\sum_{j=1}^m d_{
\tv
}[\Xi_{j,k}(D_{jj;0}),\,\Xi_j]$
due to the mutual independence
of the $(D_{jj},\,\bm L_{
\pa(j),j
})$.
Then
\bas
d_{\tv}[\Xi_k(\bm D_0),\,\Xi]\le\sum_{j=1}^m\tilde M_{j,2}\left(\frac{\delta_j}{n+\delta_j-2}\right)^k\le m\left(\max_{1\le j\le m}\tilde M_{j,2}\right)\left(\frac{\delta_{\max}}{n+\delta_{\max}-2}\right)^k,
\eas
establishing the upper bound with $\tilde M_2=m\max_{1\le j\le m}\tilde M_{j,2}$.
Now note that there exists $J\in\{1,\ldots,m\}$ such that $\delta_J=\delta_{\max}$. Then
\bas
d_{\tv}[\Xi_k(\bm D_0),\,\Xi]\ge\sum_{j=1}^m\tilde M_{j,1}\left(\frac{\delta_j}{n+\delta_j-2}\right)^k\ge\tilde M_{J,1}\left(\frac{\delta_{\max}}{n+\delta_{\max}-2}\right)^k,
\eas
etablishing the lower bound with $\tilde M_1=\tilde M_{J,1}$.
\end{proof}

\section{%
Bayesian Model Selection
}
\label{sec:supp-lasso}

\begin{proof}[%
Proof of Theorem~\ref{thm:autocorr-lasso}%
]
Begin by noting that in the Gibbs sampler in~(\ref{gibbs-lasso}), $\sigma^2_k$ may be expressed as
\basn
\sigma^2_k
=\frac{\sigma^2_{k-1}U_k+C_{\bm\tau_k}}{V_k},&&\text{ where }
U_k
\sim\chi^2_p,
\label{recur-sigma}
\easn
and where $U_k$ is independent of~$\sigma^2_{k-1}$.
Then
\begin{align}
\cov(\sigma^2_k,\sigma^2_{k+1})
&=\cov\left(\sigma^2_k,\,\frac{\sigma^2_{k}U_{k+1}+C_{\bm\tau_{k+1}}}{V_{k+1}}\right)\notag\\
&=\cov\left(\sigma^2_k,\,\frac{U_{k+1}}{V_{k+1}}\sigma^2_{k}\right)+\cov\left(\sigma^2_k,\,\frac{C_{\bm\tau_{k+1}}}{V_{k+1}}\right).\label{cov-split}
\end{align}
Recall that $U_{k+1}$ and~$V_{k+1}$ are independent of each other and of~$\sigma^2_k$, so
\begin{align}
\cov\left(\sigma^2_k,\,\frac{U_{k+1}}{V_{k+1}}\sigma^2_{k}\right)
&=E\left[\frac{U_{k+1}}{V_{k+1}}\left(\sigma^2_k\right)^2\right]-E\left(\sigma^2_k\right)\,E\left(\frac{U_{k+1}}{V_{k+1}}\sigma^2_k\right)\notag\\
&=E\left(\frac{U_{k+1}}{V_{k+1}}\right)\left\{E\left[\left(\sigma^2_k\right)^2\right]-\left[E\left(\sigma^2_k\right)\right]^2\right\}
=\frac p{n+p-2}\var\left(\sigma^2_k\right).\label{cov-var}
\end{align}
Since $V_{k+1}$ is also independent of~$\bm\tau_{k+1}$, we have
\begin{align}
\cov\left(\sigma^2_k,\,\frac{C_{\bm\tau_{k+1}}}{V_{k+1}}\right)
&=E\left(\frac{C_{\bm\tau_{k+1}}}{V_{k+1}}\sigma^2_k\right)-E\left(\frac{C_{\bm\tau_{k+1}}}{V_{k+1}}\right)\,E\left(\sigma^2_k\right)\notag\\
&=\frac1{n+p-2}\left[E\left(C_{\bm\tau_{k+1}}\sigma^2_k\right)-E\left(C_{\bm\tau_{k+1}}\right)\,E\left(\sigma^2_k\right)\right]\notag\\
&=\frac1{n+p-2}\cov\left(\sigma^2_k,\,C_{\bm\tau_{k+1}}\right)
\ge-\frac1{n+p-2}\sqrt{\var\left(\sigma^2_k\right)\,\var\left(C_{\bm\tau_{k+1}}\right)}
\label{cov-cs-c}
\end{align}
by the
Cauchy--Schwarz
inequality. 
Now observe that
\bas
\var\left(C_{\bm\tau_{k+1}}\right)
\le E\left(C_{\bm\tau_{k+1}}^2\right)
=E\left\{\left[\bm Y^T\left(\bm I_n-\bm X\bm A_{\bm\tau_{k+1}}^{-1}\bm X^T\right)\bm Y\right]^2\right\}\le\left(\bm Y^T\bm Y\right)^2,
\eas
noting once again that $\bm Y$ is nonrandom from the point of view of the Gibbs sampling Markov chain.  Combining this result
with the inequality in~(\ref{cov-cs-c})
yields
\bas
\cov\left(\sigma^2_k,\,\frac{C_{\bm\tau_{k+1}}}{V_{k+1}}\right)
\ge
-\frac{\bm Y^T\bm Y}{n+p-2}\sqrt{\var\left(\sigma^2_k\right)},
\eas
which in turn may be combined with
the results in (\ref{cov-split}) and~(\ref{cov-var})
to obtain
\bas
\cov(\sigma^2_k,\sigma^2_{k+1})\ge\frac p{n+p-2}\var\left(\sigma^2_k\right)
\left[1-\frac{\bm Y^T\bm Y}{p\,\sqrt{\var\left(\sigma^2_k\right)}}\right].
\eas
The desired result then follows from the fact that
\bas
\var\left(\sigma^2_k\right)=\var\left(\sigma^2_{k+1}\right)=\var\left(\sigma^2\mid\bm Y\right)
\eas
since $\sigma^2_k\sim\pi(\sigma^2\mid\bm Y)$ and $\sigma^2_k\sim\pi(\sigma^2\mid\bm Y)$.
\end{proof}

\section{%
Multivariate Location Models
}
\label{sec:supp-mean}

\begin{proof}[%
Proof of Theorem~\ref{thm:tv-mean}%
]
Begin by writing the Gibbs sampler in~(\ref{gibbs-mean}) as
\bas
\bm\mu_k&=\tilde{\bm\mu}+\sqrt{\frac{\sigma^2_{k-1}}{n+\lambda}}\,\bm Z_k,&&\text{ where }\bm Z_k\sim N_p(\bm0_p,\bm I_p),\\
\sigma^2_k&=\frac{(n+\lambda)\left\|\bm\mu_k-\tilde{\bm\mu}\right\|_2^2+C}{V_k},&&\text{ where }V_k\sim\chi^2_{np+p},
\eas
and where all of the $\bm Z_k$ and $V_k$ are independent.
Substituting for~$\bm\mu_k$ yields
\basn
\sigma^2_k=\frac1{V_k}\left(\sigma^2_{k-1}U_k+C\right),&&\text{ where }U_k\sim\chi^2_p,\;V_k\sim\chi^2_{np+p},
\label{marginal-mean}
\easn
and where the $U_k$ and $V_k$ are all independent.
This marginal chain is the same as the marginal chain in~(\ref{marginal-rr}) of the standard Bayesian regression Gibbs sampler, except with the degrees of freedom of $V_k$ changed from $n+p$ to~$np+p$.
Thus, the
proof is essentially identical to that of Theorem~\ref{thm:tv-rr}.
\end{proof}

We now establish a sharp bound for the $d_W$-convergence rate of the marginal $\sigma^2_k$ chain in~(\ref{marginal-mean}) of the Gibbs sampler for the multivariate mean model.
For every $k\ge0$, let $G_k(\sigma^2_0)$ denote the distribution of~$\sigma^2_k$ for the marginal chain
in~(\ref{marginal-mean}).
Let $G$ denote the stationary distribution of this chain, i.e., the true marginal posterior of~$\sigma^2$.
Then we have the following result.

\begin{thm}\label{thm:wasserstein-mean}
For the marginal chain in~(\ref{marginal-mean}) of the
Gibbs sampler for the multivariate mean model,
\bas
M_1\left(\sigma^2_0\right)\,\left(\frac p{np+p-2}\right)^k
\le d_W\left[G_k(\sigma^2_0),G\right]\le
M_2\left(\sigma^2_0\right)\,\left(\frac p{np+p-2}\right)^k
\eas
for every $k\ge0$, where
\bas
M_1\left(\sigma^2_0\right)=\left|\sigma^2_0-\frac C{np-2}\right|,\qquad
M_2\left(\sigma^2_0\right)=\left|\sigma^2_0-\frac C{np-2}\right|+\frac C{np-2}\sqrt{\frac{2}{np-4}}.
\eas
\end{thm}

\begin{proof}
The proof is essentially identical to that of Theorem~\ref{thm:wasserstein-rr}.
\end{proof}

\section{%
Normal Hierarchical Model
}
\label{sec:supp-nh}

\begin{proof}[%
Proof of Theorem~\ref{thm:tv-nh}%
]
Begin by writing the Gibbs sampler in~(\ref{gibbs-nh}) as (for every $k\ge1$)
\bas
\bm\psi_{k,i}&=\frac{\tau^2}{\sigma^2+\tau^2}\bm X_i+\frac{\sigma^2}{\sigma^2+\tau^2}\bm\mu_{k-1}+\sqrt{\frac{\sigma^2\tau^2}{\sigma^2+\tau^2}}\;\bm Y_{k,i},
&&\text{ where }\bm Y_{k,i}\sim N_p(\bm0_p,\bm I_p),\\
\bm\mu_k&=\frac1n\sum_{i=1}^n\bm\psi_{k,i}+\sqrt{\frac{\tau^2}n}\;\bm Z_k,&&\text{ where }\bm Z_k\sim N_p(\bm0_p,\bm I_p),
\eas
where $i\in\{1,\ldots,n\}$, and where all of the $\bm Y_{k,i}$ and $\bm Z_k$ are independent.
Substituting for~$\bm\psi_{k,i}$ yields (for every $k\ge1$)
\basn
\bm\mu_k&=\frac{\tau^2}{\sigma^2+\tau^2}\left(\frac1n\sum_{i=1}^n\bm X_i\right)+\frac{\sigma^2}{\sigma^2+\tau^2}\bm\mu_{k-1}+\sqrt{\frac{\sigma^2\tau^2}{\sigma^2+\tau^2}}\left(\frac1n\sum_{i=1}^n\bm Y_{k,i}\right)+\sqrt{\frac{\tau^2}n}\;\bm Z_k\notag\\
&=\frac{\tau^2}{\sigma^2+\tau^2}\left(\frac1n\sum_{i=1}^n\bm X_i\right)+\frac{\sigma^2}{\sigma^2+\tau^2}\bm\mu_{k-1}+
\sqrt{\frac{(\tau^2)^2+2\sigma^2\tau^2}{n(\sigma^2+\tau^2)}}\;\bm W_k,
\easn
where $\bm W_k\sim N_p(\bm0_p,\bm I_p)$ and where the $\bm W_k$ are all independent.
Note
that the stationary distribution of this chain 
(i.e., the true marginal posterior of~$\bm\mu$)
is the $p$-dimensional multivariate normal distribution with mean vector $n^{-1}\sum_{i=1}^n\bm X_i$ and covariance matrix $n^{-1}(\sigma^2+\tau^2)\bm I_p$.
Now define
$\tilde{\bm\mu}=\bm\mu-n^{-1}\sum_{i=1}^n\bm X_i$ and $\tilde{\bm\mu}_k=\bm\mu_k-n^{-1}\sum_{i=1}^n\bm X_i$.
It is clear that the
total variation
distance between the distribution of~$\tilde{\bm\mu}_k$ and the marginal posterior of~$\tilde{\bm\mu}$ is the same as that between the distribution of~$\bm\mu_k$ and the marginal posterior of~$\bm\mu$.
Thus, it suffices
to prove
the result in the case where $n^{-1}\sum_{i=1}^n\bm X_i=\bm0_p$, which we will henceforth assume.
Then for every $k\ge1$,
\bas
\bm\mu_k=\frac{\sigma^2}{\sigma^2+\tau^2}\bm\mu_{k-1}+
\sqrt{\frac{(\tau^2)^2+2\sigma^2\tau^2}{n(\sigma^2+\tau^2)}}\;\bm W_k,
\eas
which implies that
\bas
\bm\mu_k
&=\left(\frac{\sigma^2}{\sigma^2+\tau^2}\right)^k\bm\mu_0
+\left[\frac{(\tau^2)^2+2\sigma^2\tau^2}{n(\sigma^2+\tau^2)}\sum_{i=1}^k\left(\frac{\sigma^2}{\sigma^2+\tau^2}\right)^{i-1}\right]^{1/2}\tilde{\bm W}_k\\
&=r^k\bm\mu_0
+\left[\frac{\sigma^2+\tau^2}{n}
\left(1-r^{2k}\right)\right]^{1/2}\tilde{\bm W}_k,
\eas
where $\tilde{\bm W}_k\sim N_p(\bm0_p,\bm I_p)$ and $r=\sigma^2/(\sigma^2+\tau^2)$.  Hence,
\basn
\bm\mu_k\sim N_p\left[
r^k\bm\mu_0,\;\frac{\sigma^2+\tau^2}{n}\left(1-r^{2k}\right)\bm I_p
\right]
\label{marginal-nh}
\easn
for every $k\ge0$.
Now note that $d_{\tv}[H_k(\bm\mu_0),H]$ is at least as large as the $d_{\tv}$-distance between the $N_p[r^k\bm\mu_0,\,n^{-1}(\sigma^2+\tau^2)\bm I_p]$ and $N_p[\bm0_p,\,n^{-1}(\sigma^2+\tau^2)\bm I_p]$
distributions.
Then by elementary properties of the total variation distance between multivariate normal distributions, we have
\bas
d_{\tv}\left[H_k(\bm\mu_0),H\right]
\ge\sqrt{\frac{n}{2(\sigma^2+\tau^2)}}\|\bm\mu_0\|_2\,
{
r^k
}
\eas
for all sufficiently
large~$k$.
To establish the upper bound,
let $\tilde H_k(\bm\mu_0)$ denote the
$p$-variate normal distribution with mean~$r^k\bm\mu_0$ and covariance matrix $n^{-1}(\sigma^2+\tau^2)\bm I_p$.
Then
\bas
d_{\tv}\left[H_k(\bm\mu_0),H\right]&\le
{
d_{\tv}\left[\tilde H_k(\bm\mu_0),H\right]+d_{\tv}\left[H_k(\bm\mu_0),\tilde H_k(\bm\mu_0)\right]
}
\\
&\le\left[\frac{n}{2(\sigma^2+\tau^2)}\right]^{1/2}\left\|\bm\mu_0\right\|_2\,r^k\\
&\qquad+{
p
}\left[
\frac2{\pi(1-r^{2k})}
\log\left(\frac1{1-r^{2k}}\right)
\left(1-r^{2k}\right)^{1/r^{2k}}
\right]^{1/2}\left(1-\sqrt{1-r^{2k}}\right)\\
&\le\sqrt{\frac n{2(\sigma^2+\tau^2)}}\left\|\bm\mu_0\right\|_2\,r^k+p\,r^{2k}
\le\sqrt{\frac n{\sigma^2+\tau^2}}\left\|\bm\mu_0\right\|_2\,r^k
\eas
for all sufficiently large~$k$.
\end{proof}

\begin{rem}
The chain $\bm\mu_k$ is linear in the previous iterate.  Thus, the sharp bound above can also be obtained by evaluating the maximal correlation.  (Recall that Theorem~\ref{thm:tv-rr} for obtaining sharp rates for the standard regression model was obtained using the maximal correlation method.)
\end{rem}

We now establish a sharp bound for the $d_W$-convergence rate of the marginal $\bm\mu_k$ chain in~(\ref{gibbs-nh}) of the Gibbs sampler for the normal hierarchical model.

\begin{thm}\label{thm:wasserstein-nh}
Consider the Gibbs sampler for the normal hierarchical model in~(\ref{gibbs-nh}).  Then
\bas
M_1(\bm\mu_0)\,
{
r^k
}
\le\frac1p\,d_W\left[H_k(\bm\mu_0),H\right]\le
M_2(\bm\mu_0)\,
{
r^k
}
\eas
for every $k\ge0$, where
$r=\sigma^2/(\sigma^2+\tau^2)$ and
\bas
M_1(\bm\mu_0)=\frac1p\left\|\bm\mu_0-\frac1n\sum_{i=1}^n\bm X_i\right\|_1,
\qquad
M_2(\bm\mu_0)=\frac1p\left\|\bm\mu_0-\frac1n\sum_{i=1}^n\bm X_i\right\|_1+\sqrt{\frac{2(\sigma^2+\tau^2)}{n\pi}}.
\eas
\end{thm}

\begin{proof}
Begin by noting that $H_k(\bm\mu_0)$ was derived in~(\ref{marginal-nh}) in the proof of Theorem~\ref{thm:tv-nh}.
The lower bound then follows immediately from a comparison of the means of the distributions $H_k(\bm\mu_0)$ and~$H$.

To establish the upper bound, let $\bm\xi$ be a random variable such that $\bm\xi\sim H$.
Then
\bas
d_W\left[H_0\left(\bm\mu_0\right),H\right]=E\left(\left\|\bm\xi-\bm\mu_0\right\|_1\right)
\le\left\|\bm\mu_0\right\|_1+E\left(\left\|\bm\xi\right\|_1\right)
=\left\|\bm\mu_0\right\|_1+p\,\sqrt{\frac{2(\sigma^2+\tau^2)}{n\pi}}
=p\,M_2(\bm\mu_0),
\eas
establishing the upper bound for $k=0$.  Now assume
as an inductive hypothesis
that the upper bound holds for some arbitrary $k\ge0$.  Then there exists a random variable $\bm\xi_k$ such that $\bm\xi_k\sim H_k(\bm\mu_0)$ and
\bas
E\left(\left\|\bm\xi_k-\bm\xi\right\|_1\right)=d_W\left[H_k(\bm\mu_0),H\right]
\le p\,M_2(\bm\mu_0)\,r^k,
\eas
noting that the existence of a coupling that attains the Wasserstein distance is well known \citep[e.g.,][]{rachev1984,givens1984}.
Now let $\bm W\sim N_p(\bm0_p,\bm I_p)$ be independent of $\bm\xi$ and $\bm\xi_k$, and define random variables $\bm\xi_{k+1}$ and $\bm\xi_\star$ according to
\bas
\bm\xi_{k+1}=
r\,\bm\xi_{k}+
\sqrt{\frac{(\tau^2)^2+2\sigma^2\tau^2}{n(\sigma^2+\tau^2)}}\;\bm W,
\qquad
\bm\xi_{\star}=
r\,\bm\xi+
\sqrt{\frac{(\tau^2)^2+2\sigma^2\tau^2}{n(\sigma^2+\tau^2)}}\;\bm W,
\eas
noting that $\bm\xi_{k+1}\sim H_{k+1}(\bm\mu_0)$ and $\bm\xi_\star\sim H$ by construction.  Then
\bas
d_W\left[H_{k+1}(\bm\mu_0),H\right]\le E\left(\left\|\bm\xi_{k+1}-\bm\xi_\star\right\|_1\right)=r\,E\left(\left\|\bm\xi_{k}-\bm\xi\right\|_1\right)\le p\,M_2(\bm\mu_0)\,r^{k+1},
\eas
establishing the upper bound for every $k\ge0$ by induction.
\end{proof}

\begin{rem}
Note that Theorem~\ref{thm:wasserstein-nh} is stated with the Wasserstein distance and the
$\ell_1$~norms
multiplied by a factor
of~$1/p$.
This factor is introduced
to adjust for the fact that the
$\ell_1$~norm
(on which the Wasserstein distance is based) in $p$ dimensions is a sum of $p$ terms.
\end{rem}

%

\section{%
Bounded Geometric Convergence Rates for High-Dimensional Regression
}
\label{sec:supp-ir}

\begin{proof}[%
Proof of Theorem~\ref{thm:rosenthal}%
]
We shall use
Theorem~12 of \citet{rosenthal1995}, which requires the establishment of a drift condition and an associated minorization condition.
We use a subscript~$R$ to correspond to the notation of \citet{rosenthal1995} when necessary to avoid a conflict with the notation in the remainder of the present work.
Also note that some expressions below require the existence of a $\sigma^2_0>0$ in addition to~$\bm\beta_0\in\reals^p$, but this $\sigma^2_0$ may be taken arbitrarily since
doing so
does not affect the chain in any way.

Let $V_R(\sigma^2,\bm\beta)=\|\bm Y-\bm X\bm\beta\|_2^2$,
and observe that
\bas
\left.\bm Y-\bm X\bm\beta_1\vphantom{0^0_0}\given\sigma^2_1\right.
&\sim N_n\left(\bm Y-\bm X\tilde{\bm\beta}_{\sigma^2_1},\;\sigma^2_1\bm X\bm A_{\sigma^2_1}^{-1}\bm X^T\right).
\eas
(It
is to be understood throughout the proof that all distributions and expectations are conditional on~$\bm Y$.)
Now let $\bm X=\bm U\bm\Omega\bm V^T$,
where $\bm U$ and $\bm V$ are orthogonal with columns
$\bm u_1,\ldots,\bm u_n$ and
$\bm v_1,\ldots,\bm v_{
p
}$ (respectively),
and where $\bm\Omega$ is $n\times p$ rectangular-diagonal with $\bm\Omega=\diag_{n\times p}(\omega_1,\ldots,\omega_n)$.
(Note that these matrices depend on~$p$, although we do not indicate this dependence explicitly in the notation.)
Then
\bas
\left.\bm Y-\bm X\bm\beta_1\vphantom{0^0_0}\given\sigma^2_1\right.
&\sim N_n\left[\bm U\left(\bm I_n-\bm\Psi_{\sigma^2_1}\right)\bm U^T\bm Y,\;
\sigma^2_1\bm U\bm\Psi_{\sigma^2_1}\bm U^T\right],
\eas
where
\bas
\bm\Psi_{\sigma^2_1}=\diag\left(\frac{\omega_1^2}{\omega_1^2+\lambda\sigma^2_1},\ldots,\frac{\omega_n^2}{\omega_n^2+\lambda\sigma^2_1}\right).
\eas
Then for all $\sigma^2_1>0$,
\bas
E\left[V_R\left(\sigma^2_1,\bm\beta_1\right)\given\sigma^2_1\right]
&=\left\|\bm U\left(\bm I_n-\bm\Psi_{\sigma^2_1}\right)\bm U^T\bm Y\right\|_2^2
+\tr\left(\sigma^2_1\bm U\bm\Psi_{\sigma^2_1}\bm U^T\right)\\
&=\left\|\left(\bm I_n-\bm\Psi_{\sigma^2_1}\right)\bm U^T\bm Y\right\|_2^2
+\sigma^2_1\tr\left(\bm\Psi_{\sigma^2_1}\right)\le\bm Y^T\bm Y+n\sigma^2_1,
\eas
and hence
\bas
E\left[V_R\left(\sigma^2_1,\bm\beta_1\right)\right]
&=E\left\{E\left[V_R\left(\sigma^2_1,\bm\beta_1\right)\given\sigma^2_1\right]\right\}\\
&\le\bm Y^T\bm Y+n\,E\left(\sigma^2_1\right)
=\bm Y^T\bm Y+\frac n{n+a-2}\left(\left\|\bm Y-\bm X\bm\beta_0\right\|_2^2+s\right).
\eas
Thus, the drift condition of \citet{rosenthal1995} holds with $V_R(\sigma^2,\bm\beta)$ as given above, with constants $\lambda_R=n/(n+a-2)$ and $b_R=\bm Y^T\bm Y+ns/(n+a-2)$.

We now establish an associated minorization condition.  Let $d_R>2b_R/(1-\lambda_R)$,
and suppose $V_R(\sigma^2_0,\bm\beta_0)=\|\bm Y-\bm X\bm\beta_0\|_2^2\le d_R$.
Let $f(\sigma^2,\bm\beta\mid\sigma^2_0,\bm\beta_0)$ denote the density with respect to Lebesgue measure of the joint distribution of the
$(k+1)$st
iterate given that the $k$th iterate takes the value $(\sigma^2_0,\bm\beta_0)$.
This density may be expressed as
\bas
f(\sigma^2,\bm\beta\mid\sigma^2_0,\bm\beta_0)=f_{\bm\beta\mid\sigma^2}(\bm\beta\mid\sigma^2)\,f_{\sigma^2\mid\bm\beta_0}(\sigma^2\mid\bm\beta_0).
\eas
Now let $Q_R$ be the $\inversegamma[(n+a)/2,\,(d_R+s)/2]$ distribution,
and let $q_R$ be its density with respect to Lebesgue measure.
Then
\bas
f_{\sigma^2\mid\bm\beta_0}(\sigma^2\mid\bm\beta_0)
&=\frac{\left[\left(\|\bm Y-\bm X\bm\beta_0\|_2^2+s\right)/2\right]^{(n+a)/2}}{\Gamma\left[\left(n+a\right)/2\right]}
\left(\sigma^2\right)^{-(n+a+2)/2}
\exp\left(-\frac{\|\bm Y-\bm X\bm\beta_0\|_2^2+s}{2\sigma^2}\right)\\
&\ge\left(\frac{\|\bm Y-\bm X\bm\beta_0\|_2^2+s}{d_R+s}\right)^{(n+a)/2}q_R(\sigma^2)
\ge\left(\frac s{d_R+s}\right)^{(n+a)/2}q_R(\sigma^2).
\eas
Thus,
the minorization condition of \citet{rosenthal1995} holds with $Q_R$ and $d_R$ as given above, with $\epsilon_R=[s/(d_R+s)]^{(n+a)/2}$.
The result then follows immediately from Theorem~12 of \citet{rosenthal1995}, noting also that the quantity
$\alpha^{-(1-r)}$ of \citet{rosenthal1995}
may simply be omitted while still preserving an upper bound.
\end{proof}

\begin{proof}[%
Proof of Corollary~\ref{cor:rates-ir}%
]
For any $0<\alpha<1$, let $r_1(\alpha)=(1-\epsilon_R)^\alpha$, and let $r_2(\alpha)$ equal the quantity in square brackets on the right-hand side of Theorem~\ref{thm:rosenthal}.
Note that $(1+2b_R+\lambda_Rd_R)/(1+d_R)<1$ since $d_R>2b_R/(1-\lambda_R)$,
so it follows that there exists $0<\alpha_\star<1$ such that $r_2(\alpha_\star)<1$.
Then the result clearly holds with $r=\max\{r_1(\alpha_\star),r_2(\alpha_\star)\}$, noting that this quantity does not depend on~$p$, $\bm X_p$, or~$\bm\beta_{0,p}$.
\end{proof}

\begin{proof}[%
Proof of Theorem~\ref{thm:tv-dd}%
]
Begin by writing the Gibbs sampler in~(\ref{gibbs-dd}) as (for every $k\ge1$)
\bas
\bm\beta_k&=\tilde{\bm\beta}+\sqrt{\sigma^2_{k-1}}\;\bm A^{-1/2}\bm Z_k,
&&\text{where }\bm Z_k\sim N_p(\bm0_p,\bm I_p),\\
\sigma^2_k&=\frac1{V_k}\left[\left(\bm\beta_{k}-\tilde{\bm\beta}\right)^T\bm A\left(\bm\beta_{k}-\tilde{\bm\beta}\right)+C+s\right],
&&\text{where }V_k\sim\chi^2_{n+p+\lceil p\epsilon\rceil},
\eas
and where all of the $\bm Z_k$ and $V_k$ are independent.
Substituting for $\bm\beta_k$ yields
\basn
\sigma^2_k=\frac1{V_k}\left(\sigma^2_{k-1}U_k+C+s\right)
{
,
}
&&\text{ where }U_k\sim\chi^2_p,\;V_k\sim\chi^2_{n+p+\lceil p\epsilon\rceil},
\easn
and where the $U_k$ and $V_k$ are all independent.
This marginal chain is the same as the marginal chain in~(\ref{marginal-rr}) of the standard Bayesian regression Gibbs sampler, except with the degrees of freedom of $V_k$ changed from $n+p$ to $n+p+\lceil p\epsilon\rceil$.
Thus, 
the proof is essentially identical to that of Theorem~\ref{thm:tv-rr}.
\end{proof}

\section{%
Summary of Convergence Results
}
\label{sec:supp-table}

\begin{center}
\begin{tabular}{lc@{}cc@{}c}
\toprule
\multicolumn{1}{c}{Family}&Model&Gibbs Steps&Dimension&Convergence Rate\\\midrule[\heavyrulewidth]
&Standard&2&$p$, $1$&$=p/(n+p-2)$\\\cmidrule{2-5}
\multirow{2}{*}[-3.5pt]{Regression}&Independent-Prior&2&$p$, $1$&does not depend on~$p$\\\cmidrule{2-5}
&
Dimensionally-Dependent
&
2
&
$p$, $1$
&
$=p/(n+p+\lceil p\epsilon\rceil-2)$
\\\cmidrule{2-5}
&Lasso-Type&3&$p$, $1$, $p$&$\approx p/(n+p-2)$\\\midrule[\heavyrulewidth]
Location&Location&2&$np$, $1$&$\approx p/(np+p-2)$\\\midrule[\heavyrulewidth]
&Known Variances&2&$np$, $p$&$=\sigma^2/(\sigma^2+\tau^2)\nrightarrow1$\\\cmidrule{2-5}
Hierarchical
&Unknown Variances&4&$np$, $p$, $1$, $1$&$\to1$ as $n,p\to\infty$\\\cmidrule{2-5}
&
Unknown Variances,
&
\multirow{2}{*}{4}
&
\multirow{2}{*}{$np$, $p$, $1$, $1$}
&
\multirow{2}{*}{$\nrightarrow1$ as $n,p\to\infty$}
\\
&
Dimensionally-Dependent
&&&\\\bottomrule
\end{tabular}
\end{center}

\section{%
Details of Numerical Results
}
\label{sec:supp-numerical-details}

Each point in the plots of
Figures~\ref{fig:autocorr}~and~\ref{fig:autocorr-nhuv}
represents the average lag-one autocorrelation over
10 Gibbs sampling runs of 10,000 iterations each.
The quantities, vectors, and matrices used in each model are described separately below.

For the regression-type models in Figure~\ref{fig:autocorr}, chains were executed for each combination of values of $n\in\{10,30,100\}$ and $p\in\{10,30,100\}$.  For each of the 10~runs at each $n$ and $p$ setting, the $np$ elements of the $n\times p$ covariate matrix~$\bm X$ were drawn as independent $N(0,1)$ random variables.  Also, for each run, the $n\times1$ response vector~$\bm Y$ was generated as
$\bm Y=\bm X\bm\beta_\star+\bm\epsilon$,
where $\bm\beta_\star$ is a $p\times1$ vector with its first $p/2$ elements drawn independently as $\pm1$ with probability $1/2$ each and its remaining $p/2$ elements set to zero, and where $\bm\epsilon$ is an $n\times1$ vector of independent $t_4$ random variables multiplied by $1/2$.
The initial values were set as $\bm\beta_0=\bm1_p$ and $\sigma^2_0=1$.
For the Bayesian lasso, the regularization parameter~$\lambda$ was set to $\lambda=1$.  
For the elastic net, both regularization parameters~$\lambda_1$ and $\lambda_2$ were set to $\lambda_1=\lambda_2=1$.
The spike-and-slab prior used
$\zeta_j=1/n$
and $\kappa_j=10$ for all $j\in\{1,\ldots,p\}$.

For the hierarchical models in
the left side of
Figure~\ref{fig:autocorr-nhuv}, chains were executed for each combination of values of $n\in\{10,30,60,100,150,210\}$ and $p\in\{3,10,30,100,300\}$.
For the hierarchical models in
the center and right side of
Figure~\ref{fig:autocorr-nhuv}, chains were executed for each combination of values of $n\in\{5,15,25,35,45\}$ and $p\in\{5,15,25,35,45\}$.
For each of the
10~runs
at each $n$ and $p$ setting, the $np$ elements of the matrix~$\bm X$ were drawn as independent $t_4$ random variables multiplied by~$1/2$.
The initial values were set as $\bm\mu_0=\bm1_p$ and $\sigma^2_0=\tau^2_0=1$.

\end{document}